\documentclass[11pt,a4paper]{article}
\usepackage[utf8]{inputenc}
\usepackage{amsmath}
\usepackage{amsfonts}
\usepackage{amssymb}
\usepackage{amsthm}
\usepackage{setspace}
\usepackage{hyperref}
\usepackage{graphicx}
\usepackage{stmaryrd}
\usepackage{mathtools}
\usepackage{subcaption}
\usepackage[T1]{fontenc}
\usepackage{xcolor}
\usepackage{float}
\usepackage{tikz-cd}
\usepackage[noadjust]{cite}
\usepackage{cite}
\usepackage{ragged2e}
\usepackage[margin=1.2in]{geometry}
\usepackage{fancyhdr}
\graphicspath{ {./images/} }
\newtheorem{theorem}{Theorem}[section]
\newtheorem{proposition}[theorem]{Proposition}
\newtheorem{lemma}[theorem]{Lemma}
\newtheorem{corollary}[theorem]{Corollary}
\newtheorem*{corollary*}{Corollary}

\newtheorem*{proposition4.2}{Proposition 4.2}

\theoremstyle{definition}
\newtheorem{definition}[theorem]{Definition}
\newtheorem{example}[theorem]{Example}
\newtheorem{remark}[theorem]{Remark}

\newtheorem{question}[theorem]{Question}
\newtheorem{notation}[theorem]{Notation}

\newcommand{\Aut}{\mathrm{Aut}}
\newcommand{\Out}{\mathrm{Out}}
\newcommand{\Inn}{\mathrm{Inn}}
\newcommand{\Fix}{\mathrm{Fix}}
\newcommand{\Stab}{\mathrm{Stab}}
\newcommand{\gars}{\Delta_{ab}}
\newcommand{\Isom}{\mathrm{Isom}}

\newcommand{\h}{z}

\newcommand{\quotient}[2]{{\raisebox{.2em}{$#1$}\left/\raisebox{-.2em}{$#2$}\right.}}

\usepackage{tikz}

\begin{document}
\usetikzlibrary{arrows,positioning}

\begin{center}
{\large \textbf{Fixed subgroups in Artin groups}}
\end{center}

\begin{center}
Oli Jones and Nicolas Vaskou 
\end{center}

\begin{abstract}
\centering \justifying 

We study fixed subgroups of automorphisms of any large-type Artin group $A_{\Gamma}$. We define a natural subgroup $\Aut_\Gamma(A_\Gamma)$ of $\Aut(A_{\Gamma})$, and for every $\gamma \in \Aut_\Gamma(A_\Gamma)$ we find the isomorphism type of $\Fix(\gamma)$ and a generating set for a finite index subgroup. We show that $\Fix(\gamma)$ is a finitely generated Artin group, with a uniform bound on the rank in terms of the number of vertices of $\Gamma$. Finally, we provide a natural geometric characterisation of the subgroup $\Aut_\Gamma(A_\Gamma)$, which informally is the maximal subgroup of $\Aut(A_\Gamma)$ leaving the Deligne complex of $A_{\Gamma}$ invariant.
\end{abstract}

\noindent \rule{7em}{.4pt}\par

\small



\normalsize
\tableofcontents

\noindent \rule{7em}{.4pt}\par

Email: oj2003@hw.ac.uk. Address: School of Mathematical and Computer Sciences, Heriot-Watt University, Edinburgh, Scotland, EH14 4AS

Email: nicolas.vaskou@bristol.ac.uk. Address: School of Mathematics, University of Bristol, Bristol BS8 1UG.

\newpage

\section{Introduction}

Given a group $G$ and an automorphism $\gamma \in \Aut(G)$, one can define the \emph{fixed subgroup} of $\gamma$ as
$$\Fix(\gamma) = \{g \in G ~|~ \gamma(g) = g\}.$$
Fixed subgroups have been extensively studied, with good structural results being common in groups with some kind of negative curvature. For instance, Bestvina-Handel's celebrated theorem states that in the rank $n$ non-abelian free group $F_n$, fixed subgroups have rank at most $n$ \cite{bestvina1992train}. Other results in this spirit have been obtained in hyperbolic \cite{neumann1992fixed} and relatively hyperbolic groups \cite{minasyan2012fixed}. Recently, there has been interest in fixed subgroups in settings without negative curvature, such as right-angled Artin groups \cite{fioravanti2024coarse}.

In this paper, we study fixed subgroups in Artin groups. These groups are easily defined by a presentation. Let $\Gamma$ be a simplicial graph with vertex set $V(\Gamma)$ and edge set $E(\Gamma)$, and suppose that each edge $e^{ab}$ between two vertices $a$ and $b$ is given a coefficient $m_{ab} \in \{2, 3, 4, \cdots \}$. Then the \emph{Artin group} associated with $\Gamma$ is the group given by the following presentation:
$$A_{\Gamma} \coloneqq \langle V(\Gamma) \ | \ \underbrace{aba\cdots}_{m_{ab} \text{ terms}} = \underbrace{bab\cdots}_{m_{ab} \text{ terms}}, \forall e^{ab} \in E(\Gamma) \rangle.$$
The vertices of $\Gamma$ are called the \emph{standard generators} of $A_{\Gamma}$ associated with $\Gamma$, and the number of vertices of $\Gamma$ is called the \emph{rank} of $A_{\Gamma}$. It is conjectured that the rank of an Artin group in this sense coincides with the usual notion of rank. From here, we will always use rank in the Artin group sense. If $a$ and $b$ are not connected by an edge, we set $m_{ab} \coloneqq \infty$.

Artin groups are generally hard to study, as questions such as knowing their torsion, their centres, or a solution to the word problem, remain open in the most general setting. For this reason, most of the results regarding Artin groups concern some smaller class of Artin groups. In that spirit, we study Artin groups of \emph{large-type}, that is, where for all pairs of generators $a, b \in V(\Gamma)$, we have $m_{ab} \geq 3$.

Very little is known regarding automorphisms of Artin groups in general. Although different types of possible automorphisms have been identified, computing the full automorphism groups of Artin groups remains wide open. The automorphism group of Artin groups were computed for some smaller classes, namely: the \emph{right-angled} Artin groups, i.e. Artin groups where every coefficient is either $2$ or $\infty$ \cite{servatius1989automorphisms, laurence1995generating}, and some affine and spherical-type Artin groups \cite{charney2004automorphism}. The general “large-type” case remains open, although some progress was made: a generating set was given for the automorphism groups of connected large-type triangle free Artin groups \cite{crisp2005automorphisms}, and more recently a presentation was given for large-type \emph{free-of-infinity} Artin groups, i.e. Artin groups where every coefficient $m_{ab}$ satisfies $3 \leq m_{ab} < \infty$ \cite{vaskou2023automorphisms}.

In this paper, we describe fixed subgroups of automorphisms beyond the aforementioned classes of Artin groups for which $\Aut(A_{\Gamma})$ has been fully computed. To do so, we restrict our attention to the following subgroup of $\Aut(A_{\Gamma})$:

\begin{definition}
The group $\Aut_{\Gamma}(A_{\Gamma})$ is the subgroup of $\Aut(A_{\Gamma})$ generated by the following three kinds of automorphisms:
\begin{itemize}
        \item The subgroup $\Inn(A_{\Gamma})$ of \emph{inner automorphisms}, that is, automorphisms of the form $\varphi_g \colon h \mapsto g h g^{-1}$ for some $g \in A_{\Gamma}$.
        \item The subgroup $\Aut(\Gamma)$ of \emph{graph automorphisms}, that is, automorphisms induced by a labelled-preserving permutation $\sigma \in \Aut(\Gamma)$ of the standard generators.
        \item The order $2$ subgroup generated by the \emph{global inversion}, that is, the automorphism defined by $\iota \colon a \mapsto a^{-1}$ for every standard generator $a \in V(\Gamma)$.
    \end{itemize}
\end{definition}

In some cases it is that $\Aut_\Gamma(A_\Gamma)$ is the whole of $\Aut(A_\Gamma)$. This was for instance proved by the second author for large-type free-of-infinity Artin groups (see \cite{vaskou2023automorphisms}). Moreover, the subgroup $\Aut_{\Gamma}(A_{\Gamma})$ turns out to have a strong geometric interpretation, as we will discuss later. Note that $\Aut_\Gamma(A_\Gamma)$ in general depends on the choice of presentation graph $\Gamma$.

Fixed subgroups of Artin groups were previously studied by Crisp \cite{crisp2000symmetrical} who showed that for certain subclasses of Artin groups, including large-type, the fixed subgroups of graph automorphisms were finitely generated Artin groups. Our result is in the spirit of Crisp's, dealing with a much larger class of automorphisms, restricted to large-type Artin groups. In particular, we recover Crisp's result in the large-type case.

We briefly introduce the notion of parabolic subgroups, which will be used in the statement of our main result. It is a standard result (\cite{van1983homotopy}) that for any Artin group $A_{\Gamma}$ and any induced subgraph $\Gamma' \subseteq \Gamma$, the Artin group $A_{\Gamma'}$ naturally injects in $A_{\Gamma}$. In particular, one can unambiguously write $A_{\Gamma'}$ to denote the subgroup of $A_{\Gamma}$ spanned by the vertices of $\Gamma'$. These subgroups are called \emph{standard parabolic subgroups}, and their conjugates are called \emph{parabolic subgroups}. Parabolic subgroups lie at the heart of the study of Artin groups, and our paper is no exception.

Our main theorem may be shortly stated as follows.

\begin{theorem}\label{mainTheorem}
    Let $A_\Gamma$ be a large-type Artin group, and let $\gamma \in \Aut_\Gamma(A_\Gamma)$. 
    
    \begin{enumerate}   

    \item We determine the isomorphism type of $\Fix(\gamma)$, and provide a generating set for a finite index subgroup of $\Fix(\gamma)$.
    
    \item The subgroup $\Fix(\gamma)$ is isomorphic to one of the following:
    $$\{1\}, \ \ \mathbb{Z}, \ \ \mathbb{Z}^2, \ \ F, \ \ \mathbb{Z} \times F, \ \ A_4, \ \ A_{\Gamma'}, \ \ A_{\Gamma'} * \mathbb{Z}, \text{ or } A_{\Gamma'} * F,$$
    where $F$ is a non-abelian finitely generated free group, $A_{\Gamma'}$ is a parabolic subgroup of $A_\Gamma$, and $A_4$ is  the group with abstract presentation $\langle x, y \ | \ xyxy = yxyx \rangle$. 
    
    \item We obtain a uniform bound on the rank of the fixed subgroups, namely $$\mathrm{Rank}(\Fix(\gamma)) \leq n^2 - 2n +2,$$ where $n = \mathrm{Rank}(A_\Gamma)$.

    \item In particular, $\Fix(\gamma)$ is always a finitely generated Artin group, of one of finitely many isomorphism types. 
    \end{enumerate}
\end{theorem}

For a precise description of the isomorphism type of $\Fix(\gamma)$ along with a generating set of a finite index subgroup, we refer the reader to Theorem \ref{mainTheoremElliptic} and Theorem \ref{mainTheoremHyperbolic}, for the case where $\Gamma$ has at least 3 vertices, and Theorem \ref{dihedralFixPointsOdd} (from \cite{jones2023fixed}, included for completeness) and Theorem \ref{dihedralFixPointsEven}, for the dihedral case. The comment on rank follows from Proposition \ref{propRankElliptic} and Remark \ref{remHyperbolicRank}, with the proposition showing the bound is optimal.
\medskip

Given the results of \cite{vaskou2023automorphisms}, we obtain the following corollary.

\begin{corollary}
    If $A_{\Gamma}$ is large-type and free-of-infinity, then Theorem \ref{mainTheorem} describes the fixed subgroups of all automorphisms of $A_{\Gamma}$.
\end{corollary}

In general, Theorem \ref{mainTheorem} does not hold if we don't restrict to the subgroup $\Aut_{\Gamma}(A_{\Gamma})$:

\begin{example}
    Consider the dihedral Artin group $\langle a, b ~|~ abab = baba \rangle$, and the automorphism $\gamma\colon a \mapsto (ab)^{-2}a$, $b \mapsto b(ab)^2$. Then $\Fix(\gamma)$ is not an Artin group and not finitely generated (see Example \ref{infinitelyGenerated}).
\end{example}

The automorphism in this example does not extend to a (freely indecomposable) large-type Artin group on at least 3 generators. It would be interesting to know if our main theorem holds for irreducible Artin groups of rank at least 3, without restricting to $\Aut_\Gamma(A_\Gamma)$.

\begin{question}
    Given $A_\Gamma$ a freely indecomposable large-type Artin group of rank at least 3, what can be said about its fixed subgroups, without restriction on the automorphisms? For instance, are they finitely generated? Is there a uniform bound on the rank in terms of the rank of $A_\Gamma$? Are the fixed subgroups Artin groups?
\end{question}

We now discuss the strategy for proving Theorem \ref{mainTheorem}. The \emph{Deligne complex} $X_{\Gamma}$ associated to an Artin group is a simplicial complex built from the combinatorics of its parabolic subgroups (see \cite{charney1995k} or Definition \ref{DefDeligneComplex}). It has turned out to be an incredible tool to study Artin groups, especially for classes such as large-type Artin groups, where its geometry is better understood and carries a natural $\mathrm{CAT}(0)$ metric (see Theorem \ref{ThmCAT(0)} and Definition \ref{DefiMoussong}).

In \cite{vaskou2023automorphisms}, the author describes how one can extend the action of $A_{\Gamma}$ on $X_{\Gamma}$ to a \emph{compatible} action of the subgroup $\Aut_{\Gamma}(A_{\Gamma}) \leq \Aut(A_\Gamma)$ on $X_{\Gamma}$, (see Definition \ref{DefinitionActionOfAut}). Compatible actions are defined in Definition \ref{defCompatible}, but essentially mean that the inner automorphism $\varphi_g\colon h \mapsto ghg^{-1}$ acts like $g$ on $X_\Gamma$.

This action is semi-simple, so every automorphism $\gamma \in \Aut_\Gamma(A_\Gamma)$ has a non-empty, convex and hence $\mathrm{CAT}(0)$ minimal set $X_\Gamma^\gamma$ (see Definition \ref{defMinSet}). Given an automorphism $\gamma$, we show that the action $\Fix(\gamma) \curvearrowright X_\Gamma$ restricts to an action $\Fix(\gamma) \curvearrowright X_\Gamma^\gamma$ (see Corollary \ref{fixAction}). Our strategy then is to compute $X_\Gamma^\gamma$ for each $\gamma \in \Aut_\Gamma(A_\Gamma)$, and then calculate $\Fix(\gamma)$ by restricting our attention to the setwise stabiliser of this subspace, which is typically much smaller than the whole Deligne complex. Along the way, we understand the minset in $X_\Gamma$ of every automorphism of $A_\Gamma$ in $\Aut_\Gamma(A_\Gamma)$.

Given our strategy, the compatible action $\Aut_\Gamma(A_\Gamma) \curvearrowright X_\Gamma$ is crucial. The following theorem is a natural geometric characterisation of $\Aut_\Gamma(A_\Gamma)$, showing that it carries some maximality property for compatibility, which is yet another justification for restricting our attention to this subgroup.

\begin{theorem}\label{theoremIntroEquivariantRigidity}
    Suppose $A_\Gamma$ is a large-type Artin group with $\Gamma$ connected, and $A \leq \Aut(A_\Gamma)$ has a compatible action $A \curvearrowright X_\Gamma$. Then $A \leq \Aut_\Gamma(A_\Gamma)$.
\end{theorem}

This theorem shows that $\Aut_\Gamma(A_\Gamma)$ is the maximal subgroup of $\Aut(A_\Gamma)$ where one can use compatible actions to recover fixed subgroups. Furthermore, it follows that $\Aut_\Gamma(A_\Gamma)$ is the maximal subgroup leaving the Deligne complex $X_\Gamma$ invariant (up to equivariant isometry). This natural notion regularly appears when studying automorphism groups of groups acting on trees, for instance \cite{bass1996automorphism}, \cite{levitt2007automorphism} \cite{andrew2022free}. In the setting of groups acting on trees, the analogous object to $\Aut_\Gamma(A_\Gamma)$ can equivalently be characterised as the subgroup preserving the translation function, or a point stabiliser in the associated deformation space.

As is often the case when working with the Deligne complex, some of our arguments reduce to understanding what happens locally. In particular, one first needs to understand fixed subgroups for automorphisms of local groups such as \emph{dihedral Artin groups}, i.e. Artin groups whose defining graph is a single edge. In Section \ref{sectionDihedral}, we prove Theorem \ref{mainTheorem} for large-type dihedral Artin groups. The first author has already proved it in the case where the coefficient of the dihedral Artin group is odd \cite{jones2023fixed}, so we focus on the even case.

Let $A_m$ describe the dihedral Artin group with coefficient $m$. It is well-known (see for instance \cite{ciobanu2020equations}) that $A_m$ has a graph-of-$\mathbb{Z}$ decomposition and thus acts on the associated Bass-Serre tree $T_m$. In fact, this action extends to a compatible action of $\Aut(A_m)$. While in the general case we are interested in computing $X_{\Gamma}^{\gamma}$ for $\gamma \in \Aut_\Gamma(A_\Gamma)$, in the dihedral Artin group case we are interested in computing $T_m^{\gamma}$, the tree $T_m$ playing an analog role to that of the Deligne complex $X_{\Gamma}$. While the technical details are different, the general strategies are very much alike.

It would be interesting to know if our results extend to other classes of Artin groups, in particular those where some of our tools still exist. The FC-Type Artin groups have $\mathrm{CAT}(0)$ Deligne complexes \cite{charney1995k} (albeit with a different metric), which admit a compatible action of $\Aut_\Gamma(A_\Gamma)$. A priori our methods could be applied in this setting, albeit with the complications of higher dimension, and larger spherical parabolic subgroups.

\begin{question}
    Given $A_\Gamma$ a spherical or a FC-type Artin group, and given $\gamma \in \Aut_\Gamma(A_\Gamma)$, is $\Fix(\gamma)$ a finitely generated Artin group? Is there a uniform bound on the rank of these subgroups, in terms of the rank of $A_\Gamma$?
\end{question}

The outline of the paper is as follows. In Section \ref{sectionPreliminaries} we define the Deligne complex, recall the structure of minimal sets, define compatible actions, and prove several simple lemmas which will form the basis of our strategy. In Section \ref{sectionDihedral} we prove Theorem \ref{mainTheorem} for even dihedral Artin groups. In Section \ref{sectionGeneral}, we prove Theorem \ref{mainTheorem} when $A_{\Gamma}$ has rank at least 3, following the strategy outlined. Finally in Section \ref{sectionEquivariantRigidity}, we prove Theorem \ref{theoremIntroEquivariantRigidity}. 

Sections \ref{sectionDihedral}, \ref{sectionGeneral} and \ref{sectionEquivariantRigidity} can be read independently of each other, although Section \ref{sectionDihedral} is a good warm up for Section \ref{sectionGeneral}, due to the similarity of the methods.

\section*{Acknowledgements}

The first author would like to thank their supervisor Laura Ciobanu for her support and comments on the paper. The second author is supported by the Postdoc Mobility $\sharp$P500PT$\_$210985 of the Swiss National Science Foundation. We also thank Gemma Crowe and Giovanni Sartori for their comments on the manuscript.

 \section{Preliminaries} \label{sectionPreliminaries}

This section serves as a preliminary section in which we recall standard objects and results. Section 2.1 is dedicated to introducing the Deligne complex $X_\Gamma$, and the action of $A_{\Gamma}$ and of the subgroup $\Aut_{\Gamma}(A_\Gamma) \leq \Aut(A_{\Gamma})$ on $X_{\Gamma}$. In Section 2.2 we discuss the notion of minimal sets for elements acting on the Deligne complex. This includes a few technical results. Finally, in Section 2.3 we explain the connections with fixed subgroups of automorphisms.
\medskip

We introduce a notion that will be used throughout the paper:

\begin{definition} \label{definitionHeight} \cite[Definition 2.18]{vaskou2023isomorphism}
    The \emph{height} of an element $g \in A_{\Gamma}$ is the image of $g$ through the (well-defined) morphism $ht \colon A_{\Gamma} \to \mathbb{Z}$ that sends every standard generator to $1$.
\end{definition}

\subsection{The Deligne complex}

The Deligne complex is built out of the combinatorics of some of its parabolic subgroups. We start by recounting the definition from the introduction:

\begin{definition}
    A parabolic subgroup $g A_{\Gamma'} g^{-1}$ of an Artin group $A_{\Gamma}$ is called \emph{spherical} if the associated Coxeter group $W_{\Gamma'}$ is finite.
\end{definition}

\begin{definition} \label{DefDeligneComplex}
    The \emph{Deligne complex} associated with an Artin group $A_{\Gamma}$ is the simplicial complex $X_{\Gamma}$ defined as follows:
    \begin{enumerate}
        \item The vertices of $X_{\Gamma}$ are the left-cosets of the form $g A_{\Gamma'}$ where $g \in A_{\Gamma}$ and $A_{\Gamma'}$ is a spherical standard parabolic subgroup.
        \item Every string of inclusions of the form $g_0 A_{\Gamma_0} \subsetneq \cdots \subsetneq g_n A_{\Gamma_n}$ spans an $n$-simplex between the corresponding vertices.
    \end{enumerate}
    The Artin group $A_{\Gamma}$ naturally acts on $X_{\Gamma}$ by left-multiplication.
\end{definition}

There are (at least) two interesting metrics on $X_{\Gamma}$. We define the first one thereafter:

\begin{definition}
    The vertex set $X_{\Gamma}^{(0)}$ of the Deligne complex can be given the \emph{combinatorial metric} $d^C$, where the distance $d^C(v, v')$ between two vertices $v$ and $v'$ is their distance in the graph $X_{\Gamma}^{(1)}$ when declaring that every edge has length $1$. We will say that two vertices $v$ and $v'$ are \emph{neighbours} if $d^C(v, v') = 1$. Moreover, for any vertex $v \in X_{\Gamma}^{(0)}$.
\end{definition}

The next metric is a finer piecewise-Euclidean metric known as the Moussong metric. It can be defined for all Deligne complexes, but to make its definition simpler, we will restrict to $2$-dimensional Artin groups. We recall their definition:

\begin{definition} \label{DefiDimension2} 
    An Artin group $A_{\Gamma}$ is said to be \emph{2-dimensional} if $\Gamma$ contains at least one edge, and every triangle $T = T(a, b, c)$ with $a, b, c \in V(\Gamma)$ satisfies
    $$\frac{1}{m_{ab}} + \frac{1}{m_{ac}} + \frac{1}{m_{bc}} \leq 1.$$
\end{definition}

\begin{remark}
    (1) \emph{Large-type} Artin groups, that is Artin groups were $m_{ab} \geq 3$ for every $a, b \in V(\Gamma)$, are always $2$-dimensional.
    \\(2) Equivalent definitions of being $2$-dimensional can be found in \cite[Definition 2.2, Proposition 2.3]{vaskou2023isomorphism}. 
\end{remark}

The results in this paper will generally concern the class of large-type Artin groups, with some intermediate results more generally applying to $2$-dimensional Artin groups. Both classes have been well-studied throughout the literature. One of the major reasons is the following structural theorem, which ensures their geometry is low-dimensional and non-positively curved:

\begin{theorem} \cite[Proposition 4.4.5)]{charney1995k} \label{ThmCAT(0)}
For every $2$-dimensional Artin group $A_{\Gamma}$ the corresponding Deligne complex $X_{\Gamma}$ is $2$-dimensional, and it is CAT(0) when given the Moussong metric.
\end{theorem}

We introduce some notation before explaining how the Moussong metric is defined. This notation will actually be useful throughout the paper.

\begin{notation}
(1) The fundamental domain of the action $A_{\Gamma} \curvearrowright X_{\Gamma}$ will be denoted $K_{\Gamma}$. When $A_{\Gamma}$ is $2$-dimensional, $K_{\Gamma}$ is always isomorphic to the cone over the barycentric subdivision $\Gamma_{bar}$ of the presentation graph $\Gamma$, the cone-point being the vertex associated to $\{1\}$. See Figure \ref{figDeligne} for an example.
\\(2) To avoid the confusion between algebraic and geometric objects, we will more usually denote the cone-point vertex by $v_{\emptyset}$. The vertex that corresponds to a standard generator $a \in V(\Gamma)$ will be denoted $v_a$, while the vertex corresponding to an edge $e^{ab} \in E(\Gamma)$ will be denoted $v_{ab}$. Similarly, we will denote by $e_a$ the edge joining $v_{\emptyset}$ and $v_a$, and by $e_{a, ab}$ the edge joining $v_a$ and $v_{ab}$.
\\(3) The subcomplex of $X_{\Gamma}$ made by all edges with non-trivial stabilisers is a subgraph of the $1$-skeleton of $X_{\Gamma}$. We will call it the \emph{essential 1-skeleton} and denote it by $X_{\Gamma}^{(1)-ess}$.
\end{notation}

\begin{definition} (\cite{charney1995k}) \label{DefiMoussong}
    The \emph{Moussong metric} on the Deligne complex $d$ is defined as follows. Consider three vertices spanning a $2$-dimensional simplex. Up to conjugation, these vertices are of the form $v_{\emptyset}$, $v_a$ and $v_{ab}$. Then the simplex they span is the unique Euclidean triangle (up to isometry) that satisfies
    $$\angle_{v_{ab}}(v_{\emptyset}, v_a) = \frac{\pi}{2 m_{ab}}, \ \ \
    \angle_{v_a}(v_{\emptyset}, v_{ab}) = \frac{\pi}{2}, \ \ \text{ and } \ \
    d(v_{\emptyset}, v_a) = 1.$$
\end{definition}

\begin{figure}[H]
\centering
\includegraphics[scale=1]{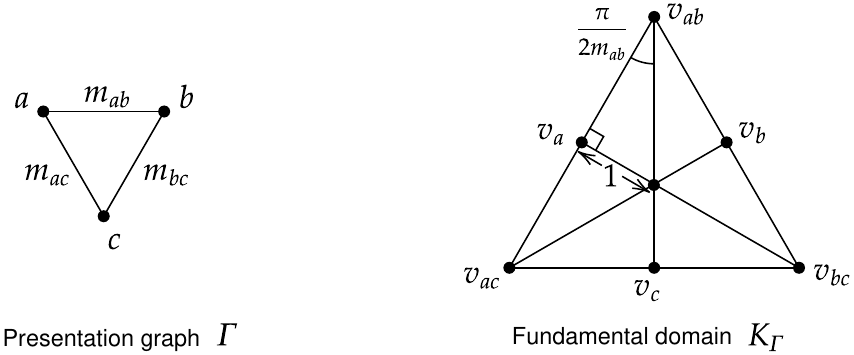}
\caption{A picture of the fundamental domain $K_{\Gamma}$ when $\Gamma$ is a triangle.}
\label{figDeligne}
\end{figure}

The aforementioned $2$-dimensional Artin groups also benefit from convenient algebraic properties, as showed in the next theorem. Let us first recall that an Artin group $A_{\Gamma}$ is called \emph{reducible} if $\Gamma$ splits as a join of two proper subgraphs $\Gamma_1$ and $\Gamma_2$, where every edge of the join is labelled by a $2$. When this happens, the group splits as $A_{\Gamma} \cong A_{\Gamma_1} \times A_{\Gamma_2}$. If $A_{\Gamma}$ is not reducible then it is called \emph{irreducible}.

\begin{theorem} \label{ThmTorsionAndCentres} (\cite{charney1995k}, \cite{brieskorn1972artin}, \cite{vaskou2022acylindrical}) Let $A_{\Gamma}$ be a $2$-dimensional Artin group. Then $A_{\Gamma}$ is torsion-free. Moreover, if $A_{\Gamma}$ is irreducible, then its centre is isomorphic to $\mathbb{Z}$ if $A_{\Gamma}$ is spherical, and it is trivial otherwise.
\end{theorem}

\begin{remark}\label{remarkCentre}
    A large-type Artin group $A_{\Gamma}$ is always irreducible. Moreover, if it has rank at least $3$, then it is always non-spherical. In that case, it has trivial centre by Theorem \ref{ThmTorsionAndCentres}. Consequently, the group $\Inn(A_{\Gamma})$ of inner automorphisms of $A_{\Gamma}$ is isomorphic to $A_{\Gamma}$ itself.
\end{remark}

We now describe how one can extend the action of $A_{\Gamma}$ on $X_{\Gamma}$ to a compatible action from $\Aut_{\Gamma}(A_{\Gamma})$. The following action was originally described in \cite{vaskou2023automorphisms}. We first introduce the following notation, that will also be used throughout the paper:

\begin{notation}
    We will often write $v_S$ to denote an arbitrary vertex of the form $v_{\emptyset}$, $v_s$ or $v_{st}$ for some $s, t \in V(\Gamma)$. The set $S$ should be thought of as the associated set $S \subseteq V(\Gamma)$ of standard generators. In particular, any vertex of $X_{\Gamma}$ can be written as $g v_S$ for some appropriate $g \in A_{\Gamma}$ and $S \subseteq V(\Gamma)$. If $\sigma \in \Aut(\Gamma)$ is a graph automorphism, we will write $v_{\sigma(S)}$ to denote the vertex $v_{\emptyset}$, $v_{\sigma(s)}$ or $v_{\sigma(s)\sigma(t)}$ respectively.
\end{notation}

\begin{definition} \label{DefinitionActionOfAut}
    Let $A_{\Gamma}$ be a $2$-dimensional Artin group. Then there is a simplicial action of $\Aut_{\Gamma}(A_{\Gamma})$ on $X_{\Gamma}$, that can be described explicitly. Let $g v_S$ be any vertex of $X_{\Gamma}$, then:
    \begin{itemize}
        \item If $\varphi_h$ is the conjugation by $h$, then $\varphi_h \cdot g v_S \coloneqq hg v_S$ ;
        \item If $\sigma$ is a graph automorphism, then $\sigma \cdot g v_S \coloneqq \sigma(g) v_{\sigma(S)}$ ;
        \item If $\iota$ is the global inversion, then $\iota \cdot g v_S \coloneqq \iota(g) v_S$.
    \end{itemize}
\end{definition}

\begin{notation}
    For better clarity, and as in Theorem \ref{DefinitionActionOfAut}, we will use the symbol “$\cdot$” to describe the action of $\Aut_{\Gamma}(A_{\Gamma})$ on $X_{\Gamma}$, while only using juxtaposition for the action of $A_{\Gamma}$ on $X_{\Gamma}$.
\end{notation}

\subsection{Minimal sets}

Throughout the paper we will often study elements of $A_\Gamma$ and $\Aut_\Gamma(A_\Gamma)$ by their dynamics on $X_\Gamma$. 

\begin{definition}\label{defMinSet}
    Let $G$ be a group acting by isometries on a metric space $(X, d)$. For any $g \in G$, the \emph{translation length} of $g$ is defined as
    $$||g|| \coloneqq \inf \{ d(x, gx) \ | \ x \in X \}.$$
    Whenever this infimum is reached, we define the \emph{minimal set} of $g$ as the set of points in $X$ for which the displacement induced by $g$ is minimal:
    $$X^g \coloneqq \{ x \in X \ | \ d(x, gx) = ||g|| \}.$$
    An element $g \in G$ with $X^g \neq \emptyset$ is said to act \emph{elliptically} on $X$ if $||g|| = 0$ and \emph{hyperbolically} if $||g|| > 0$.
\end{definition}

In our setting, that is when $X = X_{\Gamma}$ is the Deligne complex, all elements considered will act elliptically or hyperbolically. This is due to our actions being by simplicial automorphisms (hence isometries), along with the following:

\begin{remark} \cite[Chapter II.6]{bridson2013metric}
    Let $X$ be a piecewise-Eucliean simplicial complex with finitely many shapes. Then every simplicial isometry of $X$ is either elliptic or hyperbolic.
\end{remark}

Before going to more in-depth discussions about minimal sets for either elliptic or hyperbolic elements, we introduce the notion of type which we will use throughout. We will first need the following notion:

\begin{proposition} \cite[Corollary 16]{cumplido2023parabolic}
    Let $A_{\Gamma}$ be a large-type Artin group. Then for every $g \in A_{\Gamma}$ there exists a unique minimal parabolic subgroup $P_g$ containing $g$. We call $P_g$ the \emph{parabolic closure} of $g$.
\end{proposition}

The following notion should be seen as measuring the “complexity” of an element of $A_{\Gamma}$ or of a point of $X_{\Gamma}$.

\begin{definition} \label{DefiType} \cite[Definition 2.10]{vaskou2023isomorphism} Let $A_{\Gamma}$ be a $2$-dimensional Artin group. A parabolic subgroup $g A_{\Gamma'} g^{-1}$ of $A_{\Gamma}$ has \emph{type} $n$ if $|V(\Gamma')| = n$. The \emph{type} of an element $g \in A_{\Gamma}$ is the smallest integer $n$ such that $g$ is contained in a parabolic subgroup of type $n$. Finally, the \emph{type} of a point $p \in X_{\Gamma}$ is defined as the type of its stabiliser, seen as a parabolic subgroup of $A_{\Gamma}$.
\end{definition}

\begin{remark} \label{remTypePreserving}
(1) If $A_{\Gamma}$ is large-type, then the type of an element $g$ agrees with the type of its parabolic closure $P_g$.
\\(2) The types of the vertices $v_{\emptyset}$, $v_a$ and $v_{ab}$ are respectively $0$, $1$ and $2$.
\\(3) The action of $\Aut(A_{\Gamma})$ on $X_{\Gamma}$ described in Definition \ref{DefinitionActionOfAut} preserves the type of vertices.
\end{remark}

\subsubsection{Elliptic minimal sets and type}

Let $G$ be either $A_{\Gamma}$ or $\Aut_{\Gamma}(A_{\Gamma})$. In this section we briefly discuss the minimal sets of elliptic elements of the action $G \curvearrowright X_{\Gamma}$. For $g \in G$ elliptic, we will freely call $X_{\Gamma}^g$ the fixed set of $A_{\Gamma}$, since it is the set of points fixed by $g$.

\begin{lemma} \label{LemmaMinsetConvex}
    Let $G$ be a group acting by isometries on a uniquely geodesic metric space $X$, and let $g \in G$ act elliptically on $X$. Then $X^g$ is convex (in particular, it is also uniquely geodesic). Moreover, if $X$ is a simplicial complex and the action is without inversions, then $X^g$ is a simplicial subcomplex of $X$.
\end{lemma}

\begin{proof}
    Let $x$ and $y$ be any two points in $X^g$, and let $\gamma$ be the unique geodesic connecting $x$ and $y$ in $X$. The element $g$ fixes both $x$ and $y$ so it fixes $\gamma$. In particular, $\gamma$ is contained in $X^g$, so $X^g$ is convex.

    Recall that in a simplicial complex, a given point $p$ can belong to the interior of at most one simplex, that we call $\Delta_p$. To show that $X^g$ is a simplicial subcomplex of $X$, we must show that for every point $p \in X^g$, then if $\Delta_p$ exists, then it is also contained in $X^g$. By hypothesis, $g \cdot p = p$, so by unicity we obtain $g \cdot \Delta_p = \Delta_{g \cdot p} = \Delta_p$. Because the action is without inversions, this means $g$ fixes $\Delta_p$ pointwise, as wanted.
\end{proof}

The following is a very useful description of fixed sets for the action $A_{\Gamma} \hookrightarrow X_{\Gamma}$.

\begin{lemma} \label{LemmaClassificationByType} \cite[Lemma 8]{crisp2005automorphisms}
Let $g \neq 1$ be an element of a $2$-dimensional Artin group $A_{\Gamma}$. Then exactly one of the following happens:
\begin{itemize}
    \item $type(g) = 1$. Then $\langle g \rangle$ is contained in a parabolic subgroup $h \langle a \rangle h^{-1}$ for some $a \in V(\Gamma)$ and $h \in A_{\Gamma}$. Moreover, $g$ acts elliptically on $X_\Gamma$, and $X_{\Gamma}^g$ is the tree $h X_{\Gamma}^a$. Such tree are called standard trees.
    \item $type(g) = 2$. Then $\langle g \rangle$ is contained in a parabolic subgroup $h \langle a, b \rangle h^{-1}$ for some $a, b \in V(\Gamma)$ and some $h \in A_{\Gamma}$. If $m_{ab} < \infty$, then $g$ acts elliptically on $X_\Gamma$ and $X_{\Gamma}^g$ is the single vertex $h v_{ab}$. If $m_{ab} = \infty$, then $g$ acts hyperbolically on $X_{\Gamma}$.
    \item $type(g) \geq 3$. Then the element $g$ acts hyperbolically on $X_{\Gamma}$.
\end{itemize}
\end{lemma}

\subsubsection{Hyperbolic minimal sets and transverse-trees}

There is more to be said about hyperbolic minimal sets, hence the results of this section are slightly more technical. We start with the following structural theorem:

\begin{theorem} \cite[Chapter II.6]{bridson2013metric} \label{TheoremBH} Let $G$ be a group acting by isometries on a CAT(0) metric space $X$, and let $g \in G$ be an element acting hyperbolically. Then $Min(g)$ is a closed, convex and non-empty subspace of $X_{\Gamma}$ (in particular, it is CAT(0)). It is isometric to a direct product $\mathcal{T} \times \mathbb{R}$ on which $g$ acts trivially on the first component, and as a translation on the second component. Every axis $u$ of $g$ decomposes as $u = \bar{u} \times \mathbb{R}$, where $\bar{u}$ is a point of $\mathcal{T}$. Furthermore, the centraliser $C(g)$ leaves $Min(g)$ invariant. It sends axes to axes, so $C(g)$ also acts on $\mathcal{T}$.
\end{theorem}

\begin{lemma} \cite[Lemma 3.4]{vaskou2023isomorphism} \label{LemmaTransverseTree}
    Let $A_{\Gamma}$ be a $2$-dimensional Artin group with Deligne complex $X_{\Gamma}$. Let $G$ be a group acting on $X_{\Gamma}$ by isometries and let $g \in G$ be an element acting hyperbolically. Let $X_\Gamma^g \cong \mathcal{T} \times \mathbb{R}$ be the decomposition given in Theorem \ref{TheoremBH}. Then $\mathcal{T}$ is a (real-)tree. We will call $\mathcal{T}$ the \emph{transverse-tree} of $g$.
\end{lemma}

\begin{definition} \cite[Section 2.5]{vaskou2023isomorphism} \label{DefiExotic}
    A subgroup $H$ of $A_{\Gamma}$ is called a \emph{dihedral Artin subgroup} if it is abstractly isomorphic to a (non-abelian) dihedral Artin group. Such a subgroup $H$ is said to be \emph{exotic} if it is not contained in a dihedral Artin parabolic subgroup of $A_{\Gamma}$. Finally, a dihedral Artin subgroup $H$ of $A_{\Gamma}$ is said to be \emph{maximal} if it is not strictly contained in another dihedral Artin subgroup of $A_{\Gamma}$.
\end{definition}

\begin{theorem} \label{TheoremExotic} \cite[Theorem D]{vaskou2023isomorphism}
    Let $A_{\Gamma}$ be a large-type Artin group, and let $H$ be a maximal exotic dihedral Artin subgroup of $A_{\Gamma}$. Then  up to conjugation, there are three standard generators $a, b, c \in V(\Gamma)$ satisfying $m_{ab} = m_{ac} = m_{bc} = 3$ such that
    \begin{align*}
        &H = \langle s, t \ | \ stst = tsts \rangle, \\
        &s \coloneqq b^{-1}, \ \ \ t = babc, \\
        &z \coloneqq stst = tsts = abcabc.
    \end{align*}
\end{theorem}

\begin{proposition} \label{PropBassSerreExotic}
    Let $A_{\Gamma}$, $H$ and $z$ be as in Theorem \ref{TheoremExotic}. Then $H = C(z) = C(z^n)$ for any $n \neq 0$. Moreover, the central quotient $\overline{C(z)} \coloneqq \quotient{C(z)}{\langle z \rangle}$ is isomorphic to the following free product:
    $$\overline{C(z)} \cong \langle \overline{b} \rangle * \langle \overline{abc} \rangle \cong \mathbb{Z} * \left(\quotient{\mathbb{Z}}{2 \mathbb{Z}}\right). \ \ (*)$$
    Consequently, the transverse-tree $\mathcal{T}$ associated with any $z^n$ is isomorphic to the natural Bass-Serre tree associated with the free product $(*)$. A picture of the fundamental domain of $H = C(z)$ acting on $Min(z) = Min(z^n)$ is given in Figure \ref{FigureExoticMinset1}
\end{proposition}

\begin{definition} \label{DefinitionPrincipalTriangle}
    Let $a, b, c \in V(\Gamma)$ be such that $m_{ab} = m_{ac} = m_{bc} = 3$.
    The \emph{principal triangle} associated with the triplet $a, b, c$ is the subcomplex $K$ of the fundamental domain $K_{\Gamma}$ that is spanned by the vertices $v_{\emptyset}$, $v_a$, $v_b$, $v_c$, $v_{ab}$, $v_{ac}$ and $v_{bc}$ (see Figure \ref{FigureExoticMinset1}). A subcomplex of $X_{\Gamma}$ in the $A_{\Gamma}$-orbit of a principal triangle will also be called a principal triangle.     
\end{definition}

\begin{remark} \label{RemSystemColours}
    (1) A principal triangle is the union of $6$ Euclidean $2$-simplices, and one can easily check that their union is equilateral Euclidean relatively to the Moussong metric, i.e. can be isometrically embedded in the Euclidean plane as an equilateral triangle.
    \\(2) When working with principal triangles, we will usually see $K$ with a broader simplicial structure: the on obtained by not considering its edges of type $0$, and its vertices of type $0$ and $1$. In this setting, an \emph{edge} of $K$ will be a union such as $e^a \coloneqq e_{a,ab} \cup e_{a,ac}$, and $K$ has three edges (see Figure \ref{FigureExoticMinset1}).  
    \\(3) Throughout the paper we will be led to work with subcomplexes of $X_{\Gamma}$ that are tiled by principal triangles. Let $Y$ be such a subcomplex. We will often set a \emph{system of colours} on the edges of $Y$, i.e., a colouring of its edges such that two edges have the same colour if and only if they are in the same $A_{\Gamma}$-orbit. It directly follows that for every $g \in A_{\Gamma}$ for which $g \cdot Y = Y$, the simplicial isometry induced by the action of $g$ preserves the system of colours of $Y$. An example of a system of colours in showed in Figure \ref{FigureExoticMinset1}.
\end{remark}

\begin{figure}[H]
    \centering
    \includegraphics[scale=0.9]{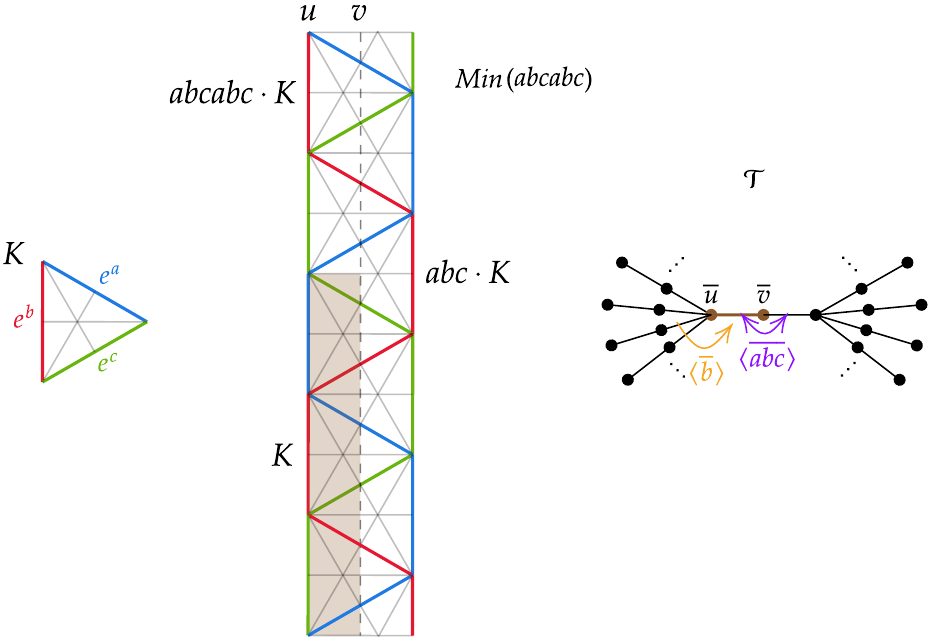}
    \caption{Consider $abcabc$ as in Theorem \ref{TheoremExotic}. \underline{Left:} A principal polygon $K$ associated with $a$, $b$ and $c$. \underline{Middle:} A small part of $Min(abcabc)$. The action of $abcabc$ is vertical. The brown area represents the fundamental domain of the action of $C(abcabc)$ on $Min(abcabc)$. Edges with the same colour are in the same $A_{\Gamma}$-orbit. \underline{Right:} Part of the transverse-tree $\mathcal{T}$ associated with $abcabc$. The brown edge represents the fundamental domain of the action of $\overline{C(abcabc)}$ on $\mathcal{T}$.}
    \label{FigureExoticMinset1}
\end{figure}

We finish with the following classification:

\begin{lemma} \cite[Section 3, Lemma 4.13]{vaskou2023isomorphism} \label{LemmaPossibleMinsets}
    Let $A_{\Gamma}$ be a large-type Artin group, and let $g \in A_{\Gamma}$ be hyperbolic. We suppose that the transverse-tree $\mathcal{T}$ is unbounded. Then we are in one of the following situations:
    \begin{enumerate}
        \item $g$ has an axis $u$ that is contained in a standard tree. Then there exists some $n \neq 0$ such that $g = z^n$, where $z$ is the element from Theorem \ref{TheoremExotic}. Moreover, $C(g) = C(z)$ is the dihedral Artin subgroup $H$ from Theorem \ref{TheoremExotic}. The minimal set $Min(g) = Min(z)$ is described in Proposition \ref{PropBassSerreExotic} and Figure \ref{FigureExoticMinset1}.
        \item $g$ has no axis that is contained in a standard tree. Then $Min(g)$ is a subcomplex of $X_{\Gamma}$ isometric to a tiling of the Euclidean plane by principal triangles. Moreover, $C(g)$ is contained in the isometry group of this tiled plane.
    \end{enumerate}
\end{lemma}

More will be said about the structure of the centralisers of hyperbolic elements in Proposition \ref{PropCentralisers}.

\subsection{Connections with fixed subgroups of automorphisms}

The question of understanding the fixed subgroups is easier for inner automorphisms, as $\Fix(\varphi_g)$ is simply the centraliser $C(g)$. In this section we give several basic results regarding automorphisms, and we discuss the connections between fixed subgroups and stabilisers of subcomplexes in the Deligne complex, hinting at how to compute fixed subgroups in the more general case.

The action of $\Aut_\Gamma(A_\Gamma)$ described in Section 2.2 is \emph{compatible} in the sense of the following definition.

\begin{definition}\label{defCompatible}
Let $G$ be a group acting on a metric space $X$ via $\Omega\colon G \rightarrow \Isom(X)$. We say that a subgroup $A$ satisfying $\Inn(G) \leq A \leq \Aut(G)$ has a \emph{compatible} action on $X$ if there is an action $\Psi\colon A \rightarrow \Isom(X)$ such that $\Psi \circ i = \Omega$, where $i\colon G \rightarrow A$ is the homomorphism $g \mapsto \varphi_g$.
\end{definition}

\begin{remark}
    It is immediate from Theorem \ref{DefinitionActionOfAut} that the action of $\Aut_\Gamma(A_\Gamma)$ described is compatible in the sense of Definition \ref{defCompatible}.
\end{remark}

We will freely use the following lemmas to manipulate automorphisms.

\begin{lemma} \label{LemmaCommuteWithInner}
    Let $G$ be a group, let $g \in G$ and let $\psi \in \Aut(G)$. Then $\psi \varphi_g = \varphi_{\psi(g)} \psi$.
\end{lemma}

\begin{proof}
    This follows from the fact that for every element $h \in A_{\Gamma}$, we have
    $$\psi \varphi_g (h) = \psi(g h g^{-1}) = \psi(g) \psi(h) \psi(g)^{-1} = \varphi_{\psi(g)} \psi(h).$$
\end{proof}

\begin{lemma}\label{compatibilityLemma}
Let $G$ be a group acting on a metric space $X$ such that $\Inn(G) \leq A \leq \Aut(G)$ acts on $X$ in a compatible way. Then for all $\psi \in A$, $g \in G$, and $x \in X$, $\psi \cdot gx = \psi(g) (\psi \cdot x)$.
\end{lemma}

\begin{proof}
Using Lemma \ref{LemmaCommuteWithInner}, we have
$$\psi \cdot gx = \psi \varphi_g \cdot x = \varphi_{\psi(g)} \psi \cdot x = \psi(g) (\psi \cdot x).$$
\end{proof}

The following corollary relates fixed subgroups and minimal sets.

\begin{definition}
    Given a group $G$ acting on a space $X$ and $Y \subseteq X$, the \emph{setwise stabiliser} of $Y$ is defined as $\Stab(Y) \coloneqq \{g \in G \ | \ \forall y \in Y, \ gy \in Y\}$.
\end{definition}

\begin{corollary}\label{fixAction}
If $G$, $A$ and $X$ are as above, then for all $\psi \in A$, we have an inclusion $\Fix(\psi) \leq \Stab(X^\psi)$.
\end{corollary}

\begin{proof}
Using Lemma \ref{compatibilityLemma}, for any $g \in \Fix(\psi)$ and any $x \in X$ we have
$$d(gx, \psi \cdot gx) = d(gx, \psi(g) (\psi \cdot x)) = d(gx, g(\psi \cdot x)) = d(x, \psi \cdot x).$$
So $x \in X^\psi$ implies $gx \in X^\psi$. It follows that $g \in \Stab(X^\psi)$.
\end{proof}

\begin{remark}
    Notice that the action of $G$ on $X$ restricts to an action of $\Fix(\psi) \leq G$ on $X$, and by Corollary \ref{fixAction} this restricts to an action of $\Fix(\psi)$ on $X^\psi$.
\end{remark}

\begin{corollary}\label{corPreserveFixMinsets}
      If $G$, $A$ and $X$ are as above, then for all $\psi \in A$ and $g \in G$ such that $g \in \Fix(\psi)$, $\psi$ preserves $X^g$.
\end{corollary}

\begin{proof}
    Similarly to Corollary \ref{fixAction}, $$d(\psi \cdot x, g(\psi \cdot x)) = d(\psi \cdot x, \psi \cdot \psi^{-1}(g)x) = d(x, \psi^{-1}(g)x) = d(x, gx),$$ so if $x \in X^g$ then $\psi \cdot x \in X^g$ too.
\end{proof}

\begin{lemma} \label{reductionToStabs}
    If $G$, $A$ and $X$ are as above, $\psi \in A$, and $X^{\psi}$ is a single point $v$, then $\psi$ restricts to an automorphism of $\Stab(v)$ and $\Fix(\psi) = \Fix(\left. \psi \right|_{\Stab(v)})$.
\end{lemma}

\begin{proof}
    By hypothesis $\psi$ fixes $v$, so for every $g \in \Stab(v)$ we have
    $$\psi(g) v = \psi(g) (\psi \cdot v) = \psi \cdot gv = v,$$
    and thus $\psi(g) \in Stab(v)$. So in particular $\psi$ restricts to an endomorphism of $\Stab(v)$. Doing the same argument with $\psi^{-1}$ shows that it restricts to an automorphism. Finally by Corollary \ref{fixAction}, we get that $\Fix(\psi) = \Fix(\left. \psi \right|_{\Stab(v)})$.
\end{proof}
 
Our strategy for computing $\Fix(\psi)$ will usually be to determine $X_\Gamma^\psi$, then by Corollary \ref{fixAction} restrict our attention to $\Stab(X_{\Gamma}^\psi)$. The following lemma will be useful to compute such fixed sets:

\begin{lemma} \label{LemmaUsefulEquation}
Let $A_{\Gamma}$ be a $2$-dimensional Artin group, and let us consider some $h \in A_{\Gamma}$ and some automorphism $\psi \coloneqq \sigma \iota^{\varepsilon}$, where $\varepsilon \in \{0, 1\}$. Let $g v_S$ be any vertex of $X_{\Gamma}$. Then:
\\ $\bullet$ $\varphi_h \psi \cdot g v_S = h \psi(g) v_{\sigma(S)}$ ;
\\ $\bullet$ $g v_S \in X_{\Gamma}^{\varphi_h \psi} \Longleftrightarrow h \psi(g) \in g A_S$ and $\sigma(S) = S.$
\end{lemma}

\begin{proof}
    The first statement directly follows from the way the action of $\Aut_{\Gamma}(A_{\Gamma})$ on $X_{\Gamma}$ was defined (Definition \ref{DefinitionActionOfAut}). The second statement follows from the first statement. Indeed, the index $S$ of $v$ determines its orbit, so we must have $S = \sigma(S)$. Now two vertices $g v_S$ and $g' v_S$ agree if and only if $g' g^{-1}$ belongs to the local group at $v_S$, that is, $A_S$. The result follows.
\end{proof}

We introduce one further notion, which will let us reduce the computation of fixed subgroups to a small number of cases.

\begin{definition}\label{definitionIsogredience}
    Given a group $G$ and $\psi_1, \psi_2 \in \Aut(G)$, we say $\psi_1$ and $\psi_2$ are \emph{isogredient} if there exists $h$ such that $\psi_1 = \varphi_h \psi_2 \varphi_h^{-1}$. 
\end{definition}

If $\psi_1 = \varphi_h \psi_2 \varphi_h^{-1}$, we will sometimes write $\psi_1 \sim_h \psi_2$. Notice that isogredience is an equivalence relation, but $\sim_h$ of course is not even symmetric in general.
\medskip

It turns out that up to conjugation, the fixed subgroup depends only on the isogredience class of an automorphism:

\begin{lemma}\label{lemmaIsogredience}
    Suppose $\psi_1 \sim_h \psi_2$. Then $\Fix(\psi_1) = h\Fix(\psi_2)h^{-1}$.
\end{lemma}
\begin{proof}
    Notice that, for any $g \in G$, $\psi_1(hgh^{-1}) = (\varphi_h \psi_2 \varphi_h^{-1}) (hgh^{-1}) = h\psi_2(g)h^{-1}$. The result follows immediately. 
\end{proof}

\section{Fixed subgroups in Dihedral Artin Groups}\label{sectionDihedral}

Many of our subsequent arguments to understand fixed subgroups of large-type Artin groups $A_{\Gamma}$ will rely on understanding the fixed subgroups of the induced automorphisms on dihedral Artin parabolic subgroups (in particular, those which are set-wise fixed so the notion of an induced automorphism makes sense). Thus, in this section, we study fixed subgroups in dihedral Artin groups. Throughout this section we let $A_m$ be a dihedral Artin group with coefficient $m$, that is $$A_m = \langle a, b \mid  \ \underbrace{aba\cdots}_{m \text{ terms}} = \underbrace{bab\cdots}_{m \text{ terms}} \rangle.$$

A dihedral Artin group $A_m$ will be called \emph{odd} or \emph{even} if $m$ is odd or even respectively. We will assume $m \geq 3$. We will write $\gars$ for the \emph{Garside element} of $A_m$, that is
$$\gars = \underbrace{aba\cdots}_{m \text{ terms}} = \underbrace{bab\cdots}_{m \text{ terms}}.$$
Note we omit reference to $m$, which will be clear from context.

Unlike in the case where $A_\Gamma$ is large-type and rank at least 3, dihedral Artin groups have non-trivial cyclic centre.

\begin{lemma}\cite{brieskorn1972artin}
    The dihedral Artin group $A_n$ has infinite cyclic centre generated by $z_{ab}$, where $z_{ab} = \gars$ (if $m_{ab}$ is even), and $z_{ab} = \gars^2$ (if $m_{ab}$ is odd).
\end{lemma}

We will restrict our attention to the fixed subgroups of automorphisms in a specific subgroup of the automorphism group of the dihedral Artin groups. As advertised, this will be the subgroup $\Aut_\Gamma(A_m)$, where $\Gamma$ is the obvious defining graph, which is just an edge labelled with $m$. In particular, $\Aut_\Gamma(A_m)$ is generated by inner automorphisms, the graph automorphism $\sigma\colon a \mapsto b, b \mapsto a$ and the global inversion $\iota\colon a \mapsto a^{-1}, b \mapsto b^{-1}$.

\begin{remark}
    In an odd dihedral Artin group $A_{2n+1}$, the graph automorphism $\sigma$ is equal to the inner automorphism $\varphi_{\gars}$. The quotient $\Aut_\Gamma(A_{2n+1})/\Inn(A_{2n+1})$ therefore forms a $C_2$ subgroup of $\Out(A_{2n+1})$. In an even dihedral Artin group $A_{2n}$, the quotient $\Aut_\Gamma(A_{2n})/\Inn(A_{2n})$ forms a $C_2 \times C_2$ subgroup of $\Out(A_{2n})$, since neither $\sigma$, $\iota$ nor their product are inner.
\end{remark}

It is worth noting that there exist automorphisms (outside of $\Aut_\Gamma(A_m)$) of even dihedral Artin groups with non-finitely generated fixed subgroups.

\begin{example}\label{infinitelyGenerated}
Let $A_{2n} = \langle a, b \ | \ (ab)^n = (ba)^n \rangle$, where $n \geq 2$, be a dihedral Artin group. Then $\gamma\colon G \rightarrow G$ given by $a \mapsto (ab)^{-n}a$ and $b \mapsto b(ab)^n$ is an automorphism. Furthermore, $\Fix(\gamma) = \{g \in G | f(g) = 0\}$, where $f\colon G \rightarrow \mathbb{Z}$ is given by $f\colon a \mapsto -1, b \mapsto 1$. This follows from the fact that $(ab)^n$ is in the centre of $G$. This fixed subgroup is not finitely generated.
\end{example}

Throughout this section we will restrict our attention to the automorphisms in $\Aut_\Gamma(A_m)$. To calculate fixed subgroups in $A_m$, we will make use of the following splittings.

\begin{lemma}\label{dihedralArtinTreeForm}
    \begin{enumerate}
        \item There is an isomorphism $\Phi\colon \langle x, y \ | \ x^2 = y^{2n+1} \rangle \rightarrow A_{2n+1} $, given by $$\Phi(x) = b(ab)^n \ \text{ and } \ \Phi(y) = ab.$$

        \item There is an isomorphism $\Psi\colon \langle x, t \ | \ tx^nt^{-1} = x^n \rangle \rightarrow A_{2n}$, given by $$\Psi(x) = ab \ \text{ and } \ \Psi(t) = b.$$
    \end{enumerate}
\end{lemma}
\begin{proof}
    It is straightforward to check that the proposed isomorphisms are homomorphisms, and that inverse homomorphisms are given by  $\Phi^{-1}(a) = y^{n+1}x^{-1}$, $\Phi^{-1}(b) = xy^{-n}$, and $\Psi^{-1}(a) = xt^{-1}$, $\Psi^{-1}(b) = t$.

\end{proof}

From here, we will often write the presentations in Lemma \ref{dihedralArtinTreeForm} for the corresponding dihedral Artin groups $A_m$. We always implicitly identify them via the isomorphisms in the lemma.

The presentations in Lemma \ref{dihedralArtinTreeForm} expose splittings as an amalgamated product (for $A_{2n+1}$) or HNN extension (for $A_{2n}$). These splittings give rise to actions on the Bass-Serre trees associated with the splittings, which will be a key tool in this section. We briefly postpone the precise definition of the Bass-Serre trees to discuss the strategy.

Our method for finding fixed subgroups in the dihedral case will mimic our methods for Artin groups of higher rank later in the paper. We will make use of a compatible action of the automorphism group in the sense of Definition \ref{defCompatible}. For the dihedral Arting groups, the compatible action will be on the Bass-Serre tree corresponding to the splittings in Lemma \ref{dihedralArtinTreeForm}.

The fixed subgroups of odd dihedral Artin groups are computed by this method in \cite{jones2023fixed}, which classifies fixed subgroups for all torus knot groups, that is groups with a presentation $\langle x, y \ | \ x^p = y^q \rangle$. We summarise the relevant results in Theorem \ref{dihedralFixPointsOdd}.

\begin{theorem}\cite{jones2023fixed}\label{dihedralFixPointsOdd}
    Let $A_{2n+1}$ be an odd dihedral Artin group. Then the fixed subgroups of non-trivial automorphisms are as follows:
    
    \begin{enumerate}
        \item If $\varphi_g$ is of finite order, then $\Fix(\varphi_g) \cong \mathbb{Z}$ with $\langle g \rangle \leq_{f.i.} \Fix(\varphi_g)$. 

        Moreover, restricting to the case $g = \gars$, $\Fix(\varphi_{\gars}) = \langle \gars \rangle$.
        \item If $\varphi_g$ is of infinite order, then $\Fix(\varphi_g) \cong \mathbb{Z}^2$ with $\langle g \rangle \times \langle \gars^2 \rangle \leq_{f.i} \Fix(\varphi_g)$.
        \item If $\varphi_g\iota$ is of finite order, then $\Fix(\varphi_g\iota) = \{1\}$.
        \item If $\varphi_g\iota$ is of infinite order, then $\Fix(\varphi_g\iota) \cong \mathbb{Z}$ with $ \langle g \rangle \leq_{f.i.} \Fix(\varphi_g\iota)$.
    \end{enumerate}

\end{theorem}

\begin{remark}
    The only point in the above theorem not explicitly shown in \cite{jones2023fixed} is the exact fixed subgroup for $\varphi_{\gars}$, but it follows easily from the proofs in that paper. In particular, it follows from the general statement that $\langle \gars \rangle \leq \Fix(\gars) \cong \mathbb{Z}$, and equality is achieved because $\langle \gars \rangle$ is a maximal cyclic subgroup (as it stabilises a vertex in the Bass-Serre tree considered in that paper).
\end{remark}

We now compute the fixed subgroups of automorphisms in $\Aut_\Gamma(A_m)$ for dihedral Artin groups $A_{2n}$. Throughout we will only consider the large-type case, that is $n \geq 2$. The remainder of this section will be proofs of the following, collected together at the end of the section.

\begin{theorem}\label{dihedralFixPointsEven}
    Let $A_{2n}$ be an even dihedral Artin group. Then we determine the fixed subgroups of each inducible automorphism.

    \begin{enumerate}

        \item If $\varphi_g$ is of finite order, then $\Fix(\varphi_g) \cong \mathbb{Z}$ with $\langle g \rangle \leq_{f.i.} \Fix(\varphi_g)$. Moreover $\langle \gars \rangle \leq_{f.i.} \Fix(\varphi_g)$.

        \item If $\varphi_g$ is of infinite order, then $\Fix(\varphi_g) \cong \mathbb{Z}^2$ with $\langle g \rangle \times \langle \Delta_{ab} \rangle \leq_{f.i.} \Fix(\varphi_g)$.

        \item If $\varphi_g\iota$ is finite order, then $\Fix(\varphi_g\iota) = \{1\}$.

        \item If $\varphi_g\iota$ is of infinite order, then $\Fix(\varphi_g\iota) \cong \mathbb{Z}$ with $\langle g\iota(g) \rangle \leq_{f.i.} \Fix(\varphi_g\iota)$ . 

        \item If $\varphi_g\sigma$ is finite order then $\Fix(\varphi_g\sigma) \cong \mathbb{Z}$. 
        
        Moreover, restricting to the case $g = 1$, $\Fix(\sigma) = \langle \Delta_{ab} \rangle$.

        \item If $\varphi_g\sigma$ is of infinite order then $\Fix(\varphi_g\sigma) \cong \mathbb{Z}^2$ with $\langle g\sigma(g) \rangle \times \langle \gars \rangle \leq_{f.i.} \Fix(\varphi_g\sigma)$.

        \item If $\varphi_g\sigma\iota$ is finite order then $\Fix(\varphi_g\sigma\iota) \cong \mathbb{Z}$. 
        
        Moreover, restricting to the case $g = 1$: for $n$ even, $\Fix(\sigma\iota) = \langle (ab)^{\frac{n}{2}}(ba)^{-\frac{n}{2}} \rangle$; For $n$ odd, $\Fix(\sigma\iota) = \langle b(ab)^{\frac{n-1}{2}}(ba)^{-\frac{n-1}{2}}a^{-1} \rangle$.

        \item If $\varphi_g\sigma\iota$ is of infinite order then $\Fix(\varphi_g \iota \sigma)  \cong \mathbb{Z}$ with $\langle g\sigma(g) \rangle \leq_{f.i.} \Fix(\varphi_g \iota \sigma)$.

    \end{enumerate}
\end{theorem}

For the proofs, we will use the presentation as a Baumslag-Solitar group from Lemma \ref{dihedralArtinTreeForm}, so $A_{2n} = \langle x,t | tx^nt^{-1} = x^n \rangle$. As advertised, we will consider the Bass-Serre tree corresponding to this splitting, which we are now ready to define (specialised to the even case).

\begin{definition}
    Given $A_{2n}$ viewed as the Baumslag-Solitar group $\langle x, t | tx^nt^{-1} = x^n \rangle$, the \emph{Bass-Serre tree}, $T$, is the following graph. Take vertices $v_{g\langle x \rangle}$ and edges $e_{g\langle x^n \rangle}$, identifying simplices where the corresponding cosets are equal. The endpoints of $e_{g\langle x^n \rangle}$ are $v_{g\langle x \rangle}$ and $v_{gt\langle x \rangle}$.
\end{definition}

We will write $G_s$ for the stabiliser of a simplex $s$. 

\begin{remark}\label{dihedralEvenOrientation}
     Each edge $e_{g\langle x \rangle}$ can be viewed as an oriented edge from $v_{g\langle x \rangle}$ to $v_{gt\langle x \rangle}$, and the group action respects this orientation.

\end{remark}

\begin{lemma}\cite{serre2002trees}\label{bassSerreLemma}
    The graph $T$ is a tree, justifying the name Bass-Serre tree. $A_{2n+1}$ acts on the left without inversion by left multiplication on cosets. The stabilisers are as follows, with vertices uniquely determined by their stabilisers: $G_{gv_{\langle x \rangle}} = g\langle x \rangle g^{-1}$ and $G_{ge_{\langle x^n \rangle}} = \langle x^n \rangle$.
\end{lemma}

Next we will build a compatible action of $\Aut(A_{2n})$ on $T$. The (outer) automorphism group is as follows.

\begin{theorem}{\cite[Theorem D]{ghrm2000treeaction}}\label{dihedralEvenOut}
    $\Out(A_{2n}) \cong C_2 \times D_{\infty}$, with generators $[\alpha]$, $[\beta]$, $[\gamma]$,
    \begin{align*}
    \alpha\colon& \ x \mapsto x^{-1}, t \mapsto t\\
    \beta\colon& \ x \mapsto x, t \mapsto t^{-1}\\
    \gamma\colon& \ x \mapsto x, t \mapsto tx,
    \end{align*}
    where $[\beta]$ and $[\gamma]$ generate the $D_\infty$ factor, and $[\alpha\beta]$ generates the $C_2$ factor.
\end{theorem}

We are now ready to construct the compatible action. This is in the spirit of \cite{ghrm2000treeaction}, which builds compatible actions for many one relator groups acting on trees, but does not handle this case.

\begin{proposition}
    Let $A_{2n} = \langle x, t | tx^nt^{-1} = x^n \rangle$ be an even dihedral Artin group and $T$ be the Bass-Serre tree. Then there is a compatible action of $\Aut(A_{2n})$ on $T$, where, given a vertex $v$ and an automorphism $\varphi$, $$G_{\varphi \cdot v} = \varphi(G_v).$$ Given that vertices of $T$ are uniquely determined by their stabilisers, this defines the action.
\end{proposition}

\begin{proof}
    First, we show that the vertex stabilisers are permuted by the automorphism. Write an arbitrary automorphism as $\varphi_h\psi$ where $\varphi_h$ is conjugation by $h$ and $\psi$ is in the subgroup generated by $\alpha$, $\beta$, $\gamma$, as defined in Theorem \ref{dihedralEvenOut}. Notice that $\psi(\langle x \rangle) = \langle x \rangle$, so $$\varphi_h\psi(g\langle x \rangle g^{-1}) = h\psi(g)\langle x \rangle \psi(g)^{-1}h^{-1}.$$ So $\varphi_h\psi$ sends vertex stabilisers to vertex stabilisers, and does so surjectively. To see this is well-defined and injective, notice that $\varphi_h\psi(g_1\langle x \rangle g_1^{-1}) = \varphi_h\psi(g_2\langle x \rangle g_2^{-1})$ if and only if $\psi(g_2^{-1}g_1) \in \langle x \rangle$. Since $\psi(\langle x \rangle) = \langle x \rangle$, this is equivalent to $g_2^{-1}g_1 \in \langle x \rangle$, and therefore $g_1\langle x \rangle g_1^{-1} = g_2 \langle x \rangle g_2^{-1}$. Thus, $\Aut(A_{2n})$ acts on the vertex set.
    
    To see that this extends to an action on the tree, we need to show that adjacent vertices are sent to adjacent vertices. We only need to check this for a generating set of $\Aut(A_{2n})$. For inner automorphisms, this follows by Lemma \ref{bassSerreLemma} and the fact that $A_{2n}$ acts on $T$. Now we consider $\gamma$, as defined in Theorem \ref{dihedralEvenOut}, and two adjacent vertices $gv_{\langle x \rangle}$ and $gtv_{\langle x \rangle}$. The image of the stabilisers is as follows, $$\gamma(g\langle x \rangle g^{-1}) = \gamma(g) \langle x \rangle \gamma(g)^{-1},$$ and, $$\gamma(gt\langle x \rangle t^{-1}g^{-1}) = \gamma(g)tx \langle x \rangle x^{-1}t^{-1}\gamma(g)^{-1} = \gamma(g)t\langle x \rangle t^{-1}\gamma(g)^{-1}.$$ In particular $\gamma \cdot gv_{\langle x \rangle} = \gamma(g)v_{\langle x \rangle}$ and $\gamma \cdot gtv_{\langle x \rangle} = \gamma(g)tv_{\langle x \rangle}$, so the image vertices are still adjacent. The check for $\alpha$ and $\beta$ is similar.
    
    Finally, we check compatibility. Given $g \in A_{2n}$, $G_{\varphi_g \cdot v} = gG_vg^{-1}$. Since $gG_vg^{-1} = G_{gv}$ by Lemma \ref{bassSerreLemma}, and vertices are characterised by their stabilisers, it must be that $\varphi_g \cdot v = gv$ as required.
\end{proof}

\begin{remark}\label{dihedralAutOrientation}
    
The action described above has edge inversions.

In particular, the action of $\beta$ inverts the orientation on all the edges of the Bass-Serre tree mentioned in Remark \ref{dihedralEvenOrientation}, and any element which inverts the orientation of the edges and setwise fixes an edge will invert that edge.

It is an easy check that $\alpha$, $\gamma$ and the inner automorphisms all preserve the orientation of the edges.

\end{remark}

The following lemma characterises the $\Aut_\Gamma(A_{2n})$.

\begin{lemma}\label{inducibleBSAutomorphisms}
    The $\Aut_\Gamma(A_{2n})$ is exactly the union of outer automorphism classes $[1]$, $[\alpha\beta]$, $[\alpha\gamma]$ and $[\beta\gamma]$.
\end{lemma}
\begin{proof}
    The inner automorphisms are the same regardless of presentation, so are certainly elements of $\Aut_\Gamma(A_{2n})$. This means $\Aut_\Gamma(A_{2n})$ is a union of outer automorphism classes.

    Now it is not hard to check, using the isomorphisms from Lemma \ref{dihedralArtinTreeForm}, that the automorphism $\iota$ known as the global inversion is given by $\varphi_t \alpha \beta$, and $\sigma$, the sole graph automorphism, is given by $\varphi_t \beta \gamma$. So the $C_2 \times C_2$ subgroup of $\Out(A_{2n})$ comprising of the $\Aut_\Gamma(A_{2n})/\Inn(A_{2n})$ is given by exactly the classes in the statement.
\end{proof}

We are now ready to consider fixed subgroups of the automorphisms in $\Aut_\Gamma(A_{2n})$. First we describe the centralisers. The following result is well known, but proved here using the language of compatible actions for completeness. Notice that for any element $g$ in any group, $C(g) = \{h | ghg^{-1} = h\} = \Fix(\varphi_g)$. We will also make use of centralisers when arguing about fixed subgroups of non-inner automorphisms.

\begin{proposition}\label{dihedralCentralisers}
    The centralisers in $A_{2n}$ are as follows:

    \begin{enumerate}
        \item If $g = x^{kn}$ for some integer $k$, then $C(g) = A_{2n}$.
        \item If $g = hx^kh^{-1}$  where $n \nmid k$, then $C(g) = \langle hxh^{-1} \rangle \cong \mathbb{Z}$.
        \item If $g$ acts hyperbolically on $T$, then $C(g) = \langle \hat{g} \rangle \times \langle x^n \rangle \cong \mathbb{Z}^2$, where $\hat{g}$ is an element of minimal non-zero translation length acting along $T^g$, the axis of translation of $g$.
    \end{enumerate}
\end{proposition}
\begin{proof}
    Since $x^{kn}$ commutes with $x$ and $t$, it is in the centre (in fact, as we will see from the other cases, the centre is exactly $\langle x^n \rangle$). So clearly its centraliser is the whole group.

    If $n \nmid k$, the action of $hx^kh^{-1}$ fixes the vertex $hv_{\langle x \rangle}$ and nothing else, so by Corollary \ref{fixAction} the centraliser of $hx^kh^{-1}$ is contained in $\langle hxh^{-1} \rangle$. Each element in this subgroup clearly commutes with $hx^kh^{-1}$ so this is the whole centraliser.

    If $g$ is not of either of the previous forms then it fixes no vertices so must act hyperbolically. By Corollary \ref{fixAction} the centraliser of $C(g) \leq \Stab(T^g)$. Take $\hat{g}$ an element acting along $T^g$ such that $l(\hat{g})$ is minimal and non-zero, where $l(\cdot)$ is the translation length.

    Suppose $h$ acts along $T^g$. Replacing $h$ by $h^{-1}$ if necessary, we may assume it acts in the same direction as $\hat{g}$. Write $l(h) = ql(\hat{g}) + r$, where $q$ and $0 \leq r < l(\hat(g)$ are integers. Then $h\hat{g}^{-q}$ acts on $T^g$ with translation length $r$. By the assumed minimality of $l(\hat{g})$, it must be that $r = 0$. Therefore $l(h) = ql(\hat{g})$, where $q$ is an integer.

    Notice that $h\hat{g}^{-q}$ fixes the entire axis $T^g$ pointwise. In particular, it is in the stabiliser of some edge, that is $h\hat{g}^{-q} \in \langle x^n \rangle$, or $h \in \langle \hat{g} \rangle \times \langle x^n \rangle$. So $\Stab(T^g) = \langle \hat{g} \rangle \times \langle x^n \rangle$. This is an abelian subgroups containing $g$ so $\Stab(T^g) \leq C(g)$. With both inclusions shown, this completes the proof.
\end{proof}

\begin{lemma}\label{dihedralOrder}
    Take $\psi \in \Aut_\Gamma(A_{2n})$. Then $\psi$ is finite order in $\Aut(A_{2n})$ if and only if the action of $\psi$ fixes a point on $T$.
\end{lemma}

\begin{proof}
    Suppose $\psi$ is of finite order. Then it acts by a finite order isometry of $T$, so must fix a point on $T$ since $T$ is a tree.

    Conversely, suppose $\psi$ fixes a point on $T$. It follows from the automorphism definitions in Theorem \ref{dihedralEvenOut} that the outer automorphism classes $[\alpha\beta]$, $[\beta\gamma]$, $[\alpha\gamma]$ generate a $C_2 \times C_2$ subgroup of $\Out(A_{2n})$. Therefore it must be that $\psi^2$ is inner, write $\varphi_h = \psi^2$. The inner automorphism $\varphi_h$ must fix the point fixed by $\psi$ and, by compatibility, $h$ acts in the same way so is elliptic. By Lemma \ref{dihedralCentralisers} means $h^n$ is central, and $\psi^{2n} = \varphi_{h^n}$ is the identity. So $\psi$ is of finite order as required.
\end{proof}

The fixed subgroups of infinite order automorphisms are closely related to centralisers.

\begin{lemma}\label{dihedralInfiniteOrder}
    Suppose $\varphi_g\psi$ an inducible automorphism of $A_{2n}$ of infinite order, where $\psi \in \{\alpha\beta, \beta\gamma, \alpha\gamma\}$. 

    If $\psi = \beta\gamma$ then $\Fix(\varphi_g\psi) = C(g\psi(g)x) \cong \mathbb{Z}^2$.
    
    Otherwise, $\Fix(\varphi_g\psi) \cong \mathbb{Z}$ with $g\psi(g) \leq_{f.i.} \Fix(\varphi_g\psi)$.
\end{lemma}
\begin{proof}

    For any automorphism $\varphi$, $\Fix(\varphi) \leq \Fix(\varphi^2)$, since if an element is fixed by an automorphism it is clearly fixed by its square.
    
    Notice that $(\varphi_g\psi)^2 = \varphi_{g\psi(g)}\psi^2$ (this follows from Lemma \ref{LemmaCommuteWithInner}). It is not hard to compute that $(\alpha\beta)^2 = (\alpha\gamma)^2 = id$, and $(\beta\gamma)^2 = \varphi_{x^{-1}}$. Thus, $\Fix(\varphi_g\psi) \leq C(g\psi(g))$, if $\psi \in \{\alpha\beta, \alpha\gamma\}$, and $\Fix(\varphi_g\beta\gamma) \leq C(g\beta\gamma(g)x^{-1})$.

    In the case $\psi \in \{\alpha\beta, \alpha\gamma\}$ we calculate, $$\varphi_g\psi(g\psi(g)) = g\psi(g)\psi^2(g)g^{-1} = g\psi(g),$$ so $\langle g\psi(g) \rangle \leq \Fix(\varphi_g\psi)$. Furthermore, $\varphi_g\psi$ restricts to an automorphism of $C(g\psi(g))$. To see this take $h \in C(g\psi(g))$ and notice,
    $$1 = [h, g \psi(g)]  \overset{apply \ (\varphi_g \psi)} = [\varphi_g \psi(h), g \psi(g)]$$ where in the first equality we use that $h$ and $g\psi(g)$ commute, and in the second we use that $g\psi(g) \in \Fix(\varphi_g\psi)$, and that homomorphisms fix the identity. 

    In the case that $\psi = \beta\gamma$, we observe $$\varphi_g\psi(g\psi(g)x^{-1}) = g\psi(g)\psi^2(g)x^{-1}g^{-1} = g\psi(g)x^{-1}gxx^{-1}g^{-1} = g\psi(g)x^{-1},$$ where the second equality is because $\psi^2 = \varphi_{x^{-1}}$. So $\langle g\psi(g)x^{-1} \rangle \leq \Fix(\varphi_g\psi) \leq C(g\psi(g)x^{-1})$, with $\varphi_g\psi$ inducing an automorphism of $C(g\psi(g)x^{-1})$ by the same arugment as in the previous case. 

    Write $h = g\psi(g)$ in the case $\psi \in \{\alpha\beta, \beta\gamma\}$, and $h = g\psi(g)x^{-1}$ in the case $\psi = \beta\gamma$.

    By Lemma \ref{dihedralOrder}, $\varphi_g\psi$ acts on $T$ without fixing a point, and thus so does $(\varphi_g\psi)^2 = \varphi_h$. Since the action is compatible, $h$ also acts hyperbolically, and by \ref{dihedralCentralisers} $C(h) \cong \mathbb{Z}^2$, with $h$ and $x^n$ spanning a finite index subgroup. 

    If $\psi = \beta\gamma$, then $\Fix(\varphi_g\psi) = C(g\psi(g)x^{-1})$. This is because $\varphi_g\psi(x^n) = x^n$, so $\varphi_g\psi$ fixes $\langle g\psi(g)x^{-1} \rangle \times \langle x^n \rangle$, and the only way to extend this to an automorphism of $C(g\psi(g)x^{-1})$ is as the identity (by viewing the automorphisms of $\mathbb{Z}^2$ as elements of $GL_2(\mathbb{Z})$, it is clear that as soon as a rank 2 subgroup of $C(h) \cong \mathbb{Z}^2$ is fixed by an automorphism, that automorphism must fix all of $C(h)$). 

    Otherwise, if $\psi = \alpha\beta$ or $\psi = \alpha\gamma$, $\varphi_g\psi(x^n) = x^{-n}$. This means that $\Fix(\varphi_g\psi)$ is a rank 1 subgroup of $C(g\psi(g))$, so is isomorphic to $\mathbb{Z}$. Since $g\psi(g)$ is fixed, the result follows.
\end{proof}

We will now turn to understanding fixed subgroups of finite order elements of $\Aut_\Gamma(A_{2n})$. By Lemma \ref{dihedralOrder} such automorphisms fix a point. The following lemma allows us to reduce to automorphisms fixing a point near $v_{\langle x \rangle}$.

\begin{lemma}\label{dihedralReduction}
    Suppose $\varphi_g\psi$ is a finite order automorphism of $A_{2n}$, where $\psi \in \{\alpha\beta, \beta\gamma, \alpha\gamma\}$.
    
    Suppose $\varphi_g\psi$ fixes a vertex $hv_{\langle x \rangle}$, where $h \in A_{2n}$. Then there is $k \in \mathbb{Z}$ such that $\varphi_g\psi = \varphi_h\varphi_{x^k}\psi\varphi_h^{-1}$. Furthermore, $\Fix(\varphi_g\psi) = h\Fix(\varphi_{x^k}\psi)h^{-1}$.

    Suppose instead that $\varphi_g\psi$ does not fix any vertex. Then $\Fix(\varphi_g\psi) \leq \langle x^n \rangle$.
\end{lemma}

\begin{proof}
    Suppose that $\varphi_g \psi$ fixes a vertex $h v_{\langle x \rangle}$. Then,
    \begin{align*}
        hv_{\langle x \rangle} &= \varphi_g\psi \cdot hv_{\langle x \rangle}\\
        &= \varphi_g  \cdot \psi(h) (\psi \cdot v_{\langle x \rangle})\\
        &= g \psi(h) v_{\langle x \rangle},\\
    \end{align*}
    where the first equality is by assumption, the second is by Lemma \ref{compatibilityLemma}, and the third is by compatibility and the fact that $\psi \cdot v_{\langle x \rangle} = v_{\langle x \rangle}$.

    Rearranging we see that $h^{-1} g \psi(h) v_{\langle x \rangle} = v_{\langle x \rangle}$, that is $h^{-1} g \psi(h) = x^k$ for some $k \in \mathbb{Z}$. Now it is clear that $\varphi_g\psi = \varphi_{hx^k\psi(h)^{-1}} = \varphi_h\varphi_{x^k}\psi\varphi_h^{-1}$. It follows immediately that $\Fix(\varphi_g\psi) = h\Fix(\varphi_{x^k}\psi)h^{-1}$ by Lemma \ref{lemmaIsogredience}.

    Now instead suppose that $\varphi_g\psi$ does not fix any vertex. Then it must fix the midpoint of a single edge $e$, which is inverted. Since $e$ is inverted, it must be that $T^{\varphi_g\psi}$ is a single point, and $\Fix(\varphi_g\psi) \leq G_e = \langle x^n \rangle$.
\end{proof}

For finite order automorphisms $\psi$ in the classes of $\alpha\beta$ and $\beta\gamma$, $T^\psi$ is a single point.

\begin{lemma}\label{ellipticAlphaBetaBetaGamma}
    Suppose $\varphi_g\alpha\beta$ is a finite order automorphism of $A_{2n}$. Then $\Fix(\varphi_g\alpha\beta) = \{1\}$.

    Suppose $\varphi_g\beta\gamma$ is a finite order automorphism of $A_{2n}$. Then the fixed subgroups of the automorphism are as follows. 

    \begin{enumerate}
        \item If it fixes a vertex $hv_{\langle x \rangle}$ then $\Fix(\varphi_g\beta\gamma) = h\langle x \rangle h^{-1}$.
        \item Otherwise, $\Fix(\varphi_g\beta\gamma) = \langle x^n \rangle$.
    \end{enumerate}
\end{lemma}

\begin{proof}
    Suppose $\varphi_g\psi$ is finite order, where $\psi \in \{\alpha\beta, \beta\gamma\}$. By Lemma \ref{dihedralOrder}, $\varphi_g\psi$ fixes a point of $T$.
    
    Suppose no vertex is fixed, then, by Lemma \ref{dihedralReduction},  $\Fix(\varphi_g\psi) \leq \langle x^n \rangle$. Since $\varphi_g\alpha\beta(x^n) = gx^{-n}g^{-1} = x^{-n}$, where the second equality is because $x^n$ is central, in the $\psi = \alpha\beta$ case $\Fix(\varphi_g\alpha\beta) = \{1\}$ as required. In the $\psi = \beta\gamma$ case, we note $\varphi_g\beta\gamma(x^n) = gx^{-n}g^{-1} = x^n$, so $\Fix(\varphi_g\beta\gamma) = \langle x^n \rangle$.
    
    Now suppose $\varphi_g\psi$ fixes $hv_{\langle x \rangle}$. Then, by Lemma \ref{dihedralReduction}, there is an integer $k$ such that $\Fix(\varphi_g\psi) = h\Fix(\varphi_{x^k}\psi)h^{-1}$. Since $\beta$ inverts the orientation of the edges, and $\alpha$, $\gamma$ and $x^k$ act without inversions, it must be that $v_{\langle x \rangle}$ is the only point of $T$ fixed by $\varphi_{x^k}\psi$ as none of the adjacent edges can be fixed. Therefore, $\Fix(\varphi_{x^k}\psi) \leq \langle x \rangle$.

    It is now easy to check that, in the $\psi = \alpha\beta$ case, $\Fix(\varphi_{x^k}\alpha\beta) = \{1\}$, and so $\Fix(\varphi_g\alpha\beta) = \{1\}$. Similarly in the other case, $\Fix(\varphi_{x^k}\beta\gamma) = \langle x \rangle$, and so $\Fix(\varphi_g\alpha\beta) = h\langle x \rangle h^{-1}$. In both cases for the final equality we use the fact that $\Fix(\varphi_g\psi) = h\Fix(\varphi_{x^k}\psi)h^{-1}$, by Lemma \ref{dihedralReduction}.

\end{proof}

The case of automorphisms $\varphi_g\alpha\gamma$ is different, since $T^{\varphi_g\alpha\gamma}$ may be a bi-infinite line. 

\begin{lemma}
    \label{ellipticAlphaGamma}
    Suppose $\varphi_g\alpha\gamma$ is a finite order automorphism of $A_{2n}$. Then $\Fix(\varphi_g\alpha\gamma) \cong \mathbb{Z}$.

    Specifically, if $g = hx^k(\alpha\gamma)(h)^{-1}$ where $h \in A_{2n}$, $k \in \mathbb{Z}$ (this is always possible by Proposition \ref{dihedralReduction}), then $\Fix(\varphi_g\alpha\gamma)$ is as follows:

    \begin{enumerate}
        \item If $n$ is odd and $k$ is even, $\Fix(\varphi_g\alpha\gamma) = h\langle x^{\frac{k}{2}}tx^{\frac{n-1}{2}}tx^{\frac{-k-n-1}{2}} \rangle h^{-1} \cong \mathbb{Z}$.

        \item If $n$ is odd and $k$ is odd, $\Fix(\varphi_g\alpha\gamma) = h\langle x^{\frac{k+n}{2}}tx^{\frac{n-1}{2}}tx^{\frac{-k-1}{2}} \rangle h^{-1} \cong \mathbb{Z}$.

        \item If $n$ is even and $k$ is even, $\Fix(\varphi_g\alpha\gamma) = h\langle x^{\frac{k}{2}}tx^{\frac{n}{2}}t^{-1}x^{\frac{-k-n}{2}} \rangle h^{-1} \cong \mathbb{Z}$.

        \item If $n$ is even and $k$ is odd, $\Fix(\varphi_g\alpha\gamma) = h\langle x^{\frac{k+1}{2}}t^{-1}x^{\frac{n}{2}}tx^{\frac{-k-n-1}{2}} \rangle h^{-1} \cong \mathbb{Z}$.
    \end{enumerate}
\end{lemma}
\begin{proof}
   
    First notice that, by Remark \ref{dihedralAutOrientation}, $\varphi_g\alpha\gamma$ respects the orientation of the edges, so acts without inversions. Therefore $\varphi_g\alpha\gamma$ must fix a vertex $hv_{\langle x \rangle}$ given that it fixes a point.

    As in Lemma \ref{dihedralReduction}, $\varphi_g\alpha\gamma = \varphi_h\varphi_{x^k}\alpha\gamma\varphi_h^{-1}$, and it is enough to consider automorphisms of the form $\varphi_{x^k}\alpha\gamma$ for $k \in \mathbb{Z}$. Notice that $v_{\langle x \rangle} \in T^{\varphi_{x^k}\alpha\gamma}$, because each of $\gamma$, $\alpha$ and $\varphi_{x^k}$ fixes $v_{\langle x \rangle}$.

    In order to understand $T^{\varphi_{x^k}\alpha\gamma}$, we will restrict our attention to two vertices in $T^{\varphi_{x^k}\alpha\gamma}$ ($v_{\langle x \rangle}$ and one of its neighbours), then argue these are the only two $\Fix(\varphi_g\alpha\gamma)$ orbits of vertices in $T^{\varphi_{x^k}\alpha\gamma}$. From here we will reconstruct all of $T^{\varphi_{x^k}\alpha\gamma}$. We separate into two cases based on the parity of $n$.

    \bigskip
    \noindent \underline{Case 1:} $n$ is odd.

    We begin by finding an element $s \in \Fix(\varphi_{x^k}\alpha\gamma)$ acting hyperbolically on $T$.

    If $k$ is even, then $s \coloneqq x^{\frac{k}{2}}tx^{\frac{n-1}{2}}tx^{\frac{-k-n-1}{2}} \in \Fix(\varphi_{x^k}\alpha\gamma)$. Indeed, \begin{align*}
        (\varphi_{x^k}\alpha\gamma)(x^{\frac{k}{2}}tx^{\frac{n-1}{2}}tx^{\frac{-k-n-1}{2}}) &= x^k x^{-\frac{k}{2}}tx^{-\frac{n-1}{2}-1}tx^{\frac{k+n+1}{2}-1} x^{-k}\\
        &= x^{\frac{k}{2}}tx^{\frac{-n-1}{2}}tx^{\frac{-k+n-1}{2}}\\
        &= x^{\frac{k}{2}}tx^{\frac{n-1}{2}}x^{-n}tx^nx^{\frac{-k-n-1}{2}}\\
        &= x^{\frac{k}{2}}tx^{\frac{n-1}{2}}tx^{\frac{-k-n-1}{2}},\\
    \end{align*} as required. If $k$ is odd, then $s \coloneqq x^{\frac{k+n}{2}}tx^{\frac{n-1}{2}}tx^{\frac{-k-1}{2}} \in \Fix(\varphi_{x^k}\alpha\gamma)$, by a similar calculation.

    In either case, $s$ acts along an axis $T^s$ through $v_{\langle x \rangle}$, which must be a subtree of $T^{\varphi_{x^k}\alpha\gamma}$ by Lemma \ref{fixAction}. Since $s$ acts with translation length 2, $T^s$ has 2 orbits of vertices under the action of $s$. We now use argue that locally to these two orbits of vertices, $T^{\varphi_{x^k}\alpha\gamma}$ coincides exactly with $T^s$. Since $T^{\varphi_{x^k}\alpha\gamma}$ is connected, this implies $T^{\varphi_{x^k}\alpha\gamma} = T^s$.

    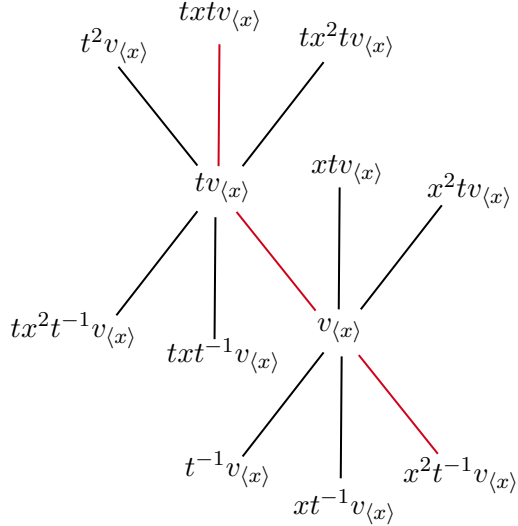
\begin{figure}[h]
    \centering

\tikzset{every picture/.style={line width=0.75pt}} 

\begin{tikzpicture}[x=0.75pt,y=0.75pt,yscale=-1,xscale=1]

\draw [color={rgb, 255:red, 208; green, 2; blue, 27 }  ,draw opacity=1 ]   (265.33,126.33) -- (304.78,176) ;
\draw    (316.78,114) -- (316.33,175.33) ;
\draw    (327.33,175.33) -- (367.78,123) ;
\draw    (205.33,178.33) -- (245.78,126) ;
\draw    (254.78,129) -- (254.33,190.33) ;
\draw    (268.33,249.33) -- (308.78,197) ;
\draw [color={rgb, 255:red, 208; green, 2; blue, 27 }  ,draw opacity=1 ]   (326.33,198.33) -- (365.78,248) ;
\draw    (317.78,199) -- (317.33,260.33) ;
\draw    (268.33,103.33) -- (308.78,51) ;
\draw    (206.33,53.33) -- (245.78,103) ;
\draw [color={rgb, 255:red, 208; green, 2; blue, 27 }  ,draw opacity=1 ]   (256.78,42) -- (256.33,103.33) ;

\draw (304,178.4) node [anchor=north west][inner sep=0.75pt]    {$v_{\langle x\rangle }$};
\draw (243,105.4) node [anchor=north west][inner sep=0.75pt]    {$tv_{\langle x\rangle }$};
\draw (301,96.4) node [anchor=north west][inner sep=0.75pt]    {$xtv_{\langle x\rangle }$};
\draw (359,105.4) node [anchor=north west][inner sep=0.75pt]    {$x^{2} tv_{\langle x\rangle }$};
\draw (229,188.4) node [anchor=north west][inner sep=0.75pt]    {$txt^{-1} v_{\langle x\rangle }$};
\draw (151,175.4) node [anchor=north west][inner sep=0.75pt]    {$tx^{2} t^{-1} v_{\langle x\rangle }$};
\draw (238,245.4) node [anchor=north west][inner sep=0.75pt]    {$t^{-1} v_{\langle x\rangle }$};
\draw (292,264.4) node [anchor=north west][inner sep=0.75pt]    {$xt^{-1} v_{\langle x\rangle }$};
\draw (347,248.4) node [anchor=north west][inner sep=0.75pt]    {$x^{2} t^{-1} v_{\langle x\rangle }$};
\draw (186,32.4) node [anchor=north west][inner sep=0.75pt]    {$t^{2} v_{\langle x\rangle }$};
\draw (235,19.4) node [anchor=north west][inner sep=0.75pt]    {$txt v_{\langle x\rangle }$};
\draw (295,29.4) node [anchor=north west][inner sep=0.75pt]    {$tx^{2} t v_{\langle x\rangle }$};

\end{tikzpicture}

    \caption{\label{diagramOfActionOdd} A fragment of the Bass-Serre tree for $A_6$. To find $T^{\alpha\gamma}$ (i.e., in the case $k = 0$, when $n = 3$), we identify $txtx^{-2}$ as a fixed element of $\alpha\gamma$. We notice that $txtx^{-2}$ acts along the red axis, then argue locally to $\langle x \rangle$ and $t \langle x \rangle$ to show that this axis is all of $T^{\alpha\gamma}$.}
    \end{figure}

    First, we study the action of $\varphi_{x^k}\alpha\gamma$ on the 1-neighbourhood of $v_{\langle x \rangle}$.

    We consider the vertices $x^itv_{\langle x \rangle}$ (for $0 \leq i < n$), which should be thought of as those directly "above" $v_{\langle x \rangle}$ in the tree (this is reflected in the figures, where vertical direction is determined by the orientation from Remark \ref{dihedralEvenOrientation}).

    Notice that $\varphi_{x^k}\alpha\gamma \cdot x^itv_{\langle x \rangle} = x^{k-i}tv_{\langle x \rangle}$. So $x^itv_{\langle x \rangle} \in T^{\varphi_{x^k}\alpha\gamma}$ if and only if $k = 2i$ (mod n). Since $n$ is odd this happens for exactly one vertex $x^itv_{\langle x \rangle}$.

    Now we consider the vertices $x^it^{-1}v_{\langle x \rangle}$ (for $0 \leq i < n$), that is those which are directly below $v_{\langle x \rangle}$ in the tree.  Notice that $\varphi_{x^k}\alpha\gamma \cdot x^it^{-1}v_{\langle x \rangle} = x^{k-i-1}t^{-1}v_{\langle x \rangle}$. So $x^it^{-1}v_{\langle x \rangle} \in T^{\varphi_{x^k}\alpha\gamma}$ if and only if $k-i+1 = i$ (mod $n$), that is if and only $k+1 = 2i$ (mod $n$). Since $n$ is odd this happens for exactly one vertex $x^it^{-1}v_{\langle x \rangle}$. 

    Second, we study the action of $\varphi_{x^k}\alpha\gamma$ on the 1-neighbourhood of $x^itv_{\langle x \rangle}$, where $i$ is fixed such that $2i = k$ (mod $n$). 

    We consider the vertices $x^itx^jtv_{\langle x \rangle}$ (for $0 \leq j < n$), that is those which are directly above $x^itv_{\langle x \rangle}$ in the tree. Notice that $\varphi_{x^k}\alpha\gamma \cdot x^itx^jtv_{\langle x \rangle} = x^{k-i}tx^{-(j+1)}tv_{\langle x \rangle}$. So $x^itx^jtv_{\langle x \rangle} \in T^{\varphi_{x^k}\alpha\gamma}$ if and only if $-j-1=j$ (mod $n$), that is $2j=n-1$ (mod $n$). Since $n$ is odd this happens for exactly one vertex $x^itx^jtv_{\langle x \rangle}$.

    Now we consider the vertices $x^itx^jt^{-1}v_{\langle x \rangle}$ (for $0 \leq j < n$), that is those which are directly below $x^itv_{\langle x \rangle}$ in the tree.  Notice that $\varphi_{x^k}\alpha\gamma \cdot x^itx^jt^{-1}v_{\langle x \rangle} = x^{k-i}tx^{-j}t^{-1}v_{\langle x \rangle}$. So $x^itx^jt^{-1}v_{\langle x \rangle} \in T^{\varphi_{x^k}\alpha\gamma}$ if and only if $-j = j$ (mod $n$), that is if and only $j = 0$ (mod $n$). That is, $v_{\langle x \rangle}$ was the only vertex below $x^itv_{\langle x \rangle}$ in $T^{\varphi_{x^k}\alpha\gamma}$.

    Since $v_{\langle x \rangle}$ and $x^itv_{\langle x \rangle}$ are representatives of the two orbits of vertices in $T^s$ under the action of $s$, and $s$ acts on $T^{\varphi_{x^k}\alpha\gamma}$ by Lemma \ref{fixAction}, the arguments above suffice to show $T^{\varphi_{x^k}\alpha\gamma} = T^s$.

    So $\Fix(\varphi_{x^k}\alpha\gamma) \leq \Stab(T^s) \cong \mathbb{Z}^2$. Since $(\varphi_{x^k}\alpha\gamma)(x^n) = x^{-n}$, $\Fix(\varphi_{x^k}\alpha\gamma) \cap \langle x^n \rangle = \{1\}$, so $\Fix(\varphi_{x^k}\alpha\gamma) \cong \mathbb{Z}$. Since $\langle s \rangle$ is the maximal cyclic subgroup containing $s$ (that is, $s$ is not a proper power), it must be that $\Fix(\varphi_{x^k}\alpha\gamma) = \langle s \rangle$. 

    \bigskip
    \noindent \underline{Case 2:} $n$ is even.    

    Again, we begin by finding an element $s \in \Fix(\varphi_{x^k}\alpha\gamma)$ acting hyperbolically on $T$.

    If $k$ is even, then $s \coloneqq x^{\frac{k}{2}}tx^{\frac{n}{2}}t^{-1}x^{\frac{-k-n}{2}} \in \Fix(\varphi_{x^k}\alpha\gamma)$. Indeed, \begin{align*}
        (\varphi_{x^k}\alpha\gamma)(x^{\frac{k}{2}}tx^{\frac{n}{2}}t^{-1}x^{\frac{-k-n}{2}}) &= x^k x^{-\frac{k}{2}}tx^{-\frac{n}{2}}t^{-1}x^{\frac{k+n}{2}} x^{-k}\\
        &= x^{\frac{k}{2}}tx^{\frac{n}{2}}x^{-n}t^{-1}x^nx^{\frac{-k-n}{2}}\\
        &= x^{\frac{k}{2}}tx^{\frac{n}{2}}t^{-1}x^{\frac{-k-n}{2}}\\
    \end{align*} as required. If $k$ is odd, then $s \coloneqq x^{\frac{k+1}{2}}t^{-1}x^{\frac{n}{2}}tx^{\frac{-k-n-1}{2}} \in \Fix(\varphi_{x^k}\alpha\gamma)$, by a similar calculation.

    As before, we see $T^s \subseteq T^{\varphi_{x^k}\alpha\gamma}$ and will argue they are equal by arguing locally to vertices representing the two orbits under $s$ acting along $T^s$. We will work in the case where $k$ is even -- if $k$ is odd the result follows similarly.

    \begin{figure}[h]
    \centering
    \begin{tikzpicture} [%
        nd/.style = {circle,fill=white,text=black,inner sep=0pt},
        tn/.style = {node distance=1pt},
        ed/.style={-,black,fill=none},
        ax/.style={-,red,fill=none}]
    
        \node[nd] (x) at (0,0) {\sffamily $v_{\langle x \rangle}$};
        
        \node[nd] (tx) [above left=of x] {\sffamily $t v_{\langle x \rangle}$};
        \node[nd] (xtx) [above right=of x] {\sffamily $xt v_{\langle x \rangle}$};
    
        \node[nd] (t^2x) [above left=of tx] {\sffamily $t^2 v_{\langle x \rangle}$};
        \node[nd] (txtx) [above right=of tx] {\sffamily $txt v_{\langle x \rangle}$};

        \node[nd] (txt^-1x) [below left=of tx] {\sffamily $txt^{-1} v_{\langle x \rangle}$};
    
        \node[nd] (t^-1x) [below left=of x] {\sffamily $t^{-1} v_{\langle x \rangle}$};
        \node[nd] (xt^-1x) [below right=of x] {\sffamily $xt^{-1} v_{\langle x \rangle}$};

        \draw[ax] (x) -- (tx);
        \draw[ax] (x) -- (xtx);
    
        \draw[ed] (tx) -- (t^2x);
        \draw[ed] (tx) -- (txtx);
    
        \draw[ax] (tx) -- (txt^-1x);
    
        \draw[ed] (x) -- (t^-1x);
        \draw[ed] (x) -- (xt^-1x);
        
    \end{tikzpicture}
    \caption{\label{diagramOfActionEven} A fragment of the Bass-Serre tree for $A_4$. Again, the example is for the automorphism $\alpha\gamma$, so $k=0$ and $n=2$. We use the same strategy as before, but this time $txtx^{-2}$ is our principal fixed element of $\alpha\gamma$. Notice in the even case the action is “horizontal” instead of “vertical”.}
    \end{figure}
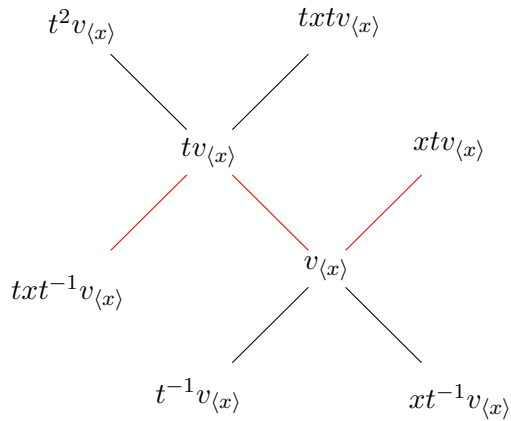

    First, we study the action of $\varphi_{x^k}\alpha\gamma$ on the 1-neighbourhood of $v_{\langle x \rangle}$.

    We consider the vertices $x^itv_{\langle x \rangle}$ (for $0 \leq i < n$), that is those which are directly above $v_{\langle x \rangle}$ in the tree. Notice that $\varphi_{x^k}\alpha\gamma \cdot x^itv_{\langle x \rangle} = x^{k-i}tv_{\langle x \rangle}$. So $x^itv_{\langle x \rangle} \in T^{\varphi_{x^k}\alpha\gamma}$ if and only if $k = 2i$ (mod n). Since $n$ and $k$ are both even, this happens for exactly two vertices $x^itv_{\langle x \rangle}$ (corresponding to $i = \frac{k}{2}$ and $i = \frac{k+n}{2}$).

    Now we consider the vertices $x^it^{-1}v_{\langle x \rangle}$ (for $0 \leq i < n$), that is those which are directly below $v_{\langle x \rangle}$ in the tree.  Notice that $\varphi_{x^k}\alpha\gamma \cdot x^it^{-1}v_{\langle x \rangle} = x^{k-i-1}t^{-1}v_{\langle x \rangle}$. So $x^it^{-1}v_{\langle x \rangle} \in T^{\varphi_{x^k}\alpha\gamma}$ if and only if $k-i+1 = i$ (mod $n$), that is if and only $k+1 = 2i$ (mod $n$). Since $n$ and $k$ are both even, this does not occur. 

    Second, we study the action of $\varphi_{x^k}\alpha\gamma$ on the 1-neighbourhood of $x^itv_{\langle x \rangle}$, where $i$ is fixed such that $2i = k$ (mod $n$). 

    We consider the vertices $x^itx^jtv_{\langle x \rangle}$ (for $0 \leq j < n$), that is those which are directly above $x^itv_{\langle x \rangle}$ in the tree. Notice that $\varphi_{x^k}\alpha\gamma \cdot x^itx^jtv_{\langle x \rangle} = x^{k-i}tx^{-(j+1)}tv_{\langle x \rangle}$. So $x^itx^jtv_{\langle x \rangle} \in T^{\varphi_{x^k}\alpha\gamma}$ if and only if $-j-1=j$ (mod $n$), that is $2j=n-1$ (mod $n$). Since $n$ is even, this does not occur.

    Now we consider the vertices $x^itx^jt^{-1}v_{\langle x \rangle}$ (for $0 \leq j < n$), that is those which are directly below $x^itv_{\langle x \rangle}$ in the tree.  Notice that $\varphi_{x^k}\alpha\gamma \cdot x^itx^jt^{-1}v_{\langle x \rangle} = x^{k-i}tx^{-j}t^{-1}v_{\langle x \rangle}$. So $x^itx^jt^{-1}v_{\langle x \rangle} \in T^{\varphi_{x^k}\alpha\gamma}$ if and only if $-j = j$ (mod $n$), that is if and only $j = 0$ (mod $n$). There are two solutions: $j=0$, which recovers $v_{\langle x \rangle} \in T^{\varphi_{x^k}\alpha\gamma}$, and $j = \frac{n}{2}$.

    As before, these local arguments suffice to show $T^s = T^{\varphi_{x^k}\alpha\gamma}$, that is $\Fix(\varphi_g\alpha\gamma) \leq \Stab(T^s) \cong \mathbb{Z}^2$. As in the previous case, $\Fix(\varphi_{x^k}\alpha\gamma) = \langle s \rangle$.
    
\end{proof}

We now put our results together, translating back to the Artin presentation, to completely classify the fixed subgroups of inducible automorphisms of even dihedral Artin groups.

\begin{proof}[Proof of Theorem \ref{dihedralFixPointsEven}]
        Several points follow immediately from our work up to this point. Points 1 and 2 follow from Proposition \ref{dihedralCentralisers}; point 3 from Proposition \ref{ellipticAlphaBetaBetaGamma}; point 4, 6 and 8 from \ref{dihedralInfiniteOrder}.

        The first part of point 5 follows from Lemma \ref{ellipticAlphaBetaBetaGamma}. 

        For the second part, note that by the proof of Lemma \ref{inducibleBSAutomorphisms}, we can see $\sigma$ as the automorphism $\varphi_t\beta\gamma$ in the Baumslag-Solitar presentation. Direct computation reveals this automorphism fixes only the midpoint of the edge from $v_{\langle x \rangle}$ to $tv_{\langle x \rangle}$, so $\Fix(\sigma)$ is contained in the centre $\langle (ab)^n \rangle$, and since $\sigma((ab)^n) = (ba)^n = (ab)^n$, $\Fix(\sigma) = \langle (ab)^n \rangle$.

        The first part of point 8 follows from Lemma \ref{ellipticAlphaGamma}. For the second part, we restrict to the case where $n$ is even. The odd case is similar. Notice that, \begin{align*}
            \sigma\iota((ab)^{\frac{n}{2}}(ba)^{-\frac{n}{2}}) &= (ab)^{-\frac{n}{2}}(ba)^{\frac{n}{2}}\\
            &= (ab)^{-\frac{n}{2}}(ab)^n(ba)^{-n}(ba)^{\frac{n}{2}}\\
            &= (ab)^{\frac{n}{2}}(ba)^{-\frac{n}{2}},
        \end{align*} so the proposed group is certainly a subgroup of $\Fix(\sigma\iota)$.
        
        We show the element $(ab)^{\frac{n}{2}}(ba)^{-\frac{n}{2}}$ has no roots. Direct computation reveals it acts on the Bass-Serre tree with translation length 2, so any root would act with translation length 1. It is height 0 (see Definition \ref{definitionHeight}), so any root would be height 0. No element with height 0 can act with translation length 1, as the element acting with translation length 1 are conjugates of $b(ab)^k$ (or its inverse) where $k \in \mathbb{Z}$, therefore no root can exist.
        
        Since $\Fix(\sigma\iota) \cong \mathbb{Z}$, it must be that $\Fix(\sigma\iota) = \langle (ab)^{\frac{n}{2}}(ba)^{-\frac{n}{2}} \rangle$ as required.
\end{proof}

\section{Fixed Subgroups of Large-Type Artin Groups}\label{sectionGeneral}

In this section we come back to the primary case of an Artin group $A_{\Gamma}$ of rank at least $3$. We will suppose throughout the section that $A_{\Gamma}$ is large-type, or at least $2$-dimensional.

Recall that that $\Aut_{\Gamma}(A_{\Gamma})$ acts on $X_{\Gamma}$ by simplicial automorphisms (see Definition \ref{DefinitionActionOfAut}). The goal of this section is to compute the fixed subgroups associated with every automorphism $\gamma \in \Aut_{\Gamma}(A_{\Gamma})$. When $\gamma$ is inner, this comes down to understanding centralisers of elements in $A_{\Gamma}$. This is dealt with in Section 4.1.

When $\gamma$ is not inner, we will generally have two methods to compute its fixed subgroups. In Section 4.2 we suppose that $\gamma$ acts elliptically on $X_{\Gamma}$. In particular, it fixes some (convex) subcomplex $X_{\Gamma}^{\gamma}$, and we know by Corollary \ref{fixAction} that $\Fix(\gamma) \leq \Stab(X_{\Gamma}^{\gamma})$. Most of the work in order to understand $\Fix(\gamma)$ will then reduce to understanding $X_{\Gamma}^{\gamma}$. In Section 4.3 we deal with the case where $\gamma$ acts hyperbolically on $X_{\Gamma}$. In that case, we will show that there always exists some element $h \in A_{\Gamma}$ such that the fixed subgroup $\Fix(\gamma)$ lies within the centraliser $C(h)$ (see Lemma \ref{infiniteOrderIntoCentraliser}). This will simplify the search for fixed points.

\begin{notation} \label{NotationAutomorphisms}
    Throughout this section, and unless mentioned otherwise, we will use $\gamma$ to describe an arbitrary element of $\Aut_{\Gamma}(A_{\Gamma})$. As $\Inn(A_{\Gamma})$ is normal in $\Aut_{\Gamma}(A_{\Gamma})$, every automorphism $\gamma$ decomposes as a product $\gamma = \varphi_g \sigma \iota^{\varepsilon}$ for some $g \in A_{\Gamma}$ and $\varepsilon \in \{0, 1\}$, where $\varphi_g$ is inner, $\sigma$ is induced by a (possibly trivial) graph automorphism, and $\iota$ is the global inversion. Finally, we will use $\psi$ to refer to an automorphism that lies in the subgroup generated by the graph automorphisms and the global inversion, i.e. $\psi = \sigma \iota^{\varepsilon}$.
\end{notation}

\subsection{Inner automorphisms}

The fixed subgroup $\Fix(\varphi_g)$ of an inner automorphism is the centraliser $C(g)$. Thus to understand the fixed subgroups of inner automorphisms we need to understand centralisers of elements in $A_{\Gamma}$. The following proposition gives the desired result:

\begin{proposition} \cite[Remark 3.6]{martin2023characterising} \label{PropCentralisers} Let $A_{\Gamma}$ be a large-type Artin group, and let $g \in A_{\Gamma} \backslash \{1\}$. Then up to conjugation, the centraliser $C(g)$ can be described as follows:
\begin{enumerate}
\item If $g$ acts elliptically on $X_{\Gamma}$ and $X_{\Gamma}^g$ is the standard tree $X_{\Gamma}^a$ for some $a \in V(\Gamma)$, then $C(g) = \langle a \rangle \times F \cong \mathbb{Z} \times F$, where $F$ is a (possibly abelian) finitely generated free group.
    
\item If $g$ acts elliptically on $X_{\Gamma}$ and $X_{\Gamma}^g$ is the single vertex $v_{ab}$, then:
    \begin{enumerate}
        \item If $g \in Z(A_{ab})$, then $C(g)$ is the dihedral Artin subgroup $A_{ab}$.
        \item If $g \notin Z(A_{ab})$ but $g^n \in Z(A_{ab})$ for some $n \neq 0$, then $C(g) \cong \mathbb{Z}$ is the maximal $\mathbb{Z}$-subgroup that contains $g$.
        \item Otherwise, $C(g) \cong \mathbb{Z}^2$ with $\langle g \rangle \times Z(A_{ab}) \leq_{f.i.} C(g)$.
        
    \end{enumerate}
\item If $g$ acts hyperbolically on $X_{\Gamma}$, we consider the transverse-tree $\mathcal{T}$ associated with $g$ (see Lemma \ref{LemmaTransverseTree}). Then:
    \begin{enumerate}
        \item If $\mathcal{T}$ is bounded, and some axis of $g$ is contained in a standard tree of the form $X_{\Gamma}^a$, then $C(g) = \langle g_0, a \rangle \cong \mathbb{Z}^2$, where $g_0$ acts with minimal non-trivial translation length along the axes of $g$.
        \item If $\mathcal{T}$ is bounded, and no axis of $g$ is contained in a standard tree of $X_{\Gamma}$, then $C(g) = \langle g_0 \rangle \cong \mathbb{Z}$, where $g_0$ acts with minimal non-trivial translation length along the axes of $g$.
        \item If $\mathcal{T}$ is unbounded and $g$ has an axis that is contained in a standard tree of the form $X_{\Gamma}^b$, then $g = (abcabc)^k$ for some $k \neq 0$ and $C(g) = \langle b, abc \rangle$, where $a, b, c$ are three standard generators with coefficients $m_{ab} = m_{ac} = m_{bc} = 3$. In that case, the centraliser $C(g)$ is abstractly isomorphic to the dihedral Artin group $\langle x, y \ | \ xyxy = yxyx \rangle$.
        \item If $\mathcal{T}$ is unbounded and $g$ has no axis that is contained in a standard tree of $X_{\Gamma}$, then $C(g) \cong \mathbb{Z}^2$ with $\langle g, abcabc \rangle \leq_{f.i.} C(g)$, where $a, b, c$ are three standard generators with coefficients $m_{ab} = m_{ac} = m_{bc} = 3$.
    \end{enumerate}
\end{enumerate}
\end{proposition}

\begin{remark}
    In Proposition \ref{PropCentralisers}.1, a basis of the free subgroup $F$ can be given explicitly. See Remark \ref{RemFreeGroupExplicitly}, or see \cite[Remark 4.6]{martin2022acylindrical} or \cite[Corollary 34]{cumplido2023parabolic}.
\end{remark}

\begin{proof}[Proof of Proposition \ref{PropCentralisers}]
    This Proposition was proved in \cite{martin2023characterising}, although some of the centralisers, namely the ones of points 2.2, 2.3 and 3.4, were only showed to be virtually abelian. We remove the “virtual” condition thereafter:
    \bigskip

    \noindent \underline{Cases 2.2. and 2.3.} First of all, note that $C(g) \subseteq A_{ab}$ (apply Lemma \ref{reductionToStabs} to $\varphi_g$).
    
    Thus the argument can be reduced to an argument on dihedral Artin groups, which was proved in Section 3 (see Theorem \ref{dihedralFixPointsOdd}.1-2.), and Theorem \ref{dihedralFixPointsEven}.1-2.). Notice that the arguments in Section 3 also justify that the maximal $\mathbb{Z}$-subgroup containing $g$ in Case 2.2 is well-defined, because this case is exactly when $g$ is in the stabiliser of a single vertex but no edge, so that vertex stabiliser is the appropriate copy of $\mathbb{Z}$.
    \medskip

    \noindent \underline{Case 3.4.} By Lemma \ref{LemmaPossibleMinsets}, we know that $X_\Gamma^g$ is a Euclidean plane $P$ tiled with principal triangles, such as describe in Figure \ref{FigureTiledPlane}. We also know that $C(g)$ is contained in $\Isom(P)$, which is virtually $\mathbb{Z}^2$. We make the following general observation:
    \medskip

    \noindent \underline{Claim:} For any Euclidean plane $P_0$ in $X_{\Gamma}$, there is no non-trivial elliptic element in $\Isom(P_0)$.
    \medskip

    \noindent \underline{Proof of the claim:} Suppose that there is some non-trivial elliptic element $g \in \Isom(P_0)$. Then $g$ fixes some vertex $v \in P_0$. Because $P_0$ is locally finite, it must be that some power $g^n$ with $n \neq 0$ fixes the star of $v$ pointwise. This star contains points with trivial stabiliser, so the above gives $g^n = 1$, which contradicts $A_{\Gamma}$ being torsion-free (Theorem \ref{ThmTorsionAndCentres}). This proves the claim.
    \medskip

    It follows that $\Isom(P)$ can only contain the identity, glide reflections, and pure translations. Consider now a glide reflection $h \in C(g) \subseteq \Isom(P)$ with axis $u$. Since the action of $h$ on $P$ commutes with that of $h$ which is a pure translation, their axes must be parallel, i.e. $u$ is an axis of $g$ too. Let us consider the vertical purple axis $\gamma$ from Figure \ref{FigureTiledPlane}. We split in three cases:
    \begin{enumerate}
        \item The angle $\angle (\gamma, u)$ is one of $0$, $\frac{\pi}{3}$,  $\frac{2 \pi}{3}$ or $\pi$ (the purple axes) ;
        \item The angle $\angle (\gamma, u)$ is one of $\frac{\pi}{6}$, $\frac{\pi}{2}$ or $\frac{5 \pi}{6}$ (the brown axes) ;
        \item None of the above (eg. the orange axis).
    \end{enumerate}
    \begin{figure}[H]
    \centering
    \includegraphics[scale=1.1]{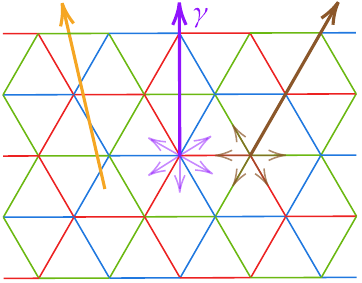}
    \caption{A partial picture of $P = X_\Gamma^g$. The possible angles that line can make with $\gamma$ are drawn in different colours. The edges of the principal triangles are given a system of colours (see Remark \ref{RemSystemColours}).}
    \label{FigureTiledPlane}
    \end{figure}

    \noindent The Case (3.) immediately gives a contradiction, as one cannot reflect along such an axis while preserving the tiling of $P$. For Case (2.), we know by hypothesis that $g$ has no axis that is contained in a standard tree. One can easily compute that if the axis $u$ of $g$ falls in Case (2.) above then $u$ is actually contained in a standard tree. This also excludes Case (2.).

    We now suppose that we are in Case (1.). The glide reflection $h$ must respect the colouring of the tiling of $P$, as it is an isometry that is induced by the group action (see Remark \ref{RemSystemColours}). One easily checks that any reflection along a line of Case (1.) does not preserve the colouring. For instance, if the line under consideration is the line $\gamma$ on Figure \ref{FigureTiledPlane}, then a glide reflection along $\gamma$ sends the pair of blue-red (from left to right) horizontal edges onto a pair of red-blue (from left to right) edges, which do not exist in $P$. This also excludes Case (1.).

    It follows that no glide reflection exists in $C(g)$, i.e. $C(g)$ only consists of pure translations. Together with the fact that $C(g)$ is virtually isomorphic to $\mathbb{Z}^2$, the above shows that $C(g)$ is actually isomorphic to $\mathbb{Z}^2$. The translations induced by $g$ and by $abcabc$ are always along different directions, as the latter induces a translation for which certain axes are contained in standard trees, while the former does not. In particular, $\langle g, abcabc \rangle \cong \mathbb{Z}^2$, and thus $\langle g, abcabc \rangle \leq_{f.i.} C(g)$.
\end{proof}

\subsection{Automorphisms acting elliptically}

In this section we compute the fixed subgroups of the (non-inner) automorphisms of $A_{\Gamma}$ that act elliptically on $X_{\Gamma}$. Let $\gamma$ be such an automorphism. Section 4.2.1 is dedicated to introducing various general tools. In Section 4.2.2 we compute the fixed subgroup of $\gamma$ when $\gamma = \varphi_g \sigma$, i.e. when there is no global inversion involved. In Section 4.2.3 we do the same thing when $\gamma = \varphi_g \sigma \iota$, i.e. when the global inversion is involved.

\begin{theorem}\label{mainTheoremElliptic}
    Let $A_\Gamma$ be a large-type Artin group of rank at least 3. Suppose $\gamma \in \Aut_\Gamma(A_\Gamma)$ is an automorphism acting elliptically on $X_\Gamma$. Then we are in one of the following cases. In the following, $h \in A_{\Gamma}$, $\sigma$ is a (possibly trivial) graph automorphism, and $a$ is a generator such that $\sigma(a) = a$.

    \begin{enumerate}
        \item $\gamma \sim_h \sigma$ and $\Fix(\gamma)$ is is the free product of a parabolic subgroup and a finitely generated free group. In particular, $$\Fix(\gamma) = h ( A_{\Gamma^\sigma}*\langle \{ \Delta_{st} \ | \ \sigma \text{ transposes $s$ and $t$}\}\rangle ) h^{-1},$$ where $\Gamma^\sigma$ is the subgraph of $\Gamma$ fixed by $\sigma$.

        \item $\gamma \sim_h \sigma\iota$ and $\Fix(\gamma)$ is a finitely generated free group. In particular, $$\Fix(\gamma) = h\langle \{w_{st} \ | \ \sigma \text{ transposes $s$ and $t$, and $m_{st}$ is even}\}\rangle h^{-1},$$ where,
        
        \begin{equation*}
            w_{st} =
            \begin{cases}
                (st)^n(ts)^{-n}, & \text{if}\ m_{st}=4n \\
                t(st)^n(ts)^{-n}s^{-1}, & \text{if}\ m_{st}=4n+2
            \end{cases}
        \end{equation*}

        \item $\gamma \sim_h \varphi_{a^k}\sigma$ and $\Fix(\varphi)$ is a direct product of $\mathbb{Z}$ and a finitely generated free group. In particular, $$\Fix(\gamma) = h(\langle a \rangle \times F)h^{-1},$$ where $F$ is the free group described explicitly in Remark \ref{RemFreeGroupExplicitly}.

        \item $\gamma \sim_h \varphi_a\sigma\iota$, and $\Fix(\varphi)$ is a finitely generated free group. In particular, $$\Fix(\gamma) = hFh^{-1},$$ where $F$ is the free group described explicitly in Remark \ref{iotaagenerators}.
        
        \item None of the previous cases apply, and $\gamma \sim_h \varphi_g\sigma\iota^{\varepsilon}$, where $g \in A_{st}$ for some standard dihedral parabolic subgroup $A_{st}$ and $\varepsilon \in \{0,1\}$. The automorphism $\varphi_{h^{-1}}\gamma\varphi_h$ restricts to an automorphism of $A_{st}$ and its fixed subgroup is isomorphic to $\{1\}$, $\mathbb{Z}$ or $\mathbb{Z}^2$, as outlined by Theorem \ref{dihedralFixPointsOdd} and Theorem \ref{dihedralFixPointsEven}.
    \end{enumerate}
\end{theorem}

\begin{proof}
    The reduction to the five cases follows from Lemma \ref{LemmaReduction}. Then, Case 1 follows from Corollary \ref{CoroFixedPointsOfSigma}; Case 2 from Corollary \ref{iotaSigmaFixedPoints}; Case 3 from Corollary \ref{CoroStabSigmaa}; Case 4 from Corollary \ref{coroFixSigmaIotaa}; and Case 5 from Lemma \ref{reductionToStabs}.
\end{proof}

As outlined in the introduction, in the below (and throughout this paper) we always use rank in the Artin sense, that is, the fewest number of vertices in any defining graph.

\begin{proposition}\label{propRankElliptic}
    The rank of $\Fix(\gamma)$ for the automorphism $\gamma$ of Theorem \ref{mainTheoremElliptic} is at most $n^2 - 2n + 2$, where $n = |V(\Gamma)|$. Moreover, there exist Artin groups $A_\Gamma$ and automorphisms $\gamma$ realising this upper bound for each $n$.
\end{proposition}

\begin{proof}
    Notice that a simplicial graph on $n$ vertices has at most $\frac{n(n-1)}{2}$ edges. Almost all cases follow immediately from this observation, and the fact that the rank of a parabolic subgroup $A_{\Gamma'}$ is $|V(\Gamma')|$. The only difficult cases are for automorphisms isogredient to $\varphi_{a^k}\sigma$ or $\varphi_{a}\sigma\iota$, falling in cases 3 or 4 of the theorem.
    
    First we consider automorphisms falling into case 3. By Remark \ref{RemFreeGroupExplicitly}, the free group $F$ has at most:
    \medskip
    
    \noindent $\bullet$ one generator $a$ of the $\mathbb{Z}$-factor ;
    \medskip
    
    \noindent $\bullet$ one generator for each edge of $\Gamma$, of which there are at most $\frac{n(n-1)}{2}$ ;
    \medskip
    
    \noindent $\bullet$ a free basis of $\pi_1(\Gamma)$, which is of rank $|E(\Gamma)| - |V(\Gamma)| + 1 \leq \frac{n(n-1)}{2} - n + 1$.
    \medskip
    
    \noindent This gives a total of at most $n^2 - 2n +2$ generators. The generators generate this fixed subgroup as an Artin group, so this is a bound on the rank in the sense of Artin groups. Moreover, it is clear from Remark \ref{RemFreeGroupExplicitly} that this bound is realised if $\Gamma$ is complete on $n$ vertices with all labels 3, by the fixed subgroup $\Fix(\varphi_a)$.

    Now we cover automorphisms falling under case 4. Here it follows from \ref{coroFixSigmaIotaa} that $F$ is the fundamental group of a subgraph of $\Gamma$, so its rank is bounded above by the rank of $\pi_1(\Gamma)$ which is (noticing that rank in our sense and the conventional sense coincide for free groups) equal to $|E| - |V| + 1 \leq \frac{n(n-1)}{2} - n + 1 = \frac{1}{2}(n^2 - 3n + 2)$, which is no more than $n^2 - 2n + 2$ for all positive integers.
\end{proof}

\subsubsection{Preliminaries}

Let $\gamma = \varphi_g \psi \in \Aut_{\Gamma}(A_{\Gamma})$ act elliptically on $X_{\Gamma}$. We know from Corollary \ref{fixAction} that $\Fix(\gamma) \leq \Stab(X_{\Gamma}^{\gamma})$, so our main goal will be to study $X_{\Gamma}^{\gamma}$.

First of all, Lemma \ref{LemmaReduction} together with Lemma \ref{reductionToStabs} show that we can reduce the study of $\Fix(\gamma)$ to that of $\Fix(\psi)$, some appropriate $\Fix(\varphi_{a^k} \psi)$, or the fixed subgroup of an automorphism of a dihedral Artin group.

\begin{lemma}\label{LemmaReduction}
Let $A_{\Gamma}$ be a $2$-dimensional Artin group, and let us consider an automorphism $\varphi_g \psi \in \Aut_{\Gamma}(A_{\Gamma})$ acting elliptically on $X_\Gamma$, where $g \in A_{\Gamma}$ and $\psi = \sigma \iota^{\varepsilon}$ for some $\varepsilon \in \{0, 1\}$. Then one of the following happens:
    \begin{enumerate}
        \item $g = h\psi(h)^{-1}$ for some $h \in A_{\Gamma}$. In particular, $\varphi_g\psi \sim_h \psi$ and $\Fix(\varphi_g\psi) = h\Fix(\psi)h^{-1}$;
        \item $g=ha^k\psi(h)^{-1}$ where $a$ is a standard generator such that $\sigma(a) = a$, $k \neq 0$ and $h \in A_{\Gamma}$. Moreover, if $\psi = \sigma\iota$, we can arrange that $k = 1$.
        
        In particular, $\varphi_g\psi \sim_h \varphi_{a^k}\psi$ and $\Fix(\varphi_g\psi) = h\Fix(\varphi_{a^k}\psi)h^{-1}$;
        
        \item $X_{\Gamma}^{\varphi_g \psi}$ is a single type 2 vertex.
    \end{enumerate}
\end{lemma}

\begin{proof}
    Choose a point of $X_\Gamma^{\varphi_g \psi}$ of lowest possible type, and write it as $hv_S$ where $h \in A_{\Gamma}$ and $v_S$ is a point in the fundamental domain $K_{\Gamma}$. We split into 3 cases, based on the type of $v_S$.
    \medskip
    
    \noindent \underline{Case 1: $type(v_S) = 0$.} By Lemma \ref{LemmaUsefulEquation}, we have $g \psi(h) \in h A_S = \{h\}$. This yields $g = h\psi(h)^{-1}$, as required. For the second part, notice that $\varphi_g\psi = \varphi_h\varphi_{\psi(h)^{-1}}\psi = \varphi_h\psi\varphi_h^{-1}$. The result follows from Lemma \ref{lemmaIsogredience}.
    \medskip

    \noindent 
    \noindent \underline{Case 2: $type(v_S) = 1$.} By Lemma \ref{LemmaUsefulEquation}, we have $g \psi(h) \in h A_S = h \langle a \rangle$ for some $a \in V(\Gamma)$. This means there is a $k \in \mathbb{Z}$ such that $g = ha^k\psi(h)^{-1}$. Furthermore it can't be that $k = 0$, or $\varphi_g \psi$ would fix the point $hv_{\emptyset}$, contradicting the minimality of $type(v_S)$:
    $$\varphi_{h\psi(h)^{-1}} \psi \cdot hv_\emptyset = h\psi(h)^{-1}\psi(h)(\psi \cdot v_\emptyset) = hv_\emptyset.$$
    It is direct from Lemma \ref{LemmaUsefulEquation} that $\sigma(a) = a$. Moreover, assume that $\psi = \sigma\iota$. Decompose $k = 2q+r$ where $q \in \mathbb{Z}$ and $r \in \{0, 1\}$. We calculate
    $$g = h a^k \psi(h)^{-1} = (h a^q) a^r (a^q) \psi(h)^{-1} = (h a^q) a^r \psi(h a^q)^{-1},$$
    and so by replacing $h$ with $h a^q$ we get the desired result (again, if $r = 0$ this would contradict the minimality of $type(v_S)$). For the final part, notice that $\varphi_g\psi = \varphi_h\varphi_{a^k}\varphi_{\psi(h)^{-1}}\psi = \varphi_h\varphi_{a^k}\psi\varphi_h^{-1}$. The result follows from Lemma \ref{lemmaIsogredience}.
    \medskip

    \noindent \underline{Case 3: $type(v_S) = 2$.} Then $h v_S$ also has type $2$, and all its neighbours in $X_{\Gamma}$ are of type 0 or 1. They are not fixed, by the minimality assumption. By Lemma \ref{LemmaMinsetConvex} the set $X_{\Gamma}^{\varphi_g \psi}$ is a convex subcomplex, so we must have $X_{\Gamma}^{\varphi_g \psi} = \{h v_S\}$.
\end{proof}

\begin{remark} \label{RemarkCases}
    Lemma \ref{LemmaReduction} has strong consequences. It implies that any automorphism $\gamma = \varphi_g \sigma \iota^{\varepsilon}$ acting elliptically, with $g \in A_{\Gamma}$ and $\varepsilon \in \{0, 1\}$, can be reduced to one of $\sigma$, $\sigma \iota$, $\varphi_{a^k} \sigma$, $\varphi_a \sigma \iota$, or an automorphism that fixes a single type $2$ vertex. In the last case, understanding $\Fix(\varphi_g \psi)$ reduces to understanding the fixed subgroups of automorphisms of dihedral Artin groups, which has already been addressed in Section \ref{sectionDihedral}.
\end{remark}

The following lemma gives more information about the structure of $X_{\Gamma}^{\gamma}$ when $\gamma$ is one of $\varphi_{a^k} \sigma$ or $\varphi_a \sigma \iota$.

\begin{lemma}\label{LemmaEss1Skeleton}
Let $A_{\Gamma}$ be a $2$-dimensional Artin group, and consider a non-trivial automorphism of the form $\varphi_{a^k}\psi$ where $k \neq 0$ and $\psi \coloneqq \sigma \iota^{\varepsilon}$ for some $\varepsilon \in \{0, 1\}$. Suppose either that $\varepsilon = 0$ or that $k = 1$. Then the set $X_{\Gamma}^{\varphi_a^k \psi}$ is a simplicial tree contained in $X_{\Gamma}^{(1)-ess}$.
\end{lemma}

\begin{proof}
    Suppose that $X_{\Gamma}^{\varphi_a^k \psi}$ contains a type $0$ vertex, say some $g v_{\emptyset}$. By Lemma \ref{LemmaUsefulEquation}, we must have $a^k \psi(g) = g$.

    The argument now relies on height, noticing that $\sigma$ preserves height, while $\sigma \iota$ preserves the parity of height (as $\iota$ reverses height).
    If $\varepsilon = 0$, then $ht(\psi(g)) = ht(g)$. As $ht(a^k) = k \neq 0$, we obtain a contradiction to the previous equation. If $k = 1$ then $ht(\psi(g))$ and $ht(g)$ have the same parity. But $ht(a) = 1$, so we also get a contradiction to the previous equation.

    Recall that $X_{\Gamma}^{\varphi_a^k \psi}$ is a subcomplex of $X_{\Gamma}$ by Lemma \ref{LemmaMinsetConvex}. Hence the above proves that under our hypotheses, $X_{\Gamma}^{\varphi_a^k \psi}$ is a graph in $X_{\Gamma}$. By Lemma \ref{LemmaMinsetConvex} again, $X_{\Gamma}^{\varphi_a^k \psi}$ is convex. It follows that it is $\operatorname{CAT}(0)$, therefore simply connected, and hence it is a tree.
\end{proof}

We now introduce some notation that will be used throughout the rest of the section.

\begin{notation} Given a graph automorphism $\sigma$, we call:
\begin{enumerate}
    \item $F_{\sigma}$ the subset of $V(\Gamma)$ of the vertices that are fixed by $\sigma$ ;
    \item $T_{\sigma}$ the subset of $V(\Gamma)$ of the vertices that are transposed by $\sigma$, where $s \in V(\Gamma)$ is said to be \emph{transposed} by $\sigma$ if there exists some $t \in V(\Gamma) \backslash \{s\}$ satisfying $m_{st} < \infty$ such that $\sigma \cdot s = t$ and $\sigma \cdot t = s$. In that case, we say that $s$ and $t$ are \emph{$\sigma$-transposed}.
\end{enumerate}
\end{notation}

\begin{lemma}\label{LemmaBaseCases}
Let $A_{\Gamma}$ be a $2$-dimensional Artin group. Let $\gamma$ be one of $\sigma$, $\sigma \iota$, $\varphi_{a^k} \sigma$ or $\varphi_a \sigma \iota$, where $a \in F_{\sigma}$ and $k \neq 0$. Then the intersection $X_{\Gamma}^{\gamma} \cap K_{\Gamma}$ is the flag completion of the following vertices:
    \begin{enumerate}
        \item \underline{If $\gamma$ is $\sigma$ or $\sigma \iota$:} The vertices are $v_{\emptyset}$, the $v_s$ with $s \in F_{\sigma}$, and the $v_{st}$ where $s, t \in F_{\sigma}$ or $s, t \in T_{\sigma}$ are $\sigma$-transposed.
        \item \underline{If $\gamma$ is $\varphi_{a^k} \sigma$ or $\varphi_a \sigma \iota$:} The vertices are $v_a$, and the $v_{as}$ where $s \in F_{\sigma}$.
        
    \end{enumerate}
\end{lemma}

\begin{proof}
    The automorphism $\iota$ acts trivially on $K_{\Gamma}$, so the cases $\sigma \iota$ and $\varphi_a \sigma \iota$ can be brought back to the cases $\sigma$ and $\varphi_a \sigma$, respectively. Now we have that:
    \\(1) Lemma \ref{LemmaUsefulEquation} gives $v_S \in X_{\Gamma}^{\sigma} \Longleftrightarrow \sigma(S) = S$. The result follows.
    \\(2) Lemma \ref{LemmaUsefulEquation} gives $v_S \in X_{\Gamma}^{\varphi_a^k \sigma} \Longleftrightarrow a^k \in A_S \text{ and } \sigma(S) = S$. The result follows.
\end{proof}

\subsubsection{Elliptic automorphisms involving no global inversion}

In this section we compute $\Fix(\gamma)$ when $\gamma$ acts elliptically on $X_{\Gamma}$ and does not involve $\iota$. More explicitly, and following Remark \ref{RemarkCases}, we will assume that $\gamma$ is one of $\sigma$ or $\varphi_{a^k} \sigma$. Note that Lemma \ref{LemmaFixOfSigma} and Corollary \ref{CoroFixedPointsOfSigma} were already known by Crisp (see \cite{crisp2000symmetrical}). We include proofs in our language here for completeness.

\begin{lemma} \label{LemmaFixOfSigma}
    A vertex $g v_S$ belongs to $X_{\Gamma}^{\sigma}$ if and only if $v_S \in X_{\Gamma}^{\sigma} \cap K_{\Gamma}$ and $g$ is a product of elements of the following forms:
        \begin{itemize}
            \item standard generators $s$ where $s \in F_{\sigma}$ ;
            \item powers of Garside elements $\Delta_{st}^k$ where $s, t \in T_{\sigma}$ are $\sigma$-transposed.
        \end{itemize}
\end{lemma}

\begin{proof}
    We know by Lemma \ref{LemmaBaseCases} that $v_{\emptyset} \in X_{\Gamma}^{\sigma}$. Let $Y$ be the subcomplex of $X_{\Gamma}$ spanned by all the vertices of the form $g v_S$ where $v_S \in X_{\Gamma}^{\sigma} \cap K_{\Gamma}$ and $g$ is a product of elements of one of the two forms described in the lemma. We want to show that $Y = X_{\Gamma}^{\sigma}$. Because the two sets meet at $v_{\emptyset}$, we can proceed by induction on the combinatorial distance to $v_{\emptyset}$. More rigorously, let $B_Y(v_{\emptyset}, n)$ and $B_{X_{\Gamma}^{\sigma}}(v_{\emptyset}, n)$ be the balls of radius $n$ around $v_{\emptyset}$ in $Y$ and in $X_{\Gamma}^{\sigma}$ respectively, relatively to the combinatorial metric $d^C$. We will show by induction on $n$ that $B_Y(v_{\emptyset}, n) = B_{X_{\Gamma}^{\sigma}}(v_{\emptyset}, n)$.

    To do so, we suppose that we reached some vertex $g v_S$ in $B_Y(v_{\emptyset}, n) = B_{X_{\Gamma}^{\sigma}}(v_{\emptyset}, n)$. In particular, we can pick $g$ such that it is a product of elements of one of the two forms described in the lemma. It is then straightforward to check that $\sigma(g) = g$. The neighbours of $g v_S$ take the form $gh v_{S'}$ for some $h \in A_S$ and some appropriate $S'$. We will show that $gh v_{S'} \in X_{\Gamma}^{\sigma}$ if and only if $v_{S'} \in X_{\Gamma}^{\sigma} \cap K_{\Gamma}$ and $h$ is of one of the two forms described in the lemma. Note that the former condition can be understood easily thanks to Lemma \ref{LemmaBaseCases}.
    \medskip

    \noindent \underline{Case 1: $g v_S = g v_{\emptyset}$ has type 0:}
    The neighbours of $g v_{\emptyset}$ have the form $g v_s$ or $g v_{st}$, where $m_{st} < \infty$. Then:
    \begin{itemize}
        \item $g v_s \in X_{\Gamma}^{\sigma} \overset{\ref{LemmaUsefulEquation}} \Longleftrightarrow \sigma(s) = s \Longleftrightarrow s \in F_{\sigma}$ ;
        \item $g v_{st} \in X_{\Gamma}^{\sigma} \overset{\ref{LemmaUsefulEquation}} \Longleftrightarrow \sigma(\{s, t\}) = \{s, t\} \Longleftrightarrow$ either $s, t \in F_{\sigma}$ or $s, t \in T_{\sigma}$ are $\sigma$-transposed.
    \end{itemize}
    In this case $h = 1$ and $v_{S'} \in \{v_s, v_{st}\}$ belongs to $X_{\Gamma}^{\sigma} \cap K_{\Gamma}$.
    \medskip

    \noindent \underline{Case 2: $g v_S = g v_s$ has type 1:} 
    
    By Lemma \ref{LemmaUsefulEquation}, we have $s \in F_{\sigma}$. The neighbours of $g v_s$ have the form $g s^n v_{\emptyset}$ or $g s^n v_{st} = g v_{st}$ for some $n \in \mathbb{Z}$, where $m_{st} < \infty$. Then:
    \begin{itemize}
        \item $g s^n v_{\emptyset} \in X_{\Gamma}^{\sigma} \overset{\ref{LemmaUsefulEquation}} \Longleftrightarrow \sigma(s) = s$, \text{ which always holds true} ;
        \item $g v_{st} \in X_{\Gamma}^{\sigma} \overset{\ref{LemmaUsefulEquation}} \Longleftrightarrow \sigma(\{s, t\}) = \{s, t\} \overset{s \in F_{\sigma}} \Longleftrightarrow$ $s, t \in F_{\sigma}$.
    \end{itemize}
    In this case $h \in \langle s \rangle$ where $s \in F_{\sigma}$, and $v_{S'} \in \{v_{\emptyset}, v_{st}\}$ belongs to $X_{\Gamma}^{\sigma} \cap K_{\Gamma}$.
    \medskip

    \noindent \underline{Case 3: $g v_S = g v_{st}$ has type 2:} By Lemma \ref{LemmaUsefulEquation}, either $s, t \in F_{\sigma}$ or $s, t \in T_{\sigma}$ are $\sigma$-transposed. The neighbours of $g v_{st}$ take the form $gh v_{\emptyset}$ or $gh v_s$ (without loss of generality) for some $h \in A_{st}$. Then:
    \begin{itemize}
        \item $gh v_{\emptyset} \in X_{\Gamma}^{\sigma} \overset{\ref{LemmaUsefulEquation}} \Longleftrightarrow \sigma(h) = h \overset{(*)} \Longleftrightarrow \begin{cases} h \in A_{st} \text{ if } s, t \in F_{\sigma} \\ h = \Delta_{st}^k \text{ if } s, t \in T_{\sigma} \text{ are } \sigma\text{-transposed} \end{cases}$ ;
        \item $gh v_s \in X_{\Gamma}^{\sigma} \overset{\ref{LemmaUsefulEquation}} \Longleftrightarrow \sigma(h) \in h \langle s \rangle$ and $\sigma(s) = s \overset{\sigma(\{s, t\}) = \{s, t\}} \Longleftrightarrow$ $s, t \in F_{\sigma}$ and $h \in A_{st}$ ;
    \end{itemize}
    where $(*)$ follows from Theorem \ref{dihedralFixPointsOdd}.(1) if $m_{st}$ is odd, and from Theorem \ref{dihedralFixPointsEven}.(5) if $m_{st}$ is even. In this case either $h \in A_{st}$ and $v_{S'} \in \{v_{\emptyset}, v_s, v_t\}$ belongs to $X_{\Gamma}^{\sigma} \cap K_{\Gamma}$ (if $s, t \in F_{\sigma}$), or $h = \Delta_{st}^k$ and $v_{S'} = v_{\emptyset}$ belongs to $X_{\Gamma}^{\sigma} \cap K_{\Gamma}$ (if $s, t \in T_{\sigma}$ are $\sigma$-transposed).
\end{proof}

\begin{corollary} \label{CoroFixedPointsOfSigma}
    The subgroup $\Fix(\sigma)$ is exactly the subgroup generated by the two kinds of elements described in the statement of Lemma \ref{LemmaFixOfSigma}. In particular, $\Fix(\sigma)$ is abstractly the Artin group $A_{\Gamma^\sigma}$, where
    $$\Gamma^{\sigma} \coloneqq Span(F_{\sigma}) \sqcup \left( \bigsqcup\limits_{s, t \in T_{\sigma} \ \sigma\text{-transp.}} w_{st} \right)$$
    is the graph obtained from $\Gamma$ by considering the subgraph $\Gamma_{\sigma}$ spanned by all the vertices of $F_{\sigma}$, and then adding an isolated vertex for every pair of $\sigma$-transposed generators.
\end{corollary}

\begin{proof}
    We first prove the first statement. Let us call $H$ the subgroup generated by the two kinds of elements described in the statement of Lemma \ref{LemmaFixOfSigma}.

    \medskip
    \noindent ($\Fix(\sigma) \subseteq H$) By Corollary \ref{fixAction} we know that $\Fix(\sigma) \leq \Stab(X_{\Gamma}^{\sigma})$. Let now $g \in \Stab(X_{\Gamma}^{\sigma})$. As $v_{\emptyset} \in X_{\Gamma}^{\sigma}$, we must have $g v_{\emptyset} \in X_{\Gamma}^{\sigma}$. By Lemma \ref{LemmaFixOfSigma} then, $g$ belongs to $H$.

    \medskip
    \noindent ($H \subseteq \Fix(\sigma)$) This is immediate from the fact that all of the generators of $H$ are fixed by $\sigma$.
    \medskip

    We now prove the second statement. We know by Corollary \ref{fixAction} that $\Fix(\sigma)$ acts on $X_\Gamma^\sigma$. The subcomplex $X_\Gamma^\sigma \cap K_{\Gamma}$ is a fundamental domain for this action. To see this, note that any point $gv_S \in X_\Gamma^\sigma$ is a translate of a point of $v_S \in X_\Gamma^\sigma \cap K_{\Gamma}$ by an element of $H$, which is clearly fixed by $\sigma$. Moreover, this translate is unique because $K_{\Gamma}$ is a fundamental domain for the action $A_{\Gamma} \curvearrowright X_{\Gamma}$.

    Consider the quotient of $Y_\Gamma^\sigma$ which collapses the subcomplex spanned by vertices $v_S$ (where $S \subseteq F_\sigma$) to $v_\emptyset$, and acts trivially on the vertices $v_S$ where $S$ is a $\sigma$-transposed pair. This quotient map can be propagated equivariantly to a continuous map $r$ from $X_\Gamma^\sigma$ onto a tree. By equivariance, the resulting tree still carries a $\Fix(\sigma)$-action, and its fundamental domain, which is $r(X_\Gamma^\sigma \cap K_{\Gamma})$, is a star where the central vertex has stabiliser $A_{F_\sigma}$, and there is one other vertex for each pair $s,t \in T_\sigma$ which were $\sigma$-transposed, with stabiliser $\langle \Delta_{st} \rangle$. The last comment follows from our computation of the fixed subgroup of $\sigma$ restricted to the relevant dihedral, in Theorem \ref{dihedralFixPointsOdd}.1 and Theorem \ref{dihedralFixPointsEven}.5.

    The stabilisers of the edges in this graph of groups decomposition are trivial (owing to the fact that their preimages under the quotient map are single type 0 points), so the graph of groups decomposition gives rise to the proposed free product decomposition of $\Fix(\sigma)$.

\end{proof}

Recall that $a \in V(\Gamma)$ is a standard generator that we assume belongs to $F_{\sigma}$.

\begin{lemma} \label{LemmaFixOfSigmaa}
    The subset $X_{\Gamma}^{\varphi_{a^k} \sigma}$ is the subtree of the standard tree $X_{\Gamma}^{a}$ whose vertices are exactly those of the form $g v_s$ and $g v_{st}$ where $\sigma(g) = g$ and $s, t \in F_{\sigma}$.
\end{lemma}

\begin{proof}
    First of all, we know by Lemma \ref{LemmaEss1Skeleton} that $X_{\Gamma}^{\varphi_{a^k} \sigma}$ is a simplicial tree inside of $X_{\Gamma}^{(1)-ess}$. By Lemma \ref{LemmaBaseCases}, we also know that $v_a \in X_{\Gamma}^{\varphi_{a^k} \sigma}$. We let $T$ be the subtree of $X_{\Gamma}^{a}$ whose vertices are exactly those of the form $g v_s$ or $g v_{st}$ where $\sigma(g) = g$ and $s, t \in F_{\sigma}$. Note that $v_a \in T$.

    We want to show that $T = X_{\Gamma}^{\varphi_{a^k} \sigma}$. To do that, we proceed as in the proof of Lemma \ref{LemmaFixOfSigma}. Our two sets meet at $v_a$, so we can proceed by induction on the combinatorial distance to $v_a$. We let $B_{T}(v_a, n)$ and $B_{X_{\Gamma}^{\varphi_{a^k} \sigma}}(v_a, n)$ be the balls of radius $n$ around $v_a$ in $T$ and in $X_{\Gamma}^{\varphi_{a^k} \sigma}$ respectively, relatively to $d^C$. We show by induction on $n$ that $B_{T}(v_a, n) = B_{X_{\Gamma}^{\varphi_{a^k} \sigma}}(v_a, n)$. The base case is trivial, so we deal with the induction step. As in Lemma \ref{LemmaFixOfSigma}, we assume that $g v_S$ belongs to both balls, and we consider a neighbour $gh v_{S'}$ of $g v_S$, for some $h \in A_S$ and some appropriate $S'$.
    \medskip
    
    \noindent \underline{Case 1: $g v_S = g v_s$ has type $1$:} As $g v_s \in T$ we have $\sigma(g) = g$ and $s \in F_{\sigma}$. As $g v_s \in X_{\Gamma}^{\varphi_{a^k} \sigma}$, we also have $a^k \sigma(g) \in g \langle s \rangle$ by Lemma \ref{LemmaUsefulEquation}.
    
    The neighbours of $g v_s$ in $X_{\Gamma}$ take the form $g v_{st}$. Using the above, we have
    \begin{itemize}
        \item $g v_{st} \in T \Longleftrightarrow \sigma(g) = g$ and $s, t \in F_{\sigma} \Longleftrightarrow t \in F_{\sigma}$ ;
        \item $g v_{st} \in X_{\Gamma}^{\varphi_{a^k} \sigma} \overset{\ref{LemmaUsefulEquation}} \Longleftrightarrow a^k \sigma(g) \in g A_{st}$ and $\sigma(\{s, t\}) = \{s, t\} \Longleftrightarrow t \in F_{\sigma}$.
    \end{itemize}
    It follows that $g v_{st} \in T$ if and only if $g v_{st} \in X_{\Gamma}^{\varphi_{a^k} \sigma}$.
    \medskip

    \noindent \underline{Case 2: $g v_S = g v_{st}$ has type $2$:} By the induction process, we know that there is some type $1$ neighbour of $g v_{st}$ in $B_{T}(v_a, n-1) = B_{X_{\Gamma}^{\varphi_{a^k} \sigma}}(v_a, n-1)$. Call this neighbour $g h v_x$ for some $h \in A_{st}$ and $x \in \{s, t\}$. Because $gh v_x$ belongs to $T$, we have
    $$\sigma(gh) = gh \ \ (1)$$ 
    But by definition of $T$, $gh v_x$ also belongs to the standard tree $X_{\Gamma}^a$. Since $X_{\Gamma}^a = X_{\Gamma}^{\varphi_a}$ by compatibility, Lemma \ref{LemmaUsefulEquation} gives $a gh \in gh \langle x \rangle$. Comparing heights, we actually have $agh = ghx$, and thus
    $$(gh)^{-1} a (gh) = x. \ \ (2)$$
    Finally, using that $g v_{st}$ belongs to $T$, we also obtain
    $$s, t \in F_{\sigma}. \ \ (3)$$
    We now consider the other type $1$ neighbours of $g v_{st}$ in $X_{\Gamma}$. Without loss of generality, and up to replacing $x$ by $y \in \{s, t \} \backslash \{x\}$, they have the form $g h h' v_s$ for some $h' \in A_{st}$. Our first claim is relatively easy:
    \medskip

    \noindent \underline{Claim 1:} $ghh' v_s \in T$ if and only if $ghh' v_s \in X_{\Gamma}^a$.
    \medskip

    \noindent \underline{Proof of Claim 1:} The “only if” case is by definition of $T$. For the “if” part, we already know that $\sigma(gh) = gh$ by (1). As $h' \in A_{st}$ and $s, t \in F_{\sigma}$ by (3), we also have $\sigma(h') = h'$. Altogether, we have $\sigma(ghh') = ghh'$ so $ghh' v_s \in T$ by definition. This finishes the proof of Claim 1.
    \medskip

    Our next claim gives the main statement of the lemma.
    \medskip

    \noindent \underline{Claim 2:} $ghh' v_s \in X_{\Gamma}^{\varphi_{a^k} \sigma}$ if and only if $ghh' v_s \in X_{\Gamma}^a$.
    \medskip

    \noindent \underline{Proof of Claim 2:} Notice that
    \begin{align*}
        ghh' v_s \in X_{\Gamma}^{\varphi_{a^k} \sigma} & \overset{\ref{LemmaUsefulEquation}} \Longleftrightarrow a^k \sigma(ghh') \in ghh' \langle s \rangle \\
        & \overset{(1)} \Longleftrightarrow a^k gh \sigma(h') \in ghh' \langle s \rangle \\
        & \overset{(3)} \Longleftrightarrow a^k ghh' \in ghh' \langle s \rangle \\
        & \Longleftrightarrow a^k ghh' = ghh' s^k \text{ (compare heights)} \\
        & \Longleftrightarrow (gh)^{-1} a^k (gh) = h' s^k (h')^{-1} \\
        & \overset{(2)} \Longleftrightarrow x^k = h' s^k (h')^{-1}. \ \ \ (4)
    \end{align*}
    Now there are three cases to consider. In each case, we will show that (4) is equivalent to $ghh' v_s \in X_{\Gamma}^a$.
    \smallskip

    \noindent \underline{Case 2.1: $x = s$.} Then
    \begin{align*}
        (4) \Longleftrightarrow h' \in C(s^k) \overset{\ref{PropCentralisers}} \Longleftrightarrow h' \in C(s) & \overset{(*)} \Longleftrightarrow h' \in \Stab(X_{\Gamma}^s) \\
        & \Longleftrightarrow h' v_s \in X_{\Gamma}^s \Longleftrightarrow ghh' v_s \in gh X_{\Gamma}^s \overset{(**)} \Longleftrightarrow ghh' v_s \in X_{\Gamma}^a,
    \end{align*}
    where $(*)$ follows from \cite[Lemmas 2.20, 2.21]{vaskou2023isomorphism}, and where $(**)$ is because the standard trees $gh X_{\Gamma}^s$ and $X_{\Gamma}^a$ both contain $gh v_s$, so in particular their local groups are both equal to $gh \langle s \rangle (gh)^{-1}$, hence the two standard trees agree.
    \smallskip

    \noindent \underline{Case 2.2: $x = t$ and $m_{st}$ is even.} Then:
    \begin{align*}
        (4) &\Longleftrightarrow t^k = (h') s^k (h')^{-1} \\
        &\Longleftrightarrow X_{\Gamma}^{t^k} = h' X_{\Gamma}^{s^k} \\
        &\Longleftrightarrow X_{\Gamma}^t = h' X_{\Gamma}^s \\
        &\Longleftrightarrow t = (h') s (h')^{-1}.
    \end{align*}
    The last statement is false as $m_{st}$ is even so $t$ and $s$ are not conjugated in $A_{st}$. Hence this case gives a contradiction.
    \smallskip

    \noindent \underline{Case 2.3: $x = t$ and $m_{st}$ is odd.} We have
    \begin{align*}
        (4) & \overset{Case \ 2.2} \Longleftrightarrow s = (h')^{-1} t (h') \\
        & \ \ \Longleftrightarrow s = (\Delta_{st} h')^{-1} s (\Delta_{st} h') \\
        & \ \ \Longleftrightarrow \Delta_{st} h' \in C_{A_{st}}(s) = \langle s, \Delta_{st}^2 \rangle \\
        & \ \ \Longleftrightarrow h' = \Delta_{st}^{2q+1} s^l \text{ for some } q, l \in \mathbb{Z} \\
        & \ \ \Longleftrightarrow ghh' v_s = gh \Delta_{st}^{2q+1} v_s \text{ for some } q \in \mathbb{Z}. \ \ (5)
    \end{align*}
    Notice that $gh v_t \in gh X_{\Gamma}^t$, and $gh v_t \in X_{\Gamma}^a$ by hypothesis, so by unicity we must have $gh X_{\Gamma}^t = X_{\Gamma}^a$. Now we use \cite[Lemma 4.3]{martin2022acylindrical} with $m_{st}$ odd, that states that the vertices adjacent to $g v_{st}$ in $gh X_{\Gamma}^t$, provided they are in the orbit of $v_s$, are exactly those of the form $gh \Delta_{st}^{2q+1} v_s$ for some $q \in \mathbb{Z}$. So $(5)$ is equivalent to $ghh' v_s \in gh X_{\Gamma}^t = X_{\Gamma}^a$.
    \medskip

    This finishes the proof of Claim 2, and of the lemma.
\end{proof}

\begin{remark} \label{RemFixOfSigmaa}
    It follows from Lemma \ref{LemmaFixOfSigmaa} that $X_{\Gamma}^{\varphi_{a^k} \sigma}$ is contained not only in $X_{\Gamma}^a$ but also in $X_{\Gamma}^{\sigma}$. Indeed, if $g v_s$ is a vertex of $X_{\Gamma}^{\varphi_{a^k} \sigma}$ then $g v_s \in T$ and thus we have $\sigma(g) = g$ and $s \in F_{\sigma}$. In particular then, $\sigma \cdot g v_s = \sigma(g) v_{\sigma(s)} = g v_s$. The same goes if we replace $g v_s$ with $g v_{st}$.
\end{remark}

\begin{corollary} \label{CoroStabSigmaa}
    The subgroup $\Fix({\varphi_{a^k} \sigma})$ is isomorphic to the direct product $\langle a \rangle \times F$, where $F$ is a finitely generated free group.
\end{corollary}

\begin{proof}
    We first start by computing $\Stab(X_{\Gamma}^{\varphi_{a^k} \sigma})$. The arguments are actually tantamount to those of \cite[Lemma 4.5]{martin2022acylindrical}, that we re-explain thereafter. Every edge stabiliser in $X_{\Gamma}^{\varphi_{a^k} \sigma}$ is $\langle a \rangle$  as $X_{\Gamma}^{\varphi_{a^k} \sigma} \subseteq X_{\Gamma}^a$ by Lemma \ref{LemmaFixOfSigmaa}. Let $e$ be such an edge, and let $g \in \Stab(X_{\Gamma}^{\varphi_{a^k} \sigma})$. Then the local group at $g e$ is both $\langle a \rangle$ and $g \langle a \rangle g^{-1}$. Since we can't have $a = g a^{-1} g^{-1}$ for height reasons, it must be that $a = g a g^{-1}$ and thus $g \in C(a)$. This means $\langle a \rangle$ is in the centre of $\Stab(X_{\Gamma}^{\varphi_{a^k} \sigma})$, and thus the quotient
    $$F \coloneqq \quotient{\Stab(X_{\Gamma}^{\varphi_{a^k} \sigma})}{\langle a \rangle}$$
    acts on $X_{\Gamma}^{\varphi_{a^k} \sigma}$ with trivial edge stabilisers. In particular then, $F$ is the fundamental group of the graph of groups $X_{\Gamma}^{\varphi_{a^k} \sigma} / F$. On one hand, the edge groups are trivial. On the other hand, we know by Remark \ref{RemFixOfSigmaa} that $X_{\Gamma}^{\varphi_{a^k} \sigma}$ is contained in $X_{\Gamma}^{\sigma}$, and only contains type $1$ and type $2$ vertices. Using Lemma \ref{LemmaFixOfSigma}, it follows that the vertex groups are:
    \begin{itemize}
        \item Either isomorphic to $\langle s, \Delta_{st} \rangle / \langle a \rangle$ when the vertex is in the orbit of $v_{st}$ and $s, t \in T_{\sigma}$ are $\sigma$-transposed. This quotient is itself virtually isomorphic to $\mathbb{Z}$ given that $\mathbb{Z} \cong \langle a \rangle \subseteq \langle s, \Delta_{st} \rangle \cong \mathbb{Z}^2$ (one can actually remove the “virtually” condition using that $ht(a) = 1$) ;
        \item Or the vertex groups are trivial, in all other scenarios.
    \end{itemize}
    
    Consequently, the group $F$ is a free group. It follows that the short exact sequence associated with the quotient splits and $\Stab(X_\Gamma^{\varphi_{a^k}\sigma}) = \langle a \rangle \rtimes F$. Finally since $a$ is central the action of $F$ is trivial, we have $\Stab(X_{\Gamma}^{\varphi_{a^k} \sigma}) = \langle a \rangle \times F$.

    We know from Corollary \ref{fixAction} that $\Fix(\varphi_{a^k} \sigma) \leq \Stab(X_{\Gamma}^{\varphi_{a^k} \sigma})$. We want to show equality. To do that, let $g$ be any element of $\Stab(X_{\Gamma}^{\varphi_{a^k} \sigma})$, and let $v$ be a type $1$ vertex of $X_{\Gamma}^{\varphi_{a^k} \sigma}$. By hypothesis $gv \in X_{\Gamma}^{\varphi_{a^k} \sigma}$, and hence
    $$gv = \varphi_{a^k} \sigma \cdot gv \overset{\ref{LemmaUsefulEquation}} = a^k \sigma(g) (\sigma \cdot v) \overset{\ref{RemFixOfSigmaa}} = a^k \sigma(g) v.$$
    The above equation implies that $g^{-1} a^k \sigma(g) \in \Stab(v)$. As $v$ has type $1$ and belongs to $X_{\Gamma}^a$, we know that $\Stab(v) = \langle a \rangle$, and hence $g^{-1} a^k \sigma(g) \in \langle a \rangle$. Comparing heights then gives $g^{-1} a^k \sigma(g) = a^k$, which gives $g^{-1} (\varphi_{a^k} \sigma)(g) = 1$. Finally, we have $g \in \Fix(\varphi_{a^k} \sigma)$, as wanted.
\end{proof}

\begin{remark} \label{RemFreeGroupExplicitly}
    As in \cite[Remark 4.6]{martin2022acylindrical}, a basis of $F$ can then be determined explicitly as follows. Let $\Gamma_{bar}$ be the barycentric subdivision of $\Gamma$. In $\Gamma_{bar}$, we call $v^s$ the vertex corresponding the standard generator $s$, and we call $v^{st}$ the vertex corresponding to the edge of $\Gamma$ connecting $s$ and $t$. We see $\Gamma_{bar}$ as having edge coefficients by giving to any edge $(v^s, v^{st})$ the coefficient $m_{st}$. We let $\Gamma^{\sigma}$ be the subgraph of $\Gamma_{bar}$ containing all the edges of the form $(v^s, v^{st})$ where $s, t \in F_{\sigma}$. We then cut $\Gamma^{\sigma}$ along every vertex of the form $v^{st}$ where $m_{st}$ is even, and call the resulting graph $\Gamma_{odd}^{\sigma}$. Finally, we call $\Gamma_{a, odd}^{\sigma}$ the connected component of $\Gamma_{odd}^{\sigma}$ that contains $a$. One can use \cite[Lemma 4.3]{martin2022acylindrical} to obtain that
    $$\Gamma_{a, odd}^{\sigma} = \quotient{X_{\Gamma}^{\varphi_{a^k} \sigma}}{F}.$$
    By Corollary \ref{CoroStabSigmaa}, the local groups of edges are all isomorphic to $\mathbb{Z}$. Choose any ordering $v^{s_1} < \cdots < v^{s_q}$ of the vertices of $\Gamma_{a, odd}^{\sigma}$ that correspond to standard generators. Then label any edge of the form $[v^s, v^{st}]$ satisfying $v^s < v^t$ and $m_{st}$ odd with the element $\Delta_{st}$, and label every other edge with the trivial element. A basis of $F$ is given as follows:
    \medskip

    \noindent $\bullet$ A set of words that label closed paths starting at $v^a$, and that span a free basis of $\pi_1(\Gamma_{a, odd}^{\sigma})$ ;
    \medskip

    \noindent $\bullet$ For every vertex $v^{st}$ in $\Gamma_{a, odd}^{\sigma}$, a word $g_{st} z_{st} g_{st}^{-1}$ where $g_{st}$ labels a path joining $v^a$ and $v^{st}$ and $z_{st} = \Delta_{st}^k$ where $k = lcm(2, m_{st})$.
\end{remark}

\begin{remark}
    While Corollary \ref{CoroStabSigmaa} shows that the stabiliser $\Stab(X_{\Gamma}^{\varphi_{a^k} \sigma})$ splits as a direct product of the form $\langle a \rangle \times F$, Remark \ref{RemFreeGroupExplicitly} describes a basis of the free group $F$. It was showed in \cite[Lemma 4.5, Remark 4.6]{martin2022acylindrical} that the centraliser $C(a)$ splits as a direct product $\langle a \rangle \times F'$ where $F'$ is also a free group. Moreover, there is an explicit basis for $F'$, which is the same as the one described in Remark \ref{RemFreeGroupExplicitly} but for $\sigma \coloneqq id$. While a basis for $F$ cannot always be extracted from a basis for $F'$, a basis for $F$ can always be completed in a basis for $F'$. This holds true because:
    \medskip
    
    \noindent $\bullet$ A basis of $\pi_1(\Gamma_{a, odd}^{\sigma})$ can always be completed to a basis of $\pi_1(\Gamma_{a, odd}^{id})$ given that $\Gamma_{a, odd}^{\sigma}$ is a subgraph of $\Gamma_{a, odd}^{id}$ ;
    \medskip

    \noindent $\bullet$ The set of vertices of the form $v^{st}$ in $\Gamma_{a, odd}^{\sigma}$ is a subset of the set of vertices of the form $v^{st}$ in $\Gamma_{a, odd}^{id}$. 
\end{remark}

\subsubsection{Elliptic automorphisms involving the global inversion}

This section completes the previous section concerning the automorphisms acting elliptically on $X_{\Gamma}$. In other words, here we compute $\Fix(\gamma)$ when $\gamma$ involves $\iota$. In regards to Remark \ref{RemarkCases}, we can assume that $\gamma$ is one of $\sigma \iota$ or $\varphi_{a} \sigma \iota$.

\begin{lemma}\label{LemmaFixOfIotaSigma}
    A vertex $g v_S$ belongs to $X_{\Gamma}^{\sigma \iota}$ if and only if $v_S \in X_{\Gamma}^{\sigma \iota} \cap K_{\Gamma}$ and $g$ is a product of elements of the following form:
        \begin{itemize}
            \item Elements in the $\mathbb{Z}$-subgroup described in Theorem \ref{dihedralFixPointsEven}.(7), where $A_{2n}$ is some even parabolic dihedral $A_{st}$ where $s, t \in T_{\sigma}$ are $\sigma$-transposed.
        \end{itemize}
\end{lemma}

\begin{proof}
    Our strategy closely follows that of Lemma \ref{LemmaFixOfSigma}, to which we refer the reader. We call $Y$ the subcomplex of $X_{\Gamma}$ spanned by the vertices of the form $g v_S$ where $v_S \in X_{\Gamma}^{\sigma \iota} \cap K_{\Gamma}$ and $g$ is a product of elements of the form described in the lemma. We proceed by induction to show that the combinatorial balls $B_Y(v_{\emptyset}, n)$ and $B_{X_{\Gamma}^{\sigma \iota}}(v_{\emptyset}, n)$ are equal. Let $g v_S$ be a vertex that belongs to both balls of radius $n$. Note that we can pick $g$ such that it is a product of elements of the form described in the lemma, and thus $(\sigma \iota)(g) = g$ by definition (the elements described in the lemma satisfy that equation, by Theorem \ref{dihedralFixPointsEven}). We want to show that a neighbour $g h v_{S'}$ of $g v_S$ belongs to $X_{\Gamma}^{\sigma \iota}$ if and only if $v_{S'} \in X_{\Gamma}^{\sigma \iota} \cap K_{\Gamma}$ and $h$ is of the form described in the lemma. Once again, the former condition can be understood easily using Lemma \ref{LemmaBaseCases}.
    \medskip

    \noindent \underline{Case 1: $g v_S = g v_{\emptyset}$ has type 0:} The neighbours of $g v_{\emptyset}$ have the form $g v_s$ or $g v_{st}$. Then:
    \begin{itemize}
        \item $g v_s \in X_{\Gamma}^{\sigma \iota} \overset{\ref{LemmaUsefulEquation}} \Longleftrightarrow \sigma(s) = s \Longleftrightarrow s \in F_{\sigma}$ ;
        \item $g v_{st} \in X_{\Gamma}^{\sigma \iota} \overset{\ref{LemmaUsefulEquation}} \Longleftrightarrow \sigma(\{s, t\}) = \{s, t\} \Longleftrightarrow$ either $s, t \in F_{\sigma}$ or $s, t \in T_{\sigma}$ are $\sigma$-transposed.
    \end{itemize}
    Here $h = 1$ and $v_{S'} \in \{v_s, v_{st}\}$ belongs to $X_{\Gamma}^{\sigma \iota} \cap K_{\Gamma}$.
    \medskip

    \noindent \underline{Case 2: $g v_S = g v_s$ has type 1:} By Lemma \ref{LemmaUsefulEquation}, we have $s \in F_{\sigma}$. The neighbours of $g v_s$ have the form $g s^n v_{\emptyset}$ or $g s^n v_{st} = g v_{st}$ for some $n \in \mathbb{Z}$, where $m_{st} < \infty$. Then:
    \begin{itemize}
        \item $g s^n v_{\emptyset} \in X_{\Gamma}^{\sigma \iota} \overset{\ref{LemmaUsefulEquation}} \Longleftrightarrow (\sigma \iota)(s^n) = s^n \Longleftrightarrow n = 0$ ;
        \item $g v_{st} \in X_{\Gamma}^{\sigma \iota} \overset{\ref{LemmaUsefulEquation}} \Longleftrightarrow \sigma(\{s, t\}) = \{s, t\} \overset{s \in F_{\sigma}} \Longleftrightarrow$ $s, t \in F_{\sigma}$.
    \end{itemize}
    In this case $h = 1$ and $v_{S'} \in \{v_{\emptyset}, v_{st}\}$ belongs to $X_{\Gamma}^{\sigma \iota} \cap K_{\Gamma}$.
    \medskip

    \noindent \underline{Case 3: $g v_S = g v_{st}$ has type 2:} By Lemma \ref{LemmaUsefulEquation}, either $s, t \in F_{\sigma}$ or $s, t \in T_{\sigma}$ are $\sigma$-transposed. The neighbours of $g v_{st}$ have the form $gh v_{\emptyset}$ or $gh v_s$ (without loss of generality). Then if $s, t \in F_{\sigma}$, we have:
    \begin{itemize}
        \item $gh v_{\emptyset} \in X_{\Gamma}^{\sigma \iota} \overset{\ref{LemmaUsefulEquation}} \Longleftrightarrow (\sigma \iota)(h) = h \Longleftrightarrow \iota(h) = h \Longleftrightarrow h = 1$ by Theorem \ref{dihedralFixPointsOdd}.(3) and Theorem \ref{dihedralFixPointsEven}.(3) ;
        \item $gh v_s \in X_{\Gamma}^{\sigma \iota} \overset{\ref{LemmaUsefulEquation}} \Longleftrightarrow (\sigma \iota)(h) \in h \langle s \rangle \Longleftrightarrow \iota(h) \in \langle s \rangle$ $\Longleftrightarrow h \in \langle s \rangle$.
    \end{itemize}
    In the second point, note that all the possible $h \in \langle s \rangle$ give a common type $1$ vertex, as $g h v_s = g v_s$ in that case. Thus we can freely pick $h = 1$. That means that in both cases, we have $h = 1$ and $v_{S'} \in \{v_{\emptyset}, v_s, v_t\}$ belongs to $X_{\Gamma}^{\sigma \iota} \cap K_{\Gamma}$.
    \medskip

    \noindent Now if $s, t\in T_{\sigma}$ are $\sigma$-transposed, we have:
    \begin{itemize}
        \item $gh v_{\emptyset} \in X_{\Gamma}^{\sigma \iota} \overset{\ref{LemmaUsefulEquation}} \Longleftrightarrow (\sigma \iota)(h) = h \overset{(*)} \Longleftrightarrow \begin{cases} 1 \text{ if } m_{st} \text{ is odd}  \\ w_{st}^k \text{ if } m_{st} \text{ is even}  \end{cases}$ ;
        \item $gh v_s \in X_{\Gamma}^{\sigma \iota} \overset{\ref{LemmaUsefulEquation}} \Longleftrightarrow (\sigma \iota)(h) \in h \langle s \rangle$ and $\sigma(s) = s$, giving a contradiction as $s \notin F_{\sigma}$ ;
    \end{itemize}
    where $(*)$ follows from Theorem \ref{dihedralFixPointsOdd}.(3) if $m_{st}$ is odd, and from Theorem \ref{dihedralFixPointsEven}.(7) if $m_{st}$ is even, and where $w_{st}^k$ is any element of the $\mathbb{Z}$-subgroup described in Theorem \ref{dihedralFixPointsEven}.(7). In this we have $h = w_{st}^k$ and $v_{S'} = v_{\emptyset}$ belongs to $X_{\Gamma}^{\sigma \iota} \cap K_{\Gamma}$.
\end{proof}

\begin{corollary}\label{iotaSigmaFixedPoints}
    The subgroup $\Fix(\sigma \iota)$ is the free group whose rank equals the number of $\sigma$-transposed pairs $s, t \in T_{\sigma}$ with $m_{st}$ even, where each pair gives as generator the element that generates the $\mathbb{Z}$-subgroup described in Theorem \ref{dihedralFixPointsEven}.(7).
\end{corollary}

\begin{proof}
    The proof is tantamount to that of Corollary \ref{CoroFixedPointsOfSigma}.
\end{proof}

Now, in line with Remark \ref{RemarkCases}, we are left to consider automorphisms of the form $\varphi_a \sigma\iota$ where $a \in F_{\sigma}$. Recall that by Lemma \ref{LemmaEss1Skeleton}, $X_{\Gamma}^{\varphi_a\sigma \iota}$ is a simplicial tree in $X_{\Gamma}$. As in Corollary \ref{CoroStabSigmaa}, we can understand $\Fix(\varphi_a \sigma \iota)$ as the fundamental group of the graph of groups arising from its action on $X_\Gamma^{\varphi_a \sigma \iota}$.

The first lemma is a useful characterisation of type 1 vertices fixed by $\varphi_a \iota$.

\begin{lemma}\label{stabilisedType1Representative}
    Suppose $s \in V(\Gamma)$. Then the vertex $gv_s$ is fixed by $\varphi_a \iota$ if and only if there is $g' \in A_\Gamma$ such that $g'v_s = gv_s$ and $g'^{-1}a\iota(g') = s$.
\end{lemma}

\begin{proof}
    Suppose $g'^{-1}a\iota(g') = s$. Then, by Lemma \ref{LemmaUsefulEquation}, $g'v_s \in X_\Gamma^{\varphi_a \iota}$. Conversely suppose that $gv_s$ is fixed by $\varphi_a \iota$. Then by Lemma \ref{LemmaUsefulEquation}, $g^{-1}a\iota(g) \in \langle s \rangle$. Notice that $ht(g^{-1}a\iota(g)) = -ht(g) + ht(a) - ht(g) = 1 - 2ht(g)$, which is odd. Thus $g^{-1}a\iota(g) = s^{2k+1}$ for some $k \in \mathbb{Z}$. Then $gs^kv_s = gv_s$, and $(gs^k)^{-1}a\iota(gs^k) = s^{-k}g^{-1}a\iota(g)s^{-k} = s^{-k} s^{2k+1} s^{-k} = s$, so define $g' \coloneqq gs^k$.
\end{proof}

Armed with the lemma, we are ready to investigate the structure of $X_{\Gamma}^{\varphi_a \sigma\iota}$. For the following, recall that if $gv_s$ is type 1, then all its neighbours in $X_{\Gamma}^{(1)-ess}$ are of the form $gv_{st}$, where $t \in V(\Gamma) \backslash \{s\}$ and $m_{st} < \infty$.

\begin{lemma}\label{lemmaSigmaIotaaTree}
    Suppose that $gv_s \in X_{\Gamma}^{\varphi_a \sigma\iota}$ is a type $1$ vertex. Then $\sigma(g) = g$ and $s \in F_{\sigma}$. Furthermore for each $t \in V(\Gamma) \backslash \{s\}$ such that $m_{st} < \infty$, $gv_{st} \in X_{\Gamma}^{\varphi_a \sigma\iota}$ if and only if $t \in F_{\sigma}$. That is, each adjacent type 2 vertex (in $X_\Gamma$) is fixed by $\varphi_a \sigma\iota$ if and only if it is fixed by $\sigma$.

    Suppose $gv_{st} \in X_{\Gamma}^{\varphi_a \sigma\iota}$ is a type 2 vertex. Then $\sigma(g) = g$ and $s, t \in F_{\sigma}$. Furthermore $gv_{st}$ has one type 1 neighbour in $X_{\Gamma}^{\varphi_a \sigma\iota}$ if $m_{st}$ is even, and two type 1 neighbours in $X_{\Gamma}^{\varphi_a \sigma\iota}$ if $m_{st}$ is odd. In the odd case, the two type 1 neighbours are in different $A_\Gamma$-orbits.
\end{lemma}

\begin{proof}
     We will prove inductively on the combinatorial distance to $v_a$ that if $gv_S$ is fixed by $\varphi_a\sigma\iota$, then $\sigma(g) = g$ and $\sigma$ fixes each $s \in S$. For brevity, we will say such vertices are respected by $\sigma$. Note that this holds at $v_a$, which belongs to $X_{\Gamma}^{\varphi_a \sigma \iota}$ by Lemma \ref{LemmaBaseCases}.

     First we consider the statements about the neighbours of type 1 vertices, which we suppose inductively are respected by $\sigma$, so $\sigma(g) = g$ and $s \in F_{\sigma}$. By Lemma \ref{LemmaUsefulEquation}, $a\sigma\iota(g) \in gA_s$. Note that $\Aut(A_{\Gamma})$ acts by type-preserving isometries (see Remark \ref{remTypePreserving}), so the type 2 vertices adjacent to $gv_s$ are fixed as a set. These are exactly the vertices $gv_{st}$, for each $t \in V(\Gamma) \backslash \{s\}$ such that $m_{st} < \infty$. Take $gv_{st}$ an arbitrary such vertex. By Lemma \ref{LemmaUsefulEquation}, given that $\sigma(s) = s$, $gv_{st} \in X^{\varphi_a\sigma\iota}$ if and only $\sigma(t) = t$ and $a\sigma\iota(g) \in gA_{st}$, which occurs if and only if $\sigma(t) = t$, as required, since $a\sigma\iota(g) \in gA_s \leq gA_{st}$. Moreover $\sigma(g) = g$ by assumption, so $gv_{st}$ is respected by $\sigma$.
     \medskip
    
    Now we move to the statement about the neighbourhood of type 2 vertices. Consider $gv_{st}$ which we assume by induction is respected by $\sigma$. An arbitrary type 1 neighbour is of the form $ghv_x$, where $h \in A_{st}$ and $x \in \{s,t\}$. Since $\sigma$ fixes $s$ and $t$, $\sigma$ fixes $h$ (and thus $gh)$ and $x$, thus we obtain that $ghv_x$ is respected by $\sigma$. Furthermore, $\sigma$ acts trivially on the neighbourhood of $gv_{st}$, so for the rest of the proof we will assume $\sigma$ is trivial, as this does not change which vertices in this neighbourhood are fixed.
    \medskip
    
    \noindent \underline{Claim:} Given a type 2 vertex $gv_{st} \in X_{\Gamma}^{\varphi_a \iota}$, at most one of the adjacent type 1 vertices in each $A_\Gamma$-orbit is in $X_{\Gamma}^{\varphi_a \iota}$.
    
    \medskip
    \noindent \underline{Proof of the Claim:} Suppose there are two fixed type 1 vertices in the $A_{\Gamma}$-orbit of $v_s$ that are adjacent to $gv_{st}$. That is, we have $ghv_s, gh'v_s \in X_{\Gamma}^{\varphi_a \iota}$, where $h, h' \in A_{st}$. We will show that $hh'^{-1} \in A_{st}$ is trivial.
    
    By Lemma \ref{stabilisedType1Representative}, $h$ and $h'$ can be chosen so that $h^{-1}g^{-1}a\iota(gh) = s = h'^{-1}g^{-1}a\iota(gh').$ Rearranging, we see that $hh'^{-1} = (g^{-1}a\iota(g)) \iota(hh'^{-1}) (g^{-1}a\iota(g))^{-1}$, which implies
    $$hh'^{-1} \in \Fix(\varphi_{g^{-1}a\iota(g)}\iota).$$
    Given that $\varphi_a \iota \cdot gv_{st} = gv_{st}$, it follows from Lemma \ref{LemmaUsefulEquation} that $g^{-1}a\iota(g) \in A_{st}$, so $\varphi_{g^{-1}a\iota(g)}\iota$ restricts to an automorphism of $A_{st}$. It is easy to compute $\varphi_{g^{-1}a\iota(g)}$ is of finite order, squaring to the identity, so by Theorem \ref{dihedralFixPointsOdd} and Theorem \ref{dihedralFixPointsEven}, $\Fix(\left.(\varphi_{g^{-1}a\iota(g)}\iota)\right|_{A_{st}}) = \{1\}$, so in particular $hh'^{-1} = 1$ as required. This completes the proof of the claim.
    \medskip

    If $gv_{st} \in X_{\Gamma}^{\varphi_a \iota}$, it must have at least one type 1 neighbour in $X_{\Gamma}^{\varphi_a \iota}$. This is because $X_{\Gamma}^{\varphi_a \iota}$ is connected, all of the vertices in $X_\Gamma^{(1)-ess}$ adjacent to $gv_{st}$ are of type 1, and $gv_{st}$ is not all of $X_{\Gamma}^{\varphi_a \iota}$ as $v_a \in X_{\Gamma}^{\varphi_a \iota}$.
    
    Without loss of generality, let this type 1 neighbour be $gv_s$. By the claim, this is the only vertex in the orbit of $v_s$ adjacent to $gv_{st}$ in $X_{\Gamma}^{\varphi_a \iota}$. By Lemma \ref{stabilisedType1Representative}, we also suppose without loss of generality that $g^{-1}a\iota(g) = s$ (this may involve choosing appropriate $g$).

    Next we study type 1 neighbours of $gv_{st}$ in the other orbit. These are all of the form $ghv_t$, where $h \in A_{st}$. We split into two cases, based on whether $m_{st}$ is odd or even.
    \medskip

    \noindent \underline{Case 1: $m_{st} = 2n+1$.} We show existence of $ghv_t \in X_{\Gamma}^{\varphi_a \iota}$.  Let $h = b_1^{-1} \dots b_{n}^{-1} b_{n+1} \dots b_{2n}$, where $b_i = s$  for even $i$, and $b_i = t$ for odd $i$. Now \begin{align*}
        (gh)^{-1}a\iota(gh) &= h^{-1}g^{-1}a\iota(g)\iota(h)\\
        &= h^{-1}s\iota(h)\\
        &= b_{2n}^{-1} \dots b_{n+1}^{-1} b_{n} \dots b_1 s b_1 \dots b_{n} b_{n+1}^{-1} \dots b_{2n}\\
    \intertext{but we notice $b_{n} \dots b_1 s b_1 \dots b_{n}$ is a length $2n+1$ with letters alternating between $s$ and $t$. In particular we can apply the relator, which corresponds to incrementing the relevant indices by 1. That is,}
        &= b_{2n}^{-1} \dots b_{n+1}^{-1} b_{n+1} \dots b_2 t b_2 \dots b_{n+1} b_{n+1}^{-1} \dots b_{2n}\\
    \intertext{and this allows for free reduction, since $b_{n+1+k} = b_{n+1-k}$ for each $k \in \mathbb{Z}$,}
        &= t.
    \end{align*}
    By Lemma \ref{stabilisedType1Representative}, the vertex $ghv_t$ is fixed by $\varphi_a \iota$. This is the unique such vertex by the earlier claim.
    \medskip

    \noindent \underline{Case 2: $m_{st} = 2n$.} Suppose for contradiction that for some $h \in A_{st}$, $ghv_t \in X_{\Gamma}^{\varphi_a \iota}$. Then by Lemma \ref{stabilisedType1Representative} (and appropriate choice of $h$), we see that \begin{align*}
        t &= (gh)^{-1}a\iota(gh)\\
        &= h^{-1}g^{-1}a\iota(g)\iota(h)\\
        &= h^{-1}s\iota(h).
    \end{align*}
    The equation $t = h^{-1}s\iota(h)$ lies in the dihedral Artin subgroup $A_{st} = \langle s, t \ | \ (st)^{n} = (ts)^{n} \rangle$. In $A_{st}$, there is a homomorphism $ht_s\colon A_{st} \rightarrow \mathbb{Z}$ given by $ht_s\colon s \mapsto 1, t \mapsto 0$ (this should be thought of as the $s$-height). Note that $ht_s$ is well-defined because $m_{st}$ is even. Now,
    $$ht_s(h^{-1}s\iota(h)) = ht_s(h^{-1}) + ht_s(s) + ht_s(\iota(h)) = -ht_s(h) + 1 - ht_s(h) = 1-2ht_s(h),$$
    which is odd, but $ht_s(t) = 0$, which contradicts $t = h^{-1}s\iota(h)$.
\end{proof}

Having understood the nature of the tree $X_{\Gamma}^{\varphi_a \sigma\iota}$, we are left to study the action of $\Fix(\varphi_a \sigma\iota)$. In the following, we say an action on a simplicial complex is \emph{cofinite} if their are finitely many orbits of simplices.

\begin{lemma}\label{sigmaiotaafreecofinite}
    The action of $\Fix(\varphi_a \sigma\iota)$ on $X_{\Gamma}^{\varphi_a \sigma\iota}$ is free and cofinite.
\end{lemma}

\begin{proof}
    First we show that the action is free. It is enough to show that the intersection of $\Fix(\varphi_a \sigma \iota)$ with the stabiliser of each point in $X_{\Gamma}^{\varphi_a \sigma\iota}$ is trivial. For type 1 vertices this is immediate since $\varphi_a\sigma\iota$ is height-reversing and no non-trivial elements with non-trivial height stabilise type 1 vertices, so we focus on the type 2 vertices.

    Now, given $gv_{st} \in X_{\Gamma}^{\varphi_a \sigma \iota}$, $\varphi_a \sigma\iota$ restricts to an automorphism of $gA_{st}g^{-1}$. To see this, note that there is some $h \in A_{st}$ such that $a (\sigma \iota)(g) = gh$, by Lemma \ref{LemmaUsefulEquation}. Then, we have
    \begin{align*}
        (\varphi_a\sigma\iota) (gA_{st}g^{-1}) &= (a(\sigma\iota)(g)) (\sigma\iota)(A_{st}) (a(\sigma\iota)(g))^{-1} \\
        &= (gh)\iota(A_{st})(gh)^{-1}\\
        &= g(\varphi_h\iota)(A_{st})g^{-1},
    \end{align*} where $s, t \in F_{\sigma}$ by Lemma \ref{lemmaSigmaIotaaTree}, since $gv_{st} \in X_{\Gamma}^{\varphi_a \sigma \iota}$.

    Notice this restriction of $\varphi_a\sigma\iota$ to $g A_{st} g^{-1}$ is finite order, since $(\varphi_a\sigma\iota)^2 = \varphi_{a\sigma\iota(a)}\sigma^2\iota^2 = \sigma^2$, which is finite order. Therefore the restriction to $g A_{st} g^{-1}$ has no fixed points, because finite order automorphisms $\varphi_h\iota$ (we lose the graph automorphism because $s$ and $t$ are fixed by $\sigma$) of dihedral Artin groups have no fixed points by Theorem \ref{dihedralFixPointsOdd}.(3) and Theorem \ref{dihedralFixPointsEven}.(3). So the action is free.

    Now we show that the action is cofinite. In particular, we show that if two type 1 vertices of $X_{\Gamma}^{\varphi_a \sigma\iota}$ are in the same $A_\Gamma$-orbit, then they are in the same $\Fix(\varphi_a\sigma \iota)$-orbit. Since there are finitely many orbits of type 1 vertices under the action of $A_\Gamma$, and the action on $X_{\Gamma}^{\varphi_a \sigma \iota}$ is dertermined by its action on type 1 vertices, this completes the proof.

    Suppose $gv_s, hv_s \in X_{\Gamma}^{\varphi_a \sigma \iota}$, where $g,h \in A_\Gamma$ and $s \in V(\Gamma)$. It follows by Lemma \ref{lemmaSigmaIotaaTree} that $gv_s, hv_s \in X_{\Gamma}^{\varphi_a \iota}$. Then by Lemma \ref{stabilisedType1Representative}, we may assume that $g^{-1}a\iota(g) = s = h^{-1}a\iota(h)$. Rearranging, we see that $a\iota(gh^{-1})a^{-1} = gh^{-1}$. Since $g$ and $h$ are each fixed by $\sigma$, it is moreover the case that $a\sigma\iota(gh^{-1})a^{-1} = gh^{-1}$. That is, $gh^{-1} \in \Fix(\varphi_a\sigma\iota)$. This demonstrates that $gv_s$ and $hv_s$ are in the same $\Fix(\varphi_a \sigma \iota)$-orbit.
\end{proof}

\begin{corollary}\label{coroFixSigmaIotaa}
    Let $\sigma$ be a (possibly trivial) graph automorphism such that $\sigma(a) = a$. The subgroup $\Fix(\varphi_a\sigma \iota)$ is free of finite rank. In particular,
    $$\Fix(\varphi_a\sigma\iota) \cong \pi_1(\Gamma_{a, odd}^{\sigma}),$$
    where $\Gamma_{a, odd}^{\sigma}$ is the graph from Remark \ref{RemFreeGroupExplicitly} which we recall is the connected component containing $a$ of the subgraph of the barycentric subdivision $\Gamma$ fixed by $\sigma$, cut along all vertices corresponding to even edges.
\end{corollary}

\begin{proof}
    Since $\Fix(\varphi_a \iota)$ acts freely and cofinitely on a tree, it is free of finite rank.  To see the rank, we notice that $X^{\varphi_a\sigma\iota}$ is exactly the universal cover of the desired subgraph $\Gamma_{a, odd}^{\sigma}$. To see this, lift the basepoint $a \in \Gamma_{a, odd}^{\sigma}$ to $v_a$, which is in $X_{\Gamma}^{\varphi_a \iota}$ by Lemma \ref{LemmaBaseCases}. Now notice that, by Lemma \ref{lemmaSigmaIotaaTree}, there is an edge from $gv_s \in X^{\varphi_a\sigma\iota}$ to $gv_{st}$ if and only if there is an edge from $s$ to $t$ in $\Gamma_{a, odd}^{\sigma}$. 

    By the cofiniteness part of Lemma \ref{sigmaiotaafreecofinite}, $gv_s, g'v_t \in X^{\varphi_a\sigma\iota}$ are in the same orbit if and only if $s = t$, so the quotient $X^{\varphi_a\sigma\iota} / \Fix(\varphi_a\sigma\iota)$ can be identified with $\Gamma_{a, odd}^{\sigma}$. The result follows.
\end{proof}

\begin{remark}\label{iotaagenerators}
    To find an explicit generating set for $\Fix(\varphi_a\sigma\iota)$, we proceed similarly to Remark \ref{RemFreeGroupExplicitly}. Let $v^{s_1} < \cdots < v^{s_q}$ be an ordering of the vertices of $\Gamma_a^{\sigma, odd}$ that correspond to standard generators. Label the directed edge from $v^s$ to $v^{st}$ where $v^s < v^t$ and $m_{st} = 2n+1$ by the word
    $$b_1^{-1}\dots b_n^{-1}b_{n+1}\dots b_{2n},$$
    where $b_i = t$ when $i$ is odd, and $b_i = s$ when $i$ is even. A basis for $\Fix(\varphi_a \sigma \iota)$ is given by a set of words that label closed paths starting at $v^a$, and that span a free basis of $\pi_1(\Gamma_{a, odd}^{\sigma})$, where when travelling backwards along a directed edge labelled $w$ we read the inverse word $w^{-1}$.
\end{remark}

\subsection{Automorphisms acting hyperbolically}

In this section we compute the fixed subgroups of the (non-inner) automorphisms of $A_{\Gamma}$ acting hyperbolically on $X_{\Gamma}$. Let $\gamma$ be such an automorphism. In Section 4.3.1 we recall and introduce various tools. In Section 4.3.2 we compute $\Fix(\gamma)$ when $\gamma = \varphi_g \sigma$, i.e. when there is no global inversion involved. In Section 4.3.3 we compute it when $\gamma = \varphi_g \sigma \iota$, i.e. when the global inversion is involved.

\begin{theorem}\label{mainTheoremHyperbolic}
    Let $A_\Gamma$ be a large-type Artin group of rank at least 3. Suppose $\gamma \in \Aut_\Gamma(A_\Gamma)$ is an automorphism acting hyperbolically on $X_\Gamma$. Write $\gamma = \varphi_g \psi$ where $g \in A_\Gamma$ and $\psi = \sigma\iota^\varepsilon$, where $\sigma$ is a (possibly trivial) graph automorphism, $\iota$ is the global inversion, and $\varepsilon \in \{0,1\}$. Let $n$ be the order of $\psi$, and let $\h \coloneqq g\psi(g)\dots\psi^{n-1}(g)$. Then $\gamma^n$ acts like $\h$, and we are in one of the following cases:

    \begin{enumerate}

        \item $\varepsilon = 0$ and $\h$ is as in case 3.1 of Proposition \ref{PropCentralisers} for some standard tree $X_{\Gamma}^g$, where $g$ is conjugated to a standard generator. Then $\Fix(\gamma) \cong \mathbb{Z}^2$, and more precisely $\langle \h, g \rangle \leq_{f.i.} \Fix(\gamma)$.

        \item $\varepsilon = 0$ and $\h$ is as in case 3.3 of Proposition \ref{PropCentralisers} for some $a, b, c \in V(\Gamma)$ such that $m_{ab} = m_{bc} = m_{ac} = 3$, and there exist $q \in \mathbb{Z} \backslash \{0\}$ and $h \in A_\Gamma$ such that:

        \begin{enumerate}
            \item $\sigma(a) = a, \sigma(b) = b, \sigma(c) = c$, and $\gamma \sim_h \varphi_{(abcabc)^q}\sigma$.

            \item $\sigma(a) = c, \sigma(b) = a, \sigma(c) = b$, and $\gamma \sim_h \varphi_{(abcabc)^q ab}\sigma$.

            \item $\sigma(a) = b, \sigma(b) = c, \sigma(c) = a$, and $\gamma \sim_h \varphi_{(abcabc)^q c^{-1} b^{-1}}\sigma$.
        \end{enumerate}

        Then $\Fix(\gamma) \cong \langle x, y \ | \ xyxy = yxyx \rangle$, and more precisely $\Fix(\gamma) = h\langle abc, b \rangle h^{-1}$.

        \item $\varepsilon = 0$ and $\h$ is as in case 3.4 of Proposition \ref{PropCentralisers} for some $a, b, c \in V(\Gamma)$ such that $m_{ab} = m_{bc} = m_{ac} = 3$, and there exists $h \in A_\Gamma$ such that $\h \in h \langle b, abc \rangle h^{-1}$, but $\langle \h \rangle \cap h \langle abcabc \rangle h^{-1} = \{1\}$. Furthermore suppose $\gamma \sim_h \varphi_{g'}\sigma$, where one of the following occurs:
        
        \begin{enumerate}
            \item $\sigma(a) = a, \sigma(b) = b, \sigma(c) = c$, and $g' \in \langle b, abc \rangle$.

            \item $\sigma(a) = c, \sigma(b) = a, \sigma(c) = b$, and $g' \in \langle b, abc \rangle a$.

            \item $\sigma(a) = b, \sigma(b) = c, \sigma(c) = a$, and $g' \in \langle b, abc \rangle c^{-1}$.
        \end{enumerate}

        Then $\Fix(\gamma) \cong \mathbb{Z}^2$, and more precisely $\Fix(\gamma) = C(\h)$.

        \item None of the previous cases apply. Then $\Fix(\gamma) \cong \mathbb{Z}$, with $\langle \h \rangle \leq_{f.i.} \Fix(\gamma)$.
    \end{enumerate}

    In all cases $\langle \h \rangle \cong \mathbb{Z}$ is fixed. Note that in particular, if $\varepsilon = 1$, then $\Fix(\gamma) \cong \mathbb{Z}$.
\end{theorem}

\begin{proof}
    First of all, $\gamma^n$ acts like $z$ because $\gamma^n = (\varphi_g \psi)^n = \varphi_\h \psi^n = \varphi_\h$. By hypothesis, $\gamma$ acts hyperbolically, which is equivalent to $\h$ acting hyperbolically, by the above. In particular, $\h$ falls under Case 3 in Proposition \ref{PropCentralisers}.

    We first suppose that $\varepsilon = 1$. Then by Lemma \ref{PropFixIotagIsVirtuallyZ} we have $\Fix(\gamma) \cong \mathbb{Z}$, and we are in Case 4 of the Theorem.

    We now suppose that $\varepsilon = 0$. If $\h = h^{-1} \h' h$ for some $h, \h \in A_{\Gamma}$, then by Lemma \ref{isogredienceEqConjugate} we have $\varphi_g \psi \sim_h \varphi_{g'}\psi$, where $g'\psi(g')\dots \psi^{n-1}(g') = \h'$. Using Lemma \ref{lemmaIsogredience} then allows to deduce $\Fix(g)$ from $\Fix(g')$, as both are isomorphic via conjugating by $h$. Thus we only consider $\h$ up to conjugation.

    If $\h$ is as in Case 3.1 of Proposition \ref{PropCentralisers}, then we directly obtain from Lemma \ref{LemmaCase3.1Hyperbolic} that we are in Case 1 of the Theorem. If $\h$ is as in Case 3.2 of Proposition \ref{PropCentralisers}, then by Lemma \ref{infiniteOrderIntoCentraliser} we have $\Fix(\gamma) \cong \mathbb{Z}$. In particular, we are in Case 4 of the Theorem. If $\h$ is as in Case 3.3 of Proposition \ref{PropCentralisers}, then using Lemma \ref{LemmaCase3.3Hyperbolic} either we are in Case 2 of the Theorem, or we have $\Fix(\gamma) \cong \mathbb{Z}$ and we are in Case 4. If $\h$ is as in Case 3.4 of Proposition \ref{PropCentralisers}, then by Lemma \ref{LemmaCase3.4Hyperbolic} either we are in Case 3 of the Theorem, or we have $\Fix(\gamma) \cong \mathbb{Z}$ and we are in Case 4.
\end{proof}

\begin{remark}\label{remHyperbolicRank}
    It is immediately apparent that all the fixed subgroups in the preceding theorem have rank at most 2.
\end{remark}

\subsubsection{Preliminaries}

The following lemma leads the strategy throughout Section 4.3. It allows to reduce the study of $\Fix(\gamma)$ to some centraliser. One again, we follows Notation \ref{NotationAutomorphisms} and we write $\gamma = \varphi_g \psi$ with $\psi = \sigma \iota^{\varepsilon}$.

\begin{lemma}\label{infiniteOrderIntoCentraliser}
    Let $A_{\Gamma}$ be a $2$-dimensional Artin group, let $\gamma = \varphi_g \psi$ be an automorphism of $A_\Gamma$, and let $n$ be the (finite) order of $\psi$. Let $\h = g \psi(g) \cdots \psi^{n-1}(g)$. Then
    $$\langle \h \rangle \leq \Fix(\gamma) \leq C(\h),$$
    and moreover, $\gamma$ restricts to an automorphism of $C(\h)$.
\end{lemma}

\begin{proof}
    First of all, note that $\gamma^n = (\varphi_g\psi)^n = \varphi_{\h}\psi^n = \varphi_{\h}$. Together with the fact that any fixed point of $\gamma$ is fixed by $\gamma^n$, this yields
    $$\Fix(\gamma) \leq \Fix(\gamma^n) = \Fix(\varphi_{\h}) = C(\h).$$
    For the other inclusion, we have
    $$(\varphi_g \psi)(\h) = g(\psi(g)\psi^2(g) \cdots \psi^n(g))g^{-1} = g\psi(g)\psi^2(g) \cdots \psi^{n-1}(g) = \h,$$
    where in the second equality we used the fact that $\psi^n(g) = g$.

    For the final part of the statement, notice that, for any $r \in A_\Gamma$ we have
    $$[r,\h] = 1 \Longleftrightarrow \gamma([r,\h]) = 1 \Longleftrightarrow [\gamma(r), \h] = 1,$$
    where the final equivalence follows from $\gamma(\h) = \h$. So $\gamma$ restricts to a bijection on $C(\h)$.
\end{proof}

\begin{remark}
    Lemma \ref{infiniteOrderIntoCentraliser} still holds if $\gamma$ acts elliptically on $X_{\Gamma}$. However in that case the element $\h$ is usually trivial, so the statement of the lemma is empty.
\end{remark}

We assume throughout the section that $\gamma \in \Aut_{\Gamma}(A_{\Gamma})$ acts hyperbolically on $X_{\Gamma}$. In particular, so does $\gamma^n$, where $n$ is the order of $\psi$. We recall is equal to $\gamma^n = \varphi_{\h}$ by Lemma \ref{infiniteOrderIntoCentraliser}, where $\h = g \psi(g) \cdots \psi^{n-1}(g)$. By compatibility, this means that the element $\h$ also acts hyperbolically. In particular, $\h$ fits into one of the following cases of Proposition \ref{PropCentralisers}.

\begin{proposition4.2}
    \begin{enumerate}
    \setcounter{enumi}{2}

    \item If $g$ acts hyperbolically on $X_{\Gamma}$, we consider the transverse-tree $\mathcal{T}$ associated with $g$ (see Lemma \ref{LemmaTransverseTree}). Then up to conjugation, we have:
    \begin{enumerate}
        \item If $\mathcal{T}$ is bounded, and some axis of $g$ is contained in a standard tree of the form $X_{\Gamma}^a$, then $C(g) = \langle g_0, a \rangle \cong \mathbb{Z}^2$, where $g_0$ acts with minimal non-trivial translation length along the axes of $g$.
        \item If $\mathcal{T}$ is bounded, and no axis of $g$ is contained in a standard tree of $X_{\Gamma}$, then $C(g) = \langle g_0 \rangle \cong \mathbb{Z}$, where $g_0$ acts with minimal non-trivial translation length along the axes of $g$.
        \item If $\mathcal{T}$ is unbounded and $g$ has an axis that is contained in a standard tree of the form $X_{\Gamma}^b$, then $g = (abcabc)^k$ for some $k \neq 0$ and $C(g) = \langle b, abc \rangle$, where $a, b, c$ are three standard generators with coefficients $m_{ab} = m_{ac} = m_{bc} = 3$. In that case, the centraliser $C(g)$ is abstractly isomorphic to the dihedral Artin group $\langle x, y \ | \ xyxy = yxyx \rangle$.
        \item If $\mathcal{T}$ is unbounded and $g$ has no axis that is contained in a standard tree of $X_{\Gamma}$, then $C(g) \cong \mathbb{Z}^2$ with $\langle g, abcabc \rangle \leq_{f.i.} C(g)$, where $a, b, c$ are three standard generators with coefficients $m_{ab} = m_{ac} = m_{bc} = 3$.
    \end{enumerate}
\end{enumerate}
\end{proposition4.2}

\noindent As usual, we only have to consider automorphisms up to isogredience, by Lemma \ref{lemmaIsogredience}. By the following lemma, this means we only have to consider $\h$ up to conjugacy.

\begin{lemma}\label{isogredienceEqConjugate}
    Let $g_1, g_2 \in A_\Gamma$, let $\psi$ be an order $n$ automorphism of $A_\Gamma$, and define $\h_i \coloneqq g_i\psi(g_i)\dots\psi^{n-1}(g_i)$ for $i \in \{1, 2 \}$. If $\varphi_{g_1}\psi = \varphi_r \varphi_{g_2}\psi \varphi_{r^{-1}}$, then $\h_1 = r\h_2r^{-1}$.
\end{lemma}
\begin{proof}
    Notice that, $$\varphi_{g_1}\psi = \varphi_r\varphi_{g_2}\psi\varphi_r = \varphi_{rg_2\psi(r)^{-1}}\psi,$$ so in particular $g_1 = rg_2r^{-1}$. Therefore, \begin{align*}
        \h_1 &= g_1\psi(g_1)\dots\psi^{n-1}(g_1)\\
        &= rg_2\psi(r)^{-1}\psi(r)\psi(g_2)\psi^2(r)^{-1}\dots\psi^{n-1}(r)\psi^{n-1}(g_2)\psi^n(r)^{-1}\\
        &=rg_2\psi(g_2)\dots\psi^{n-1}(g_2)r^{-1}\\
        &= r\h_2r^{-1}.
    \end{align*}
\end{proof}

\subsubsection{Hyperbolic automorphisms involving no global inversion}

In this section we compute $\Fix(\gamma)$ when $\gamma$ acts hyperbolically on $X_{\Gamma}$. We also assume that $\gamma$ does not involve $\iota$, i.e. we can write $\gamma = \varphi_g \sigma$. By Lemma \ref{infiniteOrderIntoCentraliser} we have
$$\langle \h \rangle \leq \Fix(\gamma) \leq C(\h),$$
where $\h \coloneqq g \sigma(g) \cdots \sigma^{n-1}(g)$ acts hyperbolically on $X_{\Gamma}$. We will treat the different possibilities for $\h$ by following Proposition \ref{PropCentralisers}.(3). The interesting cases are those where $C(\h)$ is strictly bigger than $\mathbb{Z}$, that is where $\h$ is as in cases 3.1, 3.2 and 3.4 of Proposition \ref{PropCentralisers}.

\begin{lemma}\label{LemmaCase3.1Hyperbolic}
    Suppose $\h$ is as in case 3.1 of Proposition \ref{PropCentralisers}, so $\langle \h, a \rangle \cong \mathbb{Z}^2$ is a finite index subgroup of $C(\h)$, where $a$ is a standard generator. Then $\Fix(\gamma) = C(\h) \cong \mathbb{Z}^2$.
\end{lemma}

\begin{proof}
    The subgroup $\langle \h, a \rangle$ has finite index in $C(\h)$, so in order to prove that $\gamma$ fixes $C(\h)$ it is enough to prove that it fixes $\langle \h, a \rangle$, as only the identity fixes a finite index subgroup of $\mathbb{Z}^2$. We already know that $\h \in \Fix(\gamma)$, so our goal will be to show that $a \in \Fix(\gamma)$ as well. Recall that $\gamma = \varphi_g \sigma$, so $\gamma(a) = g \sigma(a) g^{-1}$ is a conjugate of the standard generator $\sigma(a)$, and $X_{\Gamma}^{\gamma(a)}$ is a standard tree.
    \medskip

    \noindent \underline{Claim:} $\h$ and $\gamma(a)$ commute, and $X_{\Gamma}^{\gamma(a)} \cap X_{\Gamma}^a \neq \emptyset$.
    \medskip

    \noindent \underline{Proof of the Claim:} Because $\h$ is fixed by $\gamma$, we have
    $$[\h, \gamma(a)] = [\gamma(\h), \gamma(a)] = \gamma([\h, a]) = \gamma(1) = 1.$$
    Note that this means both $a$ and $\gamma(a)$ belong to $C(\h)$. Since $C(\h)$ is abelian, it cannot be that $\langle a, \gamma(a) \rangle$ is a non-abelian free group. Using \cite[Proposition C]{martin2022tits}, it follows that $X_{\Gamma}^{\gamma(a)}$ and $X_{\Gamma}^a$ intersect. This finishes the proof of the Claim.
    \medskip

    As any two standard trees which intersect, $X_{\Gamma}^{\gamma(a)} \cap X_{\Gamma}^a$ contains a type $2$ vertex. We call this vertex $v$, and we call $G_v$ its local group, which is a dihedral parabolic subgroup of $A_{\Gamma}$. Because $v \in X_{\Gamma}^{\gamma(a)}$, we have $\gamma(a) \in G_v$. As $a, \gamma(a) \in C(\h)$ and $C(\h)$ is abelian then $a$ and $\gamma(a)$ must commute, which yields $\gamma(a) \in C(a) \cap G_v = C_{G_v}(a)$.

    By Proposition \ref{PropCentralisers} again (applied to $G_v$), we know that $C_{G_v}(a) = \langle a, z \rangle$, where $z$ is the element generating the centre of $G_v$. In particular, the above implies that $\gamma(a) = a^n z^m$ for some $n, m \in \mathbb{Z}$. Because $\gamma(a)$ has type $1$, the only possibility is that $m = 0$ (see \cite{vaskou2022acylindrical}, and more precisely Case 2 in the proof of Proposition E.(2)). Comparing heights, we then obtain $n = 1$, and finally $\gamma(a) = a$, as wanted.
\end{proof}

Both remaining cases only occur when $\Gamma$ contains a triangle with all edges labelled by $3$. In this case there are exotic dihedral Artin subgroups (see Definition \ref{DefiExotic}). Up to conjugation and permuting $a$, $b$ and $c$, the centre of such a subgroup is generated by $abcabc$ (assuming maximality, see Theorem \ref{TheoremExotic}). The following technical lemma will play an important role in both remaining cases.

\begin{lemma}\label{lemabcabcfixed}
    Suppose $a$, $b$, $c$ are distinct generators of $A_\Gamma$ with $m_{ab} = m_{bc} = m_{ac} = 3$. Then $\gamma(abcabc) = abcabc$ if and only if one of the following three conditions are met:

    \begin{enumerate}
        \item $\sigma(a) = a$, $\sigma(b) = b$, $\sigma(c) = c$, and $g \in \langle b, abc \rangle = C(abcabc)$.

        \item $\sigma(a) = b$, $\sigma(b) = c$, $\sigma(c) = a$, and $ga^{-1} \in \langle b, abc \rangle$.

        \item $\sigma(a) = c$, $\sigma(b) = a$, $\sigma(c) = b$, and $gc \in \langle b, abc \rangle$.
    \end{enumerate}
\end{lemma}

\begin{proof}
    We recall that $\gamma = \varphi_g \sigma$. First assume one of the three conditions is met. Then it is straightforward to check that $\gamma(abcabc) = abcabc$. For example if $\sigma(a) = b$, $\sigma(b) = c$, $\sigma(c) = a$, and $ga^{-1} \in C(abcabc)$ then
    $$(\varphi_g \sigma)(abcabc) = g\cdot bcabca \cdot g^{-1} = ga^{-1} \cdot abcabc \cdot (ga^{-1})^{-1} = abcabc.$$
    The other two cases are similar.
    
    Let us denote by $A_{str}$ the parabolic subgroup of $A_{\Gamma}$ generated by $\{s, t, r\}$. If the element $abcabc$ is fixed by $\gamma$, then the two following parabolic closures must agree:
    \begin{itemize}
        \item that of $abcabc$, that is $A_{abc}$ ;
        \item that of $(\varphi_g \sigma)(abcabc)$, that is $g A_{\sigma(a)\sigma(b)\sigma(c)} g^{-1}$.
    \end{itemize}
    By \cite[Theorem 4.1]{paris1997parabolic}, we must have $A_{abc} = A_{\sigma(a)\sigma(b)\sigma(c)}$. We obtain $\sigma(\{a, b, c\}) = \{a, b, c\}$, as well as $g \in A_{abc}$. Now we divide into 4 cases, based on the permutation $\sigma$ induces on $\{a, b, c\}$. The identity and order 3 permutations will lead to the three cases in the statement, while the transpositions will lead to a contradiction. 
    \medskip
    
    \noindent \underline{Case 1: $\sigma$ restricts to the identity on $\{a, b, c\}$.} We get
    $$abcabc = (\varphi_g \sigma)(abcabc) = g \cdot abcabc \cdot g^{-1},$$
    which forces $g \in C(abcabc) = \langle b, abc \rangle$ (for the equality, see \cite[Lemma 4.2]{vaskou2023isomorphism}).
    \medskip

    \noindent \underline{Case 2: $\sigma$ induces an order $3$ permutation of $\{a, b, c\}$.} We assume $\sigma(a) = b$, $\sigma(b) = c$ and $\sigma(c) = a$, the other case is similar. We get
    $$abcabc = (\varphi_g \sigma)(abcabc) = g \cdot bcabca \cdot g^{-1} = (ga^{-1}) \cdot abcabc \cdot (ga^{-1})^{-1},$$
    which is satisfied if and only if $ga^{-1} \in C(abcabc) = \langle b, abc \rangle$.
    \medskip

    \noindent \underline{Case 3: $\sigma$ fixes one of $a$ or $c$ and swaps the two other standard generators.} Suppose without loss of generality that $\sigma(a) = b$, $\sigma(b) = a$ and $\sigma(c) = c$. We may do this, because if $b$ and $c$ were swapped, then we may pass to the isomorphic copy of $A_\Gamma$ under $\iota$, so $\langle abcabc, b \rangle$ becomes $\langle (cbacba)^{-1}, b^{-1} \rangle$. By taking the inverses of the generators this is $\langle cbacba, b \rangle$. We now proceed  with the rest of the argument in this isomorphic copy of $A_\Gamma$, where swapping $b$ and $c$ has become equivalent to swapping $a$ and $b$ in the old presentation. 

    We have
    $$abcabc = (\varphi_g \sigma)(abcabc) = g \cdot bacbac \cdot g^{-1}. \ \ (*)$$
    We will show that such an element $g$ cannot exist, which will provide a contradiction. Our argument partially relies on \cite[Section 3\&4]{vaskou2023isomorphism}. Since $\sigma$ preserves $\{a, b, c\}$ and $g \in A_{abc}$, we restrict ourselves to the Deligne subcomplex $X_{abc} \subseteq X_{\Gamma}$ corresponding to the parabolic subgroup $A_{abc}$. In this setting, $X_{abc}$ is tiled by principal triangles, and we put a system of colours on $X_{abc}$, relatively to the three orbits of edges of these principal triangles (see Definition \ref{DefinitionPrincipalTriangle}, Remark \ref{RemSystemColours}). A picture is drawn in Figure \ref{FigureCase3.4} (note that the picture can actually be made planar by following \cite[Lemma 4.12, Figure 17]{vaskou2023isomorphism}, but our argument does not rely on this fact).

    Let $u$ be the axis defined as the intersection $u \coloneqq X_\Gamma^{bacbac} \cap X_{\Gamma}^a$. The equation $(*)$ implies that $g u$ is an axis of $abcabc$. In particular, $g v_{ab}$ belongs to $Min(abcabc)$. Let $Y$ be the fundamental domain of the action $C(abcabc) \curvearrowright Min(abcabc)$. Then there is some $h \in C(abcabc)$ such that $hg v_{ab} \in Y$. Notice that $Y$ is the brown area from Figure \ref{FigureExoticMinset2}, and one can see $Y$ contains exactly four type $2$ vertices, and only one in the orbit or $v_{ab}$ (use the system of colours). In particular, it must be that $hg v_{ab} = v_{ab}$.
 
    We now exhibit the contradiction. The element $hg$ sends the line $u$ onto $hg u$ while fixing $v_{ab}$ and preserving the colouring of the edges of the two axes. Using Figure \ref{FigureCase3.4}, we can see that the only possibility is that
    $$hg e_3 = e_1 \ \ \text{ and } \ \ hg e_4 = e_2. \ \ (**)$$
    It is clear that $e_3 = e^a$ while $e_2 = e^b$. One can see using \cite[Figure 17]{vaskou2023isomorphism} that the edges $e_1$ and $e_4$ can be described as $e_1 = \Delta_{ab} e^a$ and $e_4 = \Delta_{ab} e^b$, where $\Delta_{ab} = aba$ is the usual Garside element. In particular, $(**)$ gives
    $$hg e^a = \Delta_{ab} e^a \ \ \text{ and } \ \ hg \Delta_{ab} e^b = e^b,$$
    which implies that there are some $k, q \in \mathbb{Z}$ such that
    $$hg = \Delta_{ab} a^k \ \ \text{ and } \ \ hg = \Delta_{ab}^{-1} b^q.$$
    This yields $\Delta_{ab}^2 = b^q a^{-k}$, a contradiction.
    \medskip

    \noindent \underline{Case 4: $\sigma$ fixes $b$ and swaps $a$ and $c$.} We have
    $$abcabc = (\varphi_g \sigma)(abcabc) = g \cdot cbacba \cdot g^{-1}.$$
    We proceed similarly as in Case 3 (although the picture is not planar). We obtain
    $$hg e_5 = e_1 \text{ with } e_5 = \Delta_{ab}^{-1} e^a \ \ \text{ and } \ \ hg e_6 = e_2 \text{ with } e_6 = e^b.$$
    Thus there are some $k, q \in \mathbb{Z}$ such that
    $$hg \Delta_{ab}^{-1} e^a = \Delta_{ab} e^a \ \text{ and } \ hg e^b = e^b, \ \text{ i.e. } \ hg = \Delta_{ab}^2 a^k \ \text{ and } \ hg = b^q.$$
    We finally obtain $\Delta_{ab}^2 = b^q a^{-k}$, a contradiction.
\end{proof}

\begin{figure}[H]
\centering
\includegraphics[scale=0.72]{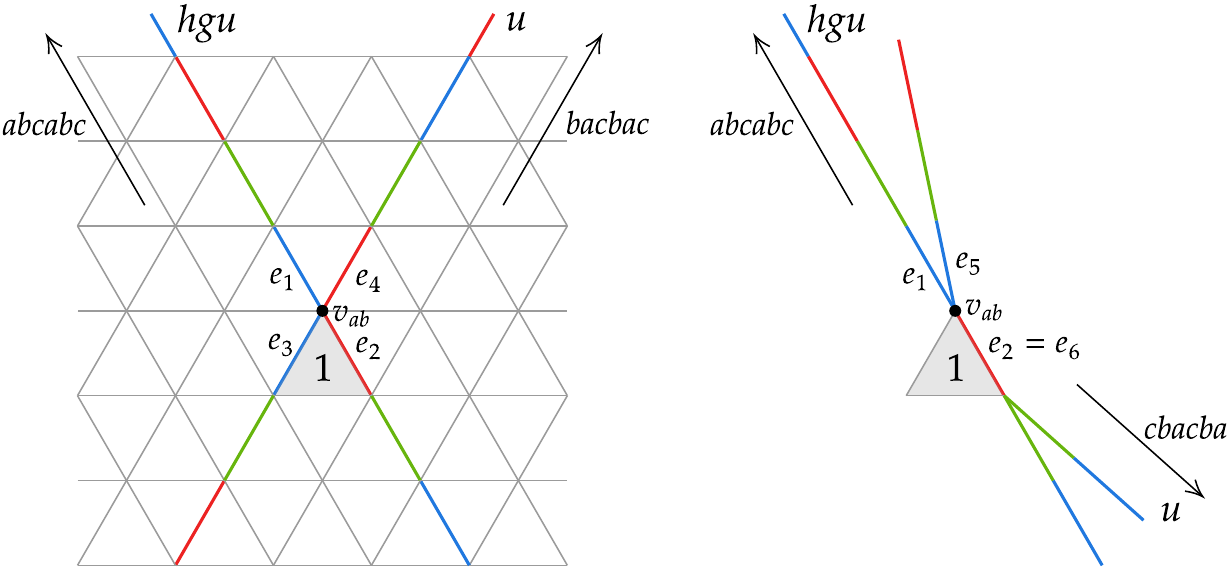}
\caption{The arguments of the proof of Lemma \ref{lemabcabcfixed}, seen as in $X_{abc}$. The principal triangle that is contained in $K_{\Gamma}$ is denoted by $1$. The edges in the orbit of $e^a$, $e^b$ and $e^c$ are drawn in blue, red and green respectively.}
\label{FigureCase3.4}
\end{figure}

\begin{lemma} \label{LemmaCase3.4Hyperbolic}
    Suppose $\h$ is as in case 3.4 of Proposition \ref{PropCentralisers}, so $\langle \h, abcabc \rangle$ is a finite index subgroup of $C(\h) \cong \mathbb{Z}^2$, where $a$, $b$ and $c$ are three standard generators satisfying $m_{ab} = m_{ac} = m_{bc} = 3$. Then
    
    \begin{enumerate}
        \item If $\sigma$ is the identity on $\{a, b, c\}$ and $g \in \langle b, abc \rangle$, then $\Fix(\gamma) = C(\h) \cong \mathbb{Z}^2$.
        \item If $\sigma(a) = b$, $\sigma(b) = c$, $\sigma(c) = a$ and $g \in \langle b, abc \rangle a$, then $\Fix(\gamma) = C(\h) \cong \mathbb{Z}^2$.
        \item If $\sigma(a) = c$, $\sigma(b) = a$, $\sigma(c) = b$ and $g \in \langle b, abc \rangle c^{-1}$, then $\Fix(\gamma) = C(\h) \cong \mathbb{Z}^2$.
        \item Otherwise, $\Fix(\gamma) \cong \mathbb{Z}$ is virtually $\langle \h \rangle$.
    \end{enumerate}
\end{lemma}

\begin{proof}
    We first notice that $\Fix(\gamma) \cong \mathbb{Z}^2$ if and only if $abcabc \in \Fix(\gamma)$. Indeed, the “if” direction is clear as $\h$ and $abcabc$ don't share any non-trivial powers. For the “only if” direction, note that $\Fix(\gamma)$ can be seen as a cocompact lattice in $C(\h)$. In particular, it has finite index in $C(\h) \cong \mathbb{Z}^2$, so the whole of $C(\h)$ is fixed by $\gamma$. In particular, $abcabc \in \Fix(\gamma)$. Using Lemma \ref{lemabcabcfixed} along with the above implies that $\Fix(\gamma) \cong \mathbb{Z}^2$ if and only if one of the conditions in the first three cases is met, and otherwise we are in case 4.
\end{proof}

\begin{lemma}\label{LemmaCase3.3Hyperbolic}
    Suppose $\h$ is as in case 3.3 of Proposition \ref{PropCentralisers}, i.e. $\h = (abcabc)^k$ for some $k \in \mathbb{Z}$, and $C(\h) = \langle b, abc \rangle$ is an exotic dihedral Artin subgroup of $A_{\Gamma}$. Then one of the following 4 cases occurs. In what follows, $q \in \mathbb{Z} \backslash \{0\}$.

    \begin{enumerate}
        \item $\sigma(a) = a$, $\sigma(b) = b$ and $\sigma(c) = c$, and $g = (abcabc)^q$. Moreover, $\Fix(\gamma) = C(\h)$.

        \item $\sigma(a) = c$, $\sigma(c) = b$ and $\sigma(b) = a$, and $g = (abcabc)^q ab$. Moreover, $\Fix(\gamma) = C(\h)$.

        \item $\sigma(a) = b$, $\sigma(b) = c$, $\sigma(c) = a$, and $g = (abcabc)^q c^{-1} b^{-1}$.  Moreover, $\Fix(\gamma) = C(\h)$.

        \item $\Fix(\gamma) \cong \mathbb{Z}$.
    \end{enumerate}
\end{lemma}
    
\begin{proof}
    Before proceeding to the main argument, we argue that $abcabc \in \Fix(\gamma)$. By Lemma \ref{infiniteOrderIntoCentraliser}, $\gamma$ restricts to an automorphism of $C((abcabc)^k) = \langle abc, b \rangle$. Therefore, it restricts to an automorphism of the centre, which is the cyclic group $\langle abcabc \rangle$ (this follows from the centraliser structure of even dihedral Artin groups, for instace Proposition \ref{dihedralCentralisers}). Since, again by Lemma \ref{infiniteOrderIntoCentraliser}, $(abcabc)^k \in \Fix(\gamma)$, it must be that the induced automorphism on $\langle abcabc \rangle$ is trivial, and $abcabc \in \Fix(\gamma)$ as required.

    We recall that $X_{\Gamma}^{\h}$ splits as a direct product of the form $\mathcal{T} \times \mathbb{R}$. The structure of $\mathcal{T}$ and the action of $C(\h)$ on $X_{\Gamma}^{\h}$ and on $\mathcal{T}$ are described in Proposition \ref{PropBassSerreExotic} and in Figure \ref{FigureExoticMinset1}. We adopt the same notation as in this figure, and we let $Y \subseteq X_{\Gamma}^{\h}$ be the fundamental domain of the action of $C(\h)$ on $X_{\Gamma}^{\h}$, as drawn on Figure \ref{FigureExoticMinset2}.

    By Corollary \ref{corPreserveFixMinsets}, $\gamma$ acts on $X_\Gamma^{\h}$. Moreover, $X_\Gamma^{\h}$ contains an axis $w$ of translation of $\gamma$. This comes from the fact that $\gamma$ acts hyperbolically on $X_\Gamma$, and all its axes are axes of $\gamma^n = \varphi_\h$, which acts like $\h$ by compatibility. Up to conjugating $\gamma$ by $\varphi_r$ for some appropriate $r \in C(\h)$, we may assume that $w$ goes through $Y$, and in particular it is a vertical axis contained in the strip $[u, v]$ (see Figure \ref{FigureExoticMinset2}).

    First suppose that $w$ is neither $u$ nor $v$. Then, because the action is by combinatorial isometries, it must be that $u$ is also an axis of $\gamma$ (actually, we have $[u, v] \subseteq X_{\Gamma}^{\gamma}$).

    Suppose now that $w = v$, and that $u$ is not an axis of $\gamma$. By the above, it must be that the only axis of $\gamma$ is $v$. Now, by Corollary \ref{fixAction}, we have $\Fix(\gamma) \leq \Stab(v)$. But $\Stab(v) \cong \mathbb{Z}$ by \cite[Lemma 3.23]{vaskou2023isomorphism} (actually, $\Stab(v) = \langle abc \rangle$). Thus we have $\langle abcabc \rangle \leq \Fix(\gamma) \leq \langle abc \rangle$, so we are in Case 4 of the lemma.

    From now we assume that $v$ is not the only axis of $\gamma$. By the above, this means $u$ is an axis of $\gamma$. Notice that $\Stab(u) = \langle b, abcabc \rangle \cong \mathbb{Z}^2$ (see Figure \ref{FigureExoticMinset2}). We consider the translation length of elements of $\Stab(u)$ along $u$, with respect to the combinatorial metric $d^C$. Note that the translation length of $abcabc$ is $6$. Recall that if $n$ is the order of $\sigma$ in $\Aut_{\Gamma}(A_{\Gamma})$ then $\gamma^n = \varphi_{\h}$. Since $\h = (abcabc)^k$, the automorphism $\gamma^n$ acts with translation length $6k$.

    We now split the proof into 3 cases, based on the automorphism induced by $\sigma$ on $\{a, b, c\}$. Since $abcabc \in \Fix(\gamma)$, we can use Lemma \ref{lemabcabcfixed}: $\sigma$ acts on ${a,b,c}$ either trivially or with order 3. In what follows, we heavily refer to Figure \ref{FigureExoticMinset2}.
    \medskip

    \noindent\underline{Case 1: $\sigma$ is the identity on $\{a,b,c\}$.} Then $\sigma$ acts trivially on the Deligne subcomplex $X_{abc}$. In particular, $\gamma = \varphi_g \sigma$ acts like $g$ on $X_{abc}$, by compatibility. Let $e^b$ be the principal edge defined by $e^b = K \cap u$. Since $\gamma$ acts by translations on $u$, it must be that $\gamma \cdot e^b = g e^b$ is also an edge of $u$. By our system of colour, this edge is in the red-coloured orbit. But this orbit of principal edges is precisely $\langle abcabc \rangle e^b$. Consequently, there is some $q \in \mathbb{Z}$ such that $g e^b = (abcabc)^q e^b$. Note that $q \neq 0$ because $g$ acts hyperbolically on $u$. This means that $(abcabc)^{-q} g \in \langle b \rangle$, and more precisely, that there is some $i \in \mathbb{Z}$ such that $g = (abcabc)^q b^i$. It follows that the translation length of $\gamma$ along $u$ is $6q$. But the translation length of $\gamma^n$ is $6k$, so we must have $k = nq$. We can now compare heights:
    $$6nq = 6k = ht((abcabc)^k) = ht(\h) = ht(g \sigma(g) \cdots \sigma^{n-1}(g)) = n \cdot ht(g) = n (6q +i),$$
    where we used that $\sigma$ preserves height. This forces $i = 0$, and thus $g = (abcabc)^q$, as wanted.

    \medskip
    \noindent\underline{Case 2: $\sigma(a) = c$, $\sigma(b) = a$, $\sigma(c) = b$.} 
    This time, because $\sigma(b) = a$, the edge $\gamma \cdot e^b = g e^a$ is a blue-coloured edge of $u$. But the blue-coloured orbit of edges in $u$ is exactly $\langle abcabc \rangle ab e^a$: this is because the only blue edge in $Y$ is $ab e^a$, and the other edges are obtained by $(abcabc)$-translations. We obtain the equation $g e^a = (abcabc)^q ab e^a$ for some $q \in \mathbb{Z}$. This yields $g = (abcabc)^q ab a^i$ for some $i \in \mathbb{Z}$. Once again, $q \neq 0$ or $g$ would be elliptic. This time, the translation length of $g$ is $6q + 2$. In particular, the translation length of $\gamma^n = \varphi_g$ is both $n(6q+2)$ and $6k$. As before, comparing heights gives:
    $$n(6q+2) = 6k = ht(\h) = n \cdot ht(g) = n(6q +2 +i).$$
    Once again, this gives $i = 0$ and $g = (abcabc)^q ab$.

    \medskip
    \noindent\underline{Case 3: $\sigma(a) = b$, $\sigma(b) = c$, $\sigma(c) = a$.} This case is almost identical to the previous one. We obtain $g = (abcabc)^q c^{-1}b^{-1} c^i$ for some $q \neq 0$. The translation length of $g$ is $6q -2$ and its height is $6q-2+i$, so as before we conclude that $i = 0$ and $g = (abcabc)^q c^{-1}b^{-1}$.

    \medskip
    Finally, we want to prove that $\Fix(\gamma) = C(abcabc)$ in each of the three previous cases. By Lemma \ref{infiniteOrderIntoCentraliser}, we only need to show $C(abcabc) \leq \Fix(\gamma)$. Recall that $C(abcabc) = \langle b, abc \rangle$, thus we only show that $abc$ and $b$ are fixed by $\gamma$. We only do the computations for Case 2, the other two cases being very similar. We have $\sigma(a) = c$, $\sigma(b) = a$ and $\sigma(c) = b$, and we now know that $g = (abcabc)^q ab$. We calculate,
    \begin{align*}
        \gamma(abc) = (\varphi_{(abcabc)^q ab}\sigma)(abc) &= (abcabc)^q ab \cdot cab \cdot (ab)^{-1} (abcabc)^{-q}\\
        &= (abcabc)^q abc (abcabc)^{-q}\\
        &= abc,
    \end{align*}
    and
    \begin{align*}
        \gamma(b) = (\varphi_{(abcabc)^q ab}\sigma)(b) &= (abcabc)^q ab \cdot a \cdot (ab)^{-1} (abcabc)^{-q}\\
        &= (abcabc)^q b (abcabc)^{-q}\\
        &= b,
    \end{align*}
    where in the penultimate equality we use that $aba = bab$.
\end{proof}

\begin{figure}[H]
    \centering
    \includegraphics[scale=0.9]{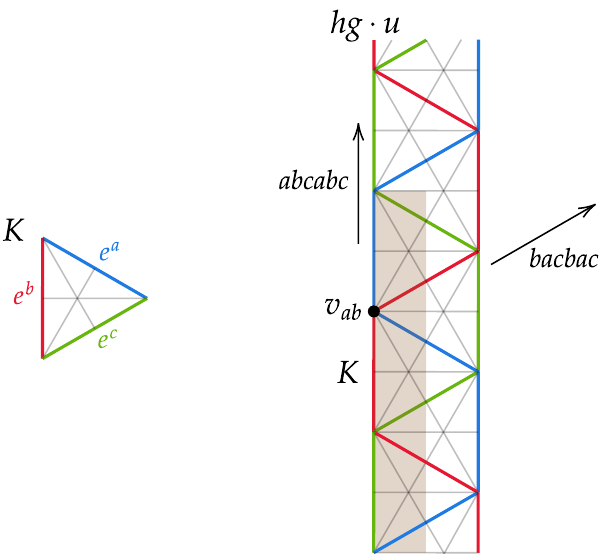}
    \caption{The arguments of the proof of Lemma \ref{LemmaCase3.3Hyperbolic}. See Figure \ref{FigureExoticMinset1} for more details.}
    \label{FigureExoticMinset2}
\end{figure}

\subsubsection{Hyperbolic automorphisms involving the global inversion}

Finally, in this section we compute $\Fix(\gamma)$ when $\gamma$ acts hyperbolically on $X_{\Gamma}$ and does involve the global inversion. In other words, $\gamma = \varphi_g \sigma \iota$. The classification given by Proposition \ref{PropCentralisers} comes very handy because of how the elements described interact with height. We begin with the following observation:

\begin{lemma} \label{LemmaHeightMustBe0}
    If $h \in \Fix(\gamma)$ then $ht(h) = 0$.
\end{lemma}

\begin{proof}
    This is because $\varphi_g$ and $\sigma$ preserve height while $\iota$ reverses height. Thus we have
    $$ht(h) = ht((\varphi_g \sigma \iota)(h)) = -ht(h).$$
\end{proof}

\begin{proposition} \label{PropFixIotagIsVirtuallyZ}
    Let $\gamma = \varphi_g \sigma \iota$ act hyperbolically on $X_\Gamma$, let $n$ be the order of $\psi \coloneqq \sigma \iota$ and let $\h \coloneqq g \psi(g) \cdots \psi^{n-1}(g)$. Then $\Fix(\gamma) \cong \mathbb{Z}$, and it is virtually $\langle \h \rangle$.
\end{proposition}

\begin{proof}
We know by Lemma \ref{infiniteOrderIntoCentraliser} that $\langle \h \rangle \leq \Fix(\gamma) \leq C(\h)$. The element $\h$ must take one of the four forms of hyperbolic elements described in Proposition \ref{PropCentralisers}. We proceed case by case:
\medskip

\noindent \underline{Case 3.3.} In this case, we must have $\h = (abcabc)^k$ for some $k \neq 0$. In particular, $ht(\h) = 6k \neq 0$. This contradicts Lemma \ref{LemmaHeightMustBe0}, as $\h \in \Fix(\gamma)$.
\medskip

\noindent \underline{Case 3.2.} Then we have $\mathbb{Z} \cong \langle \h \rangle \leq \Fix(\gamma) \leq C(\h) \cong \mathbb{Z}$. So $\Fix(\gamma)$ is virtually $\langle \h \rangle$.
\medskip

\noindent For the two remaining cases, we have $\mathbb{Z} \cong \langle \h \rangle \leq \Fix(\gamma) \leq C(\h) \cong \mathbb{Z}^2$. We will consider the map $ht \colon C(\h) \rightarrow \mathbb{Z}$ and note that, by Lemma \ref{LemmaHeightMustBe0}, we must have $\Fix(\gamma) \leq \ker(ht)$. In both cases, we will show that $Im(ht)$ is non-trivial. Consequently, it will be isomorphic to $\mathbb{Z}$, and thus $\ker(ht) \cong \mathbb{Z}^2 / \mathbb{Z} \cong \mathbb{Z}$, proving that $\Fix(\gamma) \cong \mathbb{Z}$.
\medskip

\noindent \underline{Case 3.1.} Up to conjugation, $C(\h)$ contains a non-trivial power $a^n$ of a standard generator. We have $ht(a^n) = n \neq 0$, hence $Im(ht)$ is non-trivial.
\medskip

\noindent \underline{Case 3.4.} Up to conjugation, $C(\h)$ contains an element of the form $(abcabc)^k$ for some $k \neq 0$. We have $ht((abcabc)^k) = 6k \neq 0$, so $Im(ht)$ is non-trivial.
\end{proof}

\section{Geometrically Characterising $\bf{\mathrm{\bf{Aut}}_\Gamma(A_\Gamma)}$}\label{sectionEquivariantRigidity}

In this section, we provide a natural geometric characterisation of $\Aut_\Gamma(A_\Gamma)$ as the maximal subgroup of $\Aut(A_\Gamma)$ with a compatible action on $X_\Gamma$. We start with the following notion.

\begin{definition}
    Take $X$ a metric space with $G \curvearrowright X$. Then for $\psi \in \Aut(G)$, an isometry $f$ of $X$ is \emph{$\psi$-equivariant} if for all $g \in G$, $x \in X$, $$f(gx) = \psi(g)f(x).$$
\end{definition}

\begin{lemma}\label{lemmaComposingEquivariants}
    Suppose $G \curvearrowright X$ is an action of $G$ on a metric space $X$. Suppose for $i \in \{1,2\}$, $f_i$ is a $\psi_i$-equivariant isometry, where $\psi_i \in \Aut(G)$. Then $f_1f_2$ is a $\psi_1\psi_2$-equivariant isometry.
\end{lemma}
\begin{proof}
    We see that for all $g \in G$, $x \in X$, $$f_1f_2(gx) = f_1(\psi_2(g)f_2(x)) = (\psi_1\psi_2)(g)(f_1f_2)(x).$$
\end{proof}

The next lemma relates the notion of a $\psi$-equivariant isometry in terms of compatible actions.

\begin{lemma}\label{compatibilityImpliesEquivariance}
    Suppose $\Omega\colon G \rightarrow \Isom(X)$ is an action of $G$ on a metric space $X$, and $A \leq \Aut(G)$ acts compatibly on $X$. Then if $\psi \in A$, the isometry $\Omega(\psi)$ is $\psi$-equivariant.
\end{lemma}
\begin{proof}
    Take $g \in G$ and $x \in X$. Then $$\psi \cdot gx = \psi \varphi_g \cdot x = \varphi_{\psi(g)} \psi \cdot x = \psi(g)(\psi \cdot x).$$
\end{proof}

\begin{theorem}\label{theoremEquivariantRigidity}
    Let $A_\Gamma$ be a large-type connected Artin group. Suppose $f$ is an isometry of $X_\Gamma$, which is $\psi$-equivariant for some $\psi \in \Aut(A_\Gamma)$. Then $\psi \in \Aut_\Gamma(A_\Gamma)$.
    
\end{theorem}

\begin{proof}
First, we note that any isometry of $X_\Gamma$ is combinatorial and type preserving. This follows from the fact that the type 2 vertices are exactly the points with unbounded link \cite[Theorem E, Remark 3.4]{vaskou2022acylindrical}, and from the permutation of the type 2 vertices we can recover the entire isometry.

Suppose that $f$ is a $\psi$-equivariant isometry of $X_\Gamma$, for some $\psi \in \Aut(A_\Gamma)$. Consider the type $0$ vertex $v_{\emptyset}$ that lies in $K_{\Gamma}$. Since $f$ sends type $0$ vertices onto type $0$ vertices, we know that $f(v_{\emptyset}) = g v_{\emptyset}$ for some $g \in A_{\Gamma}$. By Lemma \ref{compatibilityImpliesEquivariance} and Lemma \ref{lemmaComposingEquivariants} then, $\varphi_{g^{-1}}f$ is a $\varphi_{g^-1}\psi$-equivariant isometry fixing $v_\emptyset$. So up to post-composing $\psi$ by an inner automorphism, we can assume $f$ fixes $v_\emptyset$, and do hereafter.

In particular, $f$ now preserves the star and the link of $v_{\emptyset}$, i.e. $K_{\Gamma}$ and its boundary $\partial K_{\Gamma}$. This means $f$ induces an automorphism $\sigma$ of $\partial K_{\Gamma}$ and thus of (the barycentric subdivision of) $\Gamma$. Up to post-composing by the isometry of $X_\Gamma$ induced by the graph automorphism $\sigma^{-1}$, $f$ fixes $\partial K_{\Gamma}$ and $K_{\Gamma}$ pointwise. As before, from here we assume $f$ fixes $\partial K_\Gamma$, up to post-composing $\psi$ by $\sigma^{-1}$ ($\sigma^{-1}\psi$-equivariance follows from Lemma \ref{compatibilityImpliesEquivariance} and Lemma \ref{lemmaComposingEquivariants}).

We now come to a more local argument. Consider any type $2$ vertex $v_{ab}$ of $K_{\Gamma}$. By the previous argument, $\psi$ fixes $v_{ab}$ and hence preserves the link $lk_{X_{\Gamma}}(v_{ab})$. Since $f$ fixes $K_{\Gamma}$ pointwise, it actually fixes the intersection $lk_{X_{\Gamma}}(v_{ab}) \cap K_{\Gamma} = e_a \cup e_b$ pointwise. We can now apply \cite[Proposition 40]{crisp2005automorphisms}: $f$ acts on $lk_{X_{\Gamma}}(v_{ab})$ either trivially, or as the global inversion $\iota$ (sending an edge $g e_s$ to $\iota \cdot g e_s = \iota(g) e_s$, for $s \in \{a, b\}$). This means that up to post-composing with the action of $\iota^{\varepsilon}$ for an appropriate $\varepsilon \in \{0, 1\}$, we can assume that $f$ fixes the whole of $lk_{X_{\Gamma}}(v_{ab})$ pointwise. As usual, we now only control $\psi$ up to $\iota^\varepsilon$.
\medskip

Now we stop modifying $f$, and show that it is the identity on $X_\Gamma$.
\medskip

\noindent \underline{Claim:} For every pair of adjacent vertices $s, t \in V(\Gamma)$, the automorphism $f$ fixes $lk_{X_{\Gamma}}(v_{st})$ pointwise.
\medskip

\noindent \underline{Proof of Claim:} Since $\Gamma$ is by assumption connected, we proceed via induction on the minimal number $n$ such that $v_0, v_1 \dots v_n$ is a sequence of adjacent vertices of $\Gamma$ where $v_n \in \{s, t\}$ and $\{v_0, v_1\} = \{a, b\}$. The base case, when $\{s,t\} = \{a,b\}$ is already compete.

So suppose, inductively and without loss of generality, that there is a vertex $u \in V(\Gamma)$ such that $lk_{X_\Gamma}(v_{us})$ is fixed pointwise by $f$. Therefore $f$ fixes both the edges $e_s$ and $s e_s$. Note that both $e_s$ and $s e_s$ are contained in $lk_{X_{\Gamma}}(v_{st})$. Using \cite[Proposition 40]{crisp2005automorphisms} again, we obtain that $f$ acts trivially on $lk_{X_{\Gamma}}(v_{st})$. Inductively, this finishes the proof of the claim.
\medskip

Crucially, it follows from the claim that for each standard generator $s$ (which is contained in an edge of $\Gamma$ by connectedness) the point $sv_\emptyset$ is fixed by $f$. Using $\psi$-equivariance, this yields
$$s v_{\emptyset} = f(sv_\emptyset) = \psi(s)f(v_\emptyset) = \psi(s)v_\emptyset.$$
So $\psi(s) = s$ since $v_\emptyset$ has trivial stabiliser. Since $\psi$ is the identity on every generator, $\psi$ is the identity. Throughout the proof, we post-composed $\psi$ by an automorphism $\varphi_g\sigma\iota^{\varepsilon} \in \Aut_\Gamma(A_\Gamma)$. The result follows.
\end{proof}

We obtain Theorem \ref{theoremIntroEquivariantRigidity} as an immediate corollary.

\begin{corollary}\label{corollaryMaximalCompatible}
    Suppose $A_\Gamma$ is a large-type connected Artin group, and $A \leq \Aut(A_\Gamma)$ acts compatibly on $X_\Gamma$. Then $A \leq \Aut_\Gamma(A_\Gamma)$.
\end{corollary}
\begin{proof}
    Take such an $A$ and $\psi \in A$ arbitrary. Then by Lemma \ref{compatibilityImpliesEquivariance} there is a $\psi$-equivariant isometry of $X_\Gamma$, and by Theorem \ref{theoremEquivariantRigidity}, $\psi \in \Aut_\Gamma(A_\Gamma)$.
\end{proof}

\bibliographystyle{amsalpha}
\bibliography{bibliography}

\end{document}